\documentclass[reqno]{amsart}
\usepackage{amssymb}
\usepackage{amsmath}
\usepackage[mathscr]{euscript}
\usepackage[small]{caption}
\usepackage{graphicx}
\graphicspath{ {images/} }
\usepackage{color}

\makeatletter
\@addtoreset{equation}{section}
\makeatother

\newtheorem{theorem}{Theorem}[section]
\newtheorem{lemma}[theorem]{Lemma}
\newtheorem{proposition}[theorem]{Proposition}

\newtheorem{definition}[theorem]{Definition}
\newtheorem{remark}[theorem]{Remark}

\newcounter{as}[section]

\newtheorem{asser}[as]{Assertion}

\newcommand{\mc}[1]{{\mathcal #1}}
\newcommand{\mf}[1]{{\mathfrak #1}}

\newcommand{\bb}[1]{{\mathbb #1}}
\newcommand{\bs}[1]{{\boldsymbol #1}}

\newcommand{\<}{\langle}
\renewcommand{\>}{\rangle}
\renewcommand{\Cap}{{\rm cap}}

\begin{document}

\title[Metastability in non-reversible diffusion processes] {
  Dirichlet's and Thomson's principles for non-selfadjoint elliptic
  operators with application to non-reversible metastable diffusion
  processes.}

\begin{abstract}
  We present two variational formulae for the capacity in the context
  of non-selfadjoint elliptic operators. The minimizers of these
  variational problems are expressed as solutions of boundary-value
  elliptic equations.  We use these principles to provide a sharp
  estimate for the transition times between two different wells for
  non-reversible diffusion processes.  This estimate permits to
  describe the metastable behavior of the system.
\end{abstract}

\author{C. Landim, M. Mariani, I. Seo}

\address{\noindent IMPA, Estrada Dona Castorina 110, CEP 22460 Rio de
  Janeiro, Brasil and CNRS UMR 6085, Universit\'e de Rouen, France.
  \newline e-mail: \rm \texttt{landim@impa.br} }

\address{\noindent Faculty of Mathematics, National Research University Higher School of Economics, 6 Usacheva St., 119048 Moscow, Russia. \newline e-mail: \rm
  \texttt{mmariani@hse.ru}}

\address{\noindent Department of Mathematical Sciences, Seoul National University Gwanak-Ro 1, Gwanak-Gu 08826, Seoul, Republic of Korea. \newline
  e-mail: \rm \texttt{insuk.seo@snu.ac.kr} }


\maketitle

\section{Introduction}
\label{sec-1}

This article is divided in two parts. In the first one, we present two
variational formulae which extend the classical Dirichlet's and
Thomson's principles to non-selfadjoint elliptic operators. In the
second one, we use these formulae to describe the metastable behavior
of a non-reversible diffusion process in a double-well potential
field.

Fix a smooth, bounded, domain (open and connected) $\Omega \subset\bb
R^d$, $d\ge 2$, and a smooth function $f\colon \bb R^d\to\bb R$. Denote by
$\Omega_f$ the set of functions $v\colon \overline{\Omega} \to \bb R$ such
that $v = f$ on $\partial \Omega$, the boundary of $\Omega$.  The
classical Dirichlet's principle \cite{K12, AH96} states that the energy
\begin{equation*}
\int_{\Omega} \Vert \nabla u(\bs{x}) \Vert^2 \, d\bs{x}
\end{equation*}
is minimized on $\Omega_f$ by the harmonic function on $\Omega$ which
takes the value $f$ at the boundary, that is, by the solution of
\begin{equation}
\label{1-1}
\text{$\Delta u \;=\; 0 $ on $\Omega$  and $u=f$ on $\partial\Omega$.}
\end{equation}

When $\Omega = \mc D\setminus \overline{\mc B}$, where $\mc B\subset
\mc D\subset \bb R^d$ are smooth domains, and $f=1$, $0$ on $\mc B$,
$\mc D^c$, respectively, the minimal energy is called the capacity. In
electrostatics, it corresponds to the total electric charge on the
conductor $\partial \mc B$ held at unit potential and grounded at $\mc
D^c$. It is denoted by $\Cap_\mc D(\mc B)$ and can be represented, by
the divergence theorem, as
\begin{equation}
\label{1-2}
\Cap_\mc D(\mc B) \;=\;
-\, \int_{\partial \mc B} \frac{\partial h}{\partial \bs n_\mc B} \, d \sigma \;,
\end{equation}
where $h$ is the harmonic function which solves \eqref{1-1}, $\bs n_\mc B$
is the outward normal vector to $\partial \mc B$, and $\sigma$ the surface
measure at $\partial \mc B$.

These results have long been established for self-adjoint operators of
the form $(Lu)(\bs{x}) = e^{V(\bs{x})} \nabla \cdot [ e^{-V(\bs{x})} \bb S(\bs{x}) \nabla
u (\bs{x})]$, provided $\bb S(\bs x)$, $\bs x\in\bb R^d$, are smooth, positive-definite, symmetric
matrices, and $V$ is a smooth potential. They have been extended, more
recently, by Pinsky \cite{p1, p95} to the case in which the operator
$L$ is not self-adjoint. In this situation, the minimization formula
for the capacity, mentioned above, has to be replaced by a minmax
problem.

The first main result of this article provides two variational
formulae for the capacity \eqref{1-2} in terms of divergence-free
flows. In contrast with the minmax formulae, the first optimization
problem is expressed as an infimum, while the second one is expressed
as a supremum, simplifying the task of obtaining lower and upper
bounds for the capacity.

Analogous Dirichlet's and Thomson's principles were obtained by
Gaudilli\`ere and Landim \cite{GL} (the Dirichlet's principle) and by
Slowik \cite{Slo} (the Thomson principle) for continuous-time Markov
chains. \smallskip

In the second part of the article, we use the formulae for the
capacity to examine the metastable behavior of a non-reversible
diffusion in a double well potential.

Let $U\colon \bb R^d\to \bb R$ be a smooth, double-well potential which
diverges at infinity, and let $\bb M$ be a non-symmetric,
positive-definite matrix. We impose in Section~\ref{sec3} further assumptions on $U$. Denote by $\bb{M}^{\dagger}$ the transpose of $\bb
M$ and by $\bb S = (\bb{M}+\bb{M}^{\dagger})/2$ its symmetric part.
Consider the diffusion $X^\epsilon_t$, $\epsilon>0$, described by the
SDE
\begin{equation}
\label{1-3}
dX_{t}^{\epsilon}\;=\;-\bb{M} \, (\nabla U) (X_{t}^{\epsilon})\, dt
\,+\,\sqrt{2\epsilon}\, \bb{K}\, dW_{t}\;,
\end{equation}
where $W_{t}$ is a standard $d$-dimensional Brownian motion, and
$\bb{K}$ is the symmetric, positive-definite square root of $\bb{S}$, i.e.,
$\bb{S}=\bb K \bb K$.

Assume that $U$ has two local minima, denoted by $\bs m_1$, $\bs m_2$,
separated by a single saddle point $\bs \sigma$, and that $U(\bs m_2)
\le U(\bs m_1)$. The stationary state of $X^\epsilon_t$, given by
$\mu_\epsilon(d\bs{x}) \sim \exp\{- U(\bs{x})/\epsilon\}\, d\bs{x}$, is concentrated
in a neighborhood of $\bs m_2$ when the previous inequality is strict.

If $X_{t}^{\epsilon}$ starts from a neighborhood of $\bs m_1$, it
remains there for a long time in the small noise limit $\epsilon\to 0$
until it overcomes the potential barrier and jumps to a neighborhood
of $\bs m_2$ through the saddle point $\bs \sigma$.  Denote by
$\tau_\epsilon$ the hitting time of a neighborhood of $\bs m_2$.  The
asymptotic behavior of the mean value of $\tau_\epsilon$ as
$\epsilon\to 0$ has been the object of many studies.

The Arrhenius' law \cite{a89} asserts that the mean value is
logarithmic equivalent to the potential barrier: $\lim_{\epsilon\to 0}
\epsilon \log \bb E_{\bs m_1} [\tau_\epsilon] = U(\bs \sigma) - U(\bs
m_1) =:\Delta U$, where $\bb E_{\bs m_1}$ represents the expectation
of the diffusion $X^\epsilon_t$ starting from $\bs m_1$. The
sub-exponential corrections, known as the Eyring-Kramers formula
\cite{e35, k40}, have been computed when the matrix $\bb M$ is
symmetric and the potential non-degenerate at the critical
points. Assuming that the Hessian of the potential is positive
definite at $\bs m_1$ and that it has a unique negative eigenvalue at
$\bs \sigma$, denoted by $- \lambda$, while all the others are
strictly positive, the sub-exponential prefactor is given by
\begin{equation*}
\bb{E}_{\bs{m}_{1}}\left[\tau_{\epsilon}\right]\;=\;
\left[1+o_{\epsilon}(1)\right]\, \frac{2\pi}{\lambda}
\,\frac{\sqrt{-\det\left[(\text{Hess }U) \, (\bs{\sigma})\right]}}
{\sqrt{\det\left[(\text{\rm Hess }U)\, (\bs{m}_{1})\right]}}\,\,
e^{\Delta U/\epsilon}\;,
\end{equation*}
where $o_{\epsilon}(1)\to 0$ as $\epsilon$ vanishes.

This estimate appears in articles published in the 60's.  A rigorous
proof was first obtained by Bovier, Eckhoff, Gayrard, and Klein
\cite{BEGK1} with arguments from potential theory, and right after by
Helffer, Klein and Nier \cite{hkn} through Witten Laplacian
analysis. We refer to Berglund \cite{b13} and Bouchet and Reygner
\cite{BR} for an historical overview and further references.

Recentlty, Bouchet and Reygner \cite{BR} extended the Eyring-Kramers
formula to the non-reversible setting. They showed that in this context
the negative eigenvalue $- \lambda$ has to be replaced by the unique
negative eigenvalue of $(\text{Hess }U) \, (\bs{\sigma}) \, \bb M$.

We present below a rigorous proof of this result, based on the
variational formulae obtained for the capacity in the first part of
the article, and on the approach developed by Bovier, Eckhoff,
Gayrard, and Klein \cite{BEGK1} in the reversible case. This estimate
permits to describe the metastable behavior of the diffusion
$X^\epsilon_t$ in the small noise limit.  Analogous results have been
derived for random walks in a potential field in \cite{LS1, LS2}.

\section{Notation and Results}
\label{sec0}

We start by introducing the main assumptions. We frequently refer to
\cite{gt} for results on elliptic equations and to \cite{f, p95} for
results on diffusions.

\subsection{A Dirichlet's and a Thomson's principle}
Fix $d\ge 2$, and denote by $C^k(\bb R^d)$, $0\le k\le \infty$, the
space of real functions on $\bb R^d$ whose partial derivatives up to
order $k$ are continuous. Let $\bb M_{m,n}$, $1\le m,n\le d$, be
functions in $C^2(\bb R^d)$ for which there exists a finite constant
$C_0$ such that
\begin{equation}
\label{3-1}
\sum_{m,n=1}^d \bb M_{m,n}(\bs{x})^2 \;\le\; C_0 \quad\text{for all
  $x\in\bb R^d$.}
\end{equation}
Denote by $\bb M(\bs{x})$ the matrix whose entries are $\bb
M_{m,n}(\bs{x})$. Assume that the matrices $\bb M(\bs{x})$, $\bs x\in \bb R^d$, are
uniformly positive-definite: There exist $0<\lambda<\Lambda$ such that
for all $\bs x$, $\bs \xi\in \bb R^d$,
\begin{equation}
\label{19}
\lambda \, \Vert\bs \xi \Vert^2 \;\le\;
\bs \xi \cdot \bb M (\bs{x})\bs  \xi \le \Lambda \, \Vert\bs \xi \Vert^2\;,
\end{equation}
where $\eta\cdot \xi$ represents the scalar product in $\bb R^d$, and
$\Vert x\Vert$ the Euclidean norm.


Let $V$ be a function in $C^3(\bb R^d)$ such that $\int_{\bb R^d}
\exp\{-V(\bs{x})\}\, d\bs{x} <\infty$, and assume, without loss of generality,
that $\int_{\bb R^d} \exp\{-V(\bs{x})\}\, d\bs{x} = 1$. Denote by $\mu$ the
probability measure on $\bb R^d$ defined by $\mu(d\bs{x}) = \exp\{-V(\bs{x})\}
d\bs{x}$.

Denote by $\mathcal L$ the differential operator which acts on functions in
$C^2(\bb R^d)$ as
\begin{equation}\label{gen1}
(\mathcal L f) (\bs{x}) \;=\; e^{V(\bs{x})} \, \nabla \cdot
\big\{ e^{-V(\bs{x})} \bb M(\bs{x}) (\nabla f)(\bs{x}) \big\}\;.
\end{equation}
In this formula, $\nabla g$ represents the gradient of a function
$g\colon\bb R^d \to \bb R$ and $\nabla \cdot \Phi$ the divergence of a
vector field $\Phi \colon\bb R^d \to \bb R^d$.  The previous formula can
be rewritten as
\begin{equation}
\label{15}
(\mathcal L f) (\bs{x}) \;=\; \sum_{j=1}^d b_j(\bs{x}) \, \partial_{x_j} f (\bs{x})
\;+\; \sum_{j,k=1}^d \bb S_{j,k} (\bs{x}) \, \partial^2_{x_j, x_k} f (\bs{x}) \;,
\end{equation}
where the drift $\bs b=(b_1, \dots, b_d)$ is given by
\begin{equation}
\label{drift1}
b_j(\bs{x}) \;=\; \sum_{k=1}^d \Big\{  \partial_{x_k} \bb M_{k,j}(\bs{x})
\;-\; (\partial_{x_k} V) (\bs{x}) \, \bb M_{k,j}(\bs{x}) \Big\}\;,
\end{equation}
and where $\bb S(\bs{x}) = (1/2) [\bb M(\bs{x}) + \bb M^\dagger(\bs{x})]$ represents the symmetric part of the matrix $\bb
M (\bs{x})$,
$M^\dagger(\bs{x})$ being the transpose of $\bb M(\bs{x})$.

Denote by $\mc B(r)\subset \bb R^d$, $r>0$, the open ball of radius $r$ centered at the
origin, and by $\partial \mc B(r)$ its boundary. We assume that
\begin{equation}
\label{18}
\lim_{n\to\infty} \inf_{\bs z\not\in \mc B(n)} V(\bs z) \;=\; \infty\;,
\end{equation}
and that there exist $r_1>0$, $c_1>0$ such that
\begin{equation}
\label{20}
(\mathcal L V)(\bs{x}) \;\le\; -\, c_1
\end{equation}
for all $\bs x$ such that $\Vert \bs x\Vert \ge r_1$. By \eqref{15},
this last condition can be rewritten as
\begin{align*}
& \sum_{j,k=1}^d (\partial_{x_j} \bb M_{j,k})(\bs{x}) \, (\partial_{x_k} V) (\bs{x})
\;+\; \sum_{j,k=1}^d \bb S_{j,k} (\bs{x}) \, \partial^2_{x_j, x_k} V (\bs{x})
\;+\; c_1 \\
&\qquad \;\le\;  (\nabla V)(\bs{x}) \cdot \bb S(\bs{x}) (\nabla V)(\bs{x})
\end{align*}
for all $\bs x$ such that $\Vert \bs x\Vert \ge r_1$.

It follows from the first condition in \eqref{18} that $V$ is bounded
below by a finite constant: there exists $c_2\in \bb R$ such that
$V(\bs{y}) \ge c_2$ for all $\bs y\in\bb R^d$. Of course, $\bb M(\bs{x}) = \bb I$,
where $\bb I$ represents the identity matrix, and $V(\bs{x}) = \Vert
x\Vert^2 + c$ satisfy all previous hypotheses for an appropriate
constant $c$.

The regularity of $\bb M$ and $V$, and assumptions \eqref{3-1},
\eqref{19} are sufficient to guarantee the existence of smooth
solutions of some Dirichlet problems. By \cite[Theorem~6.1.3]{p95}, these conditions together with \eqref{18}, \eqref{20} yield that the process whose generator is given by $\mc L$ is positive recurrent.

\smallskip\noindent{\bf Elliptic equations.}  Fix a domain (open and connected set) $\Omega \subseteq\bb R^d$. Denote by $C^k(\Omega)$, $k\ge 0$, the space of functions on $\Omega$ whose partial derivatives
up to order $k$ are continuous, and by $C^{k,\alpha}(\Omega)$, $0<\alpha <1$, the space of function in $C^k(\Omega)$ whose $k$-th order partial derivatives are H\"older continuous with exponent $\alpha$.

Denote by $\overline{\Omega}$ the closure of $\Omega$ and by $\partial
\Omega$ its boundary.  The domain $\Omega$ is said to have a
$C^{k,\alpha}$-boundary, if for each point $\bs x\in\partial\Omega$,
there is a ball $\mc B\subset \bb R^d$ centered at $\bs x$ and a
one-to-one map $\psi$ from $\mc B$ onto $\mc C\subset \bb R^d$ such
that
\begin{align*}
&\psi(\mc B \cap \Omega) \subset \{\bs
z\in\bb R^d : z_d >0\}\;,\;\;\;\psi(\mc B \cap \partial\Omega) \subset \{\bs
z\in\bb R^d : z_d =0\}\;,\\
&\qquad \psi \in C^{k,\alpha} (\mc B)\;,\;\mbox{and}\;\;\; \psi^{-1} \in
C^{k,\alpha} (\mc C)\;.
\end{align*}
Denote by $L^2(\Omega)$ the space of functions $f\colon\Omega \to \bb R$
endowed with the scalar product $\<\,\cdot\,, \,\cdot\,\>_\mu$ given
by
\begin{equation*}
\< f\,,\, g\>_\mu \;=\; \int_{\Omega} f\, g\, d\mu\;,
\end{equation*}
and by $W^{1,2}(\Omega)$ the Hilbert space of weakly differentiable
functions endowed with the scalar product $\<f,\,g\>_1$ given by
\begin{equation*}
\< f ,\, g\>_1 \;=\; \int_{\Omega}  f \, g \, d \mu \;+\;
\int_{\Omega} \nabla f \cdot \nabla g \, d \mu \;.
\end{equation*}

Fix $0<\alpha<1$, a function $\mf g$ in $L^2(\Omega) \cap
C^{\alpha}(\overline{\Omega})$ and a function $\mf b$ in
$W^{1,2}(\Omega)\cap C^{2,\alpha}(\overline{\Omega})$.  Assume that
$\Omega$ has a $C^{2,\alpha}$-boundary.  It follows from assumptions
\eqref{3-1}, \eqref{19} and Theorems 8.3, 8.8 and 9.19 in \cite{gt}
that the Dirichlet boundary-value problem
\begin{equation}
\label{22}
\begin{cases}
(\mathcal L u) (\bs{x}) \;=\; -\, \mf g(\bs{x}) &  x\in \Omega \;, \\
u (\bs{x}) \;=\; \mf b(\bs{x}) &   x\in \partial \Omega \;.
\end{cases}
\end{equation}
has a unique solution in $W^{1,2}(\Omega)\cap
C^{2,\alpha}(\overline{\Omega})$. Moreover, by the maximum principle,
\cite[Theorem~8.1]{gt}, if $\mf g =0$,
\begin{equation}
\label{25}
\inf_{x\in \partial \Omega} \mf b(\bs{x}) \;\le\;
\inf_{y \in \Omega} u(\bs{y}) \;\le\;
\sup_{y \in \Omega} u(\bs{y}) \;\le\;
\sup_{x \in \partial \Omega} \mf b (\bs{x}) \;.
\end{equation}
The proofs of Theorems 8.3 and 8.8 in \cite{gt} require simple
modifications since it is easier to work with $\mu(d\bs{x})$ as reference
measure than the Lebesgue measure.

\smallskip\noindent{\bf Dirichlet's and Thomson's principles.}
We will frequently assume that a pair of sets $\mc A$, $\mc B$ with
$C^1$-boundaries fulfill the following conditions. Denote by $d(\mc A,\,\mc B)$
the distance between the sets $\mc A$, $\mc B$, $d(\mc A,\,\mc B) = \inf\{\Vert \bs x-\bs y\Vert
: \bs x\in \mc A \,,\, \bs y\in \mc B\}$, and by $\sigma (\partial \mc A)$ the measure of
the boundary of $\mc A$.

\smallskip
\noindent{\bf Assumption S.} The sets $\mc A$, $\mc B$ are bounded domains
of $\bb R^d$ with $C^{2,\alpha}$-boundaries, for some $0<\alpha<1$,
and finite perimeter, $\sigma(\mc A)<\infty$,
$\sigma(\mc B)<\infty$. Moreover, $d(\mc A,\,\mc B)>0$, and the set $\Omega =
(\overline{\mc A}\cup \overline{\mc B})^c$ is a domain. \smallskip

Denote by $h_{\mc A,\mc B}$ the unique solution of the Dirichlet
problem \eqref{22} with $\Omega = (\overline{\mc A}\cup \overline{\mc
  B})^c$, $\mf g=0$, and $\mf b$ such that $\mf b (\bs{x}) = 1$, $0$
if $x\in \partial \mc A$, $\partial \mc B$, respectively. The function
$h_{\mc A,\mc B}$ is called the {\sl equilibrium potential} between
$\mc A$ and $\mc B$.  Similarly, denote by $h^*_{\mc A,\mc B}$ the
solution to \eqref{22} with $\bb M$ replaced by its transpose $\bb
M^\dagger$ and the same functions $\mf g$ and $\mf b$.

The capacity between $A$ and $B$, denoted by $\Cap (\mc A,\,\mc B)$,
is defined as
\begin{equation}
\label{06}
\Cap (\mc A,\,\mc B) \;=\; \int_{\partial \mc A} \big[ \bb M(\bs{x}) \, \nabla
h_{\mc A,\mc B}(\bs{x}) \big] \cdot \bs n_\Omega(\bs{x})  \, e^{-V(\bs{x})} \sigma (d\bs{x})\;,
\end{equation}
where $\sigma(d\bs{x})$ represents the surface measure on the boundary
$\partial \Omega$ and $\bs n_\Omega$ represents the outward normal
vector to $\partial \Omega$ (and, therefore, the inward normal vector
to $\partial \mc A \cup \partial \mc B$). Since $\partial A$ is the
$1$-level set of the equilibrium potential $h_{\mc A,\mc B}$ which, by
the maximum principle, is bounded by $1$. Therefore $\nabla h_{\mc
    A,\mc B} = c (\bs{x}) \, \bs n_{\Omega}(\bs{x})$ for some nonnegative scalar function $c\colon \partial \mc A \to \bb R$, so that
\begin{equation*}
\bb M (\bs{x}) \, (\nabla h_{\mc A,\mc B}) (\bs{x})
\cdot \bs n_{\Omega}(\bs{x})
\;=\; \bb S(\bs{x}) \, (\nabla h_{\mc A,\mc B})
(\bs{x}) \cdot \bs n_{\Omega}(\bs{x}) \ge 0\;.
\end{equation*}
The capacity can also be expressed as
\begin{equation*}
\Cap (\mc A,\,\mc B) \;=\;
-\, \int_{\partial \mc A} \frac{\partial h_{\mc A,\mc B}}{\partial v}
\, e^{-V} \, d \sigma \;,
\end{equation*}
where $v(\bs{x}) = \bb M^\dagger (\bs{x}) \, \bs n_\mc A(\bs{x})$.

In Section~\ref{sec1} we develop a more systematic approach to the
study of capacities, and we introduce here just the basic notation
used in the Propositions~\ref{prop1}-\ref{prop2} below. For vector
fields $\varphi \colon \Omega \to \bb R^d$ define the quadratic form
\begin{equation*}
\< \varphi \,,\, \varphi \> \;=\; \int_{\Omega} \varphi(\bs{x}) \cdot \bb
S(\bs{x})^{-1} \varphi(\bs{x}) \, e^{V(\bs{x})}\, d\bs{x}\;.
\end{equation*}
Let $\mathcal F = \mathcal F_{\mc A,\mc B}$ be the Hilbert space of
vector fields $\varphi$ such that $\< \varphi \,,\, \varphi
\><+\infty$ up to a.e.\ identification, with the scalar product
induced by polarization. By assumption \eqref{19}, $\varphi \in
\mathcal F$ iff $\int_\Omega \Vert \varphi (\bs{x})\Vert^2 e^{V(\bs
  x)}d\bs x <\infty$.

For $c\in \bb R$, let $\mathcal F^{(c)}$ be the space of vector fields
$\varphi \in \mathcal F$ of class $C^1(\overline{\Omega})$ such that
\begin{gather}
(\nabla \cdot \varphi) (\bs{x}) \;=\; 0 \quad\text{for $\bs x\in \Omega$}\;, \label{gg1} \\
-\, \int_{\partial A} \varphi (\bs{x}) \cdot \bs n_\Omega (\bs{x}) \,
\sigma(d\bs{x}) \;=\; c \;=\; \int_{\partial B} \varphi (\bs{x})
\cdot \bs n_\Omega(\bs{x}) \, \sigma(d\bs{x}) \;.\label{gg2}
\end{gather}
The reason for the minus sign is due to the convention that $\bs
n_\Omega(\bs{x})$ is the inward normal to $\partial \mc A$. The integrals over
$\partial \mc A$, $\partial \mc B$ are well defined because $\varphi$ is
continuous, and $\mc A$, $\mc B$ have finite perimeter. On the other hand, the
integral over $\partial \mc B$ must be equal to minus the integral over
$\partial \mc A$ because $\varphi$ is divergence free on $\Omega$.

For a function $f \colon \overline{\Omega} \to \bb R$ in
$C^2(\overline{\Omega})\cap W^{1,2}(\Omega)$, denote by
$\Phi_f$, $\Phi^*_f$, $\Psi_f$ the elements of $\mathcal F$ given by
\begin{equation}
\label{2-3}
\Psi_f \;=\; e^{-V} \bb S \nabla f\;, \qquad
\Phi_f \;=\; e^{-V} \bb M^\dagger \nabla f \;, \qquad
\Phi^*_f \;=\; e^{-V} \bb M \nabla f\;.
\end{equation}
Denote by $\mathcal C^{a,b}_{\mc A,\mc B}$, $a$, $b\in \bb R$ the set of
bounded functions $f$ in $C^2(\overline{\Omega})\cap W^{1,2}(\Omega)$
such that $f(\bs{x})=a$, $\bs x\in \partial A$, $f(\bs{y})=b$, $\bs y\in \partial B$.

\begin{proposition}[Dirichlet's Principle]
\label{prop1}
Let $\mc A$, $\mc B$ be two disjoint open subsets of $\bb R^d$ satisfying Assumption S. Then,
\begin{equation}
\label{26}
\Cap (\mc A,\,\mc B) \;=\; \inf_{f\in \mathcal C^{1,0}_{\mc A,\mc B}}
\, \inf_{\varphi \in \mathcal  F^{(0)}} \,
\< \Phi_f - \varphi \,,\, \Phi_f - \varphi \>\;.
\end{equation}
The minimum is attained at $f=(1/2) (h_{\mc A,\mc B} + h^*_{\mc A,\mc B})$, and
$\varphi = \Phi_f - \Psi_{h_{\mc A,\mc B}}$.
\end{proposition}

\begin{proposition}[Thomson principle]
\label{prop2}
Let $\mc A$, $\mc B$ be disjoint two open subsets of $\bb R^d$ satisfying Assumptions S. Then,
\begin{equation}
\label{28}
\Cap (\mc A,\,\mc B) \;=\; \sup_{f\in \mathcal C^{0,0}_{\mc A,\mc B}} \,
\sup_{\varphi \in \mathcal  F^{(1)}} \,
\frac 1{\< \Phi_f - \varphi \,,\, \Phi_f - \varphi \>} \;\cdot
\end{equation}
The maximum is attained at $f=(h_{\mc A,\mc B} - h^*_{\mc A,\mc B})/2 \, \Cap (\mc A,\,\mc B)$, and
$\varphi = \Phi_f - \Psi_{g_{\mc A,\mc B}}$, where $g_{\mc A,\mc B} = h_{\mc A,\mc B}/\Cap
(\mc A,\mc B)$.
\end{proposition}

These principles can be extended to other contexts, provided that the underlying process admits a unique, stationary probability measure and that the sets $\mc A$ and $\mc B$ are recurrent and sufficiently regular. In particular, in Section~\ref{sec7} we establish these principles for non-reversible elliptic operators in compact manifolds without boundaries.

The article is organized as follows. In Sections~\ref{sec1} and \ref{sec7} we prove Propositions~\ref{prop1} and \ref{prop2} and an
extension of these results to the context of elliptic operators in compact manifolds without boundaries. As an application of these results, in Section~\ref{sec3}, we introduce a class of non-reversible metastable diffusion process and we provide a sharp estimate for the transition time between two wells, the so-called Eyring-Kramers formula. The proof of this estimate is given in Sections~\ref{sec4}--\ref{sec8}.

\section{Dirichlet's and Thomson's principles on Euclidean spaces}
\label{sec1}

Denote by $C([0,\infty), \bb R^d)$ the space of continuous functions
$\omega$ from $\bb R_+$ to $\bb R^d$ endowed with the topology of
uniform convergence on bounded intervals. Let $X_t$, $t\ge 0$ be the
one-dimensional projections: $X_t(\omega) = \omega(t)$, $\omega \in
C([0,\infty), \bb R^d)$. We sometimes represent $X_t$ as $X(t)$.

It follows from the assumptions on $\bb M$ and $V$, and from \cite
[Theorems~1.10.4 and 1.10.6]{p95} that there exists a unique solution, denoted hereafter by $\{\bb P_x : x\in \bb R^d\}$, to the martingale problem associated to the generator $\mc L$ introduced in \eqref{gen1}. Moreover, the family $\{\bb P_x : x\in \bb R^d\}$ is strong Markov and possesses the Feller property. Expectation with
respect to $\bb P_x$ is expressed as $\bb E_x$.

The process $X_t$ can be represented in terms of a stochastic
differential equation.  Denote by $\bb K\colon\bb R^d \to \bb R^d \otimes
\bb R^d$ the square root of $\bb S (\bs{x})$, in the sense that $\bb K
(\bs{x})$ is a positive-definite, symmetric matrix such that $\bb K
(\bs{x}) \bb K(\bs{x}) = \bb S (\bs{x})$. By \cite[Lemma~6.1.1]{f},
the entries of $\bb K$ inherit the regularity properties of $\bb S$:
$\bb K_{m,n}$ belongs to $C^3(\bb R^d)$ for $1\le m,\,n\le d$. Recall
the drift $\bs{b}(\cdot)$ from \eqref{drift1}. Then, the process $X_t$
is the unique solution of the stochastic differential equation
\begin{equation*}
d X_t \;=\; \bs b (X_t) \, dt \;+\; \sqrt{2} \, \bb K (X_t) \, dB_t\;,
\end{equation*}
where $B_t$ stands for a $d$-dimensional Brownian motion.

In view of the regularity of $\bb M$ and $V$ and conditions \eqref{18}, \eqref{20}, by \cite[Theorem~6.1.3]{p95}, the process $X_t$ is positive recurrent. Furthermore, by \cite[Theorem~4.9.6]{p95}, $\bb E_{\bs x} [H_\mc C]<\infty$ for all open sets $\mc
C\subset \bb R^d$ and all $\bs x\not\in \mc C$. Finally, an elementary computation shows that the probability measure $\mu$ is stationary.

The solutions of the elliptic equation \eqref{22} can be represented
in terms of the process $X_t$.  Fix $0<\alpha<1$, a bounded function
$\mf g$ in $L^2(\Omega) \cap C^{\alpha}(\overline{\Omega})$ and a
bounded function $\mf b$ in $W^{1,2}(\Omega)\cap
C^{2,\alpha}(\overline{\Omega})$.  Assume that $\Omega$ has a
$C^{2,\alpha}$-boundary.  It follows from the proof of \cite[Theorem
6.5.1]{f} and from the positive recurrence that the unique solution
$u$ of \eqref{22} can be represented as
\begin{equation}
\label{23}
u(\bs{x}) \;=\; \bb E_{\bs x} \big[ \mf b (X (H_{\Omega^c}))\, \big]
\;+\; \bb E_{\bs x} \Big[ \int_0^{H_{\Omega^c}} \mf g (X_t) \,
dt \, \Big]\;.
\end{equation}

In particular if $\mc A$, $\mc B$ represent two open sets satisfying
Assumption S, the equilibrium potential between $\mc A$ and $\mc B$,
introduced just above \eqref{06}, is given by
\begin{equation}
\label{30}
h_{\mc A,\mc B} (\bs{x}) \;=\; \bb P_{\bs x} [H_\mc A<H_\mc B]\;.
\end{equation}

\subsection{Properties of the capacity}

We present in this subsection some elementary properties of the
capacity.  We begin with an alternative formula for the
capacity. Unless otherwise stated, until the end of this section, the
open subsets $\mc A$, $\mc B$ satisfy Assumption S.

\begin{lemma}
\label{lem1}
Recall that the capacity $\Cap(\mc A,\mc B)$ was defined in \eqref{06}. It holds
\begin{equation*}
\Cap (\mc A,\,\mc B) \;=\; \int_{\bb R^d}
\nabla h_{\mc A,\mc B}(\bs{x}) \cdot \bb S(\bs{x}) \, \nabla
h_{\mc A,\mc B}(\bs{x}) \, \mu(d\bs{x})\;.
\end{equation*}
\end{lemma}

\begin{proof}
Since the function $h_{\mc A,\mc B}$ is harmonic on $\Omega =
(\overline{\mc A}\cup \overline{\mc B})^c$, and since it
is equal to $1$ on the set $\partial \mc A$ and $0$ on
the set $\partial \mc B$, the capacity $\Cap (\mc A,\,\mc B)$ can be written as
\begin{equation}
\label{c11}
\begin{aligned}
& \int_{\partial \mc A} h_{\mc A,\mc B}(\bs{x})  \big[ \bb M(\bs{x}) \, \nabla
h_{\mc A,\mc B}(\bs{x}) \big] \cdot \bs n_{\Omega}(\bs{x})  \, e^{-V(\bs{x})} \sigma(d\bs{x}) \\
&\quad \;+\;
\int_{\partial \mc B} h_{\mc A,\mc B}(\bs{x}) \big[ \bb M(\bs{x}) \, \nabla
h_{\mc A,\mc B}(\bs{x}) \big] \cdot \bs n_{\Omega}(\bs{x})  \, e^{-V(\bs{x})} \sigma(d\bs{x}) \\
&\qquad \;-\;
\int_{\Omega} h_{\mc A,\mc B}(\bs{x}) \, \nabla \cdot \big[ e^{-V(\bs{x})} \, \bb M(\bs{x})
\, \nabla h_{\mc A,\mc B}(\bs{x}) \big] \, d\bs{x} \;.
\end{aligned}
\end{equation}
Note that the function $h_{\mc A,\mc B}$ belongs to
$C^{2+\alpha}(\overline{\Omega})\cap W^{1,2}(\Omega)$ and the matrix
$\bb S$ represents the symmetric part of the matrix $\bb M$. Hence, if
we apply the divergence theorem at the thrid term in the previous
expression, then the resulting boundary terms exactly concide with the
first two terms of the same expression, and we obtain that \eqref{c11}
equals
\begin{equation*}
\int_{\Omega} \nabla h_{\mc A,\mc B}(\bs{x}) \cdot \bb{S}(\bs x) \,\nabla h_{\mc A,\mc B} (\bs{x}) \mu (d\bs x)\;.
\end{equation*}
As the equilibrium potential is constant in $\mc A\cup \mc B$, we may replace
in the last formula the integration domain $\Omega$ to $\bb R^d$, which
completes the proof.
\end{proof}

Since $h_{\mc B,\mc A} = 1-h_{\mc A,\mc B}$, it follows from the previous lemma that
the capacity is symmetric: for every disjoint subsets $\mc A$, $\mc B$ of $\bb
R^d$,
\begin{equation}
\label{08}
\Cap (\mc A,\,\mc B) \;=\; \Cap (\mc B,\,\mc A) \;.
\end{equation}

\smallskip\noindent{\bf Adjoint generator.}  Denote by $\mathcal L^*$
the $L^2(\mu)$ adjoint of the generator $\mathcal L$, which acts on
functions in $C^2(\bb R^d)$ as
\begin{equation*}
(\mathcal L^* f) (\bs{x}) \;=\; e^{V(\bs{x})} \, \nabla \cdot \big\{ e^{-V(\bs{x})}
\, \bb M^\dagger(\bs{x}) (\nabla f)(\bs{x}) \big\} \;.
\end{equation*}
Let $\mathcal S$ be the symmetric part of the generator $\mathcal L$,
defined as $\mathcal S = (1/2) (\mathcal L + \mathcal L^*)$ and acting
on $C^2(\mathbb R^d)$ as $\mathcal S f= e^V \nabla \cdot (e^{-V}\mathbb
S \nabla f)$.

Denote by $\Cap^*(\mc A,\,\mc B)$ the capacity between the open sets
$\mc A$, $\mc B$ with respect to the adjoint generator $\mathcal
L^*$. In view of \eqref{06}, this capacity $\Cap^*(\mc A,\,\mc B)$ is
defined as
\begin{equation}
\label{10}
\Cap^* (\mc A,\,\mc B) \;=\; \int_{\partial \mc A} \big[
\bb M^\dagger(\bs{x}) \, \nabla
h^*_{\mc A,\mc B}(\bs{x}) \big] \cdot \bs n_{\Omega}(\bs{x})
\, e^{-V(\bs{x})} \sigma(d\bs{x})\;,
\end{equation}
where $h^*_{\mc A,\mc B}\colon \bb R^d \to [0,1]$, called the equilibrium
potential between $\mc A$ and $\mc B$ for the adjoint generator, is
the unique solution in $C^2(\overline{\Omega})\cap W^{1,2}(\Omega)$ of
the elliptic equation
\begin{equation*}
\begin{cases}
(\mathcal L^* u) (\bs{x}) \;=\; 0 &  x\in \Omega \;, \\
u (\bs{x}) \;=\; \chi_\mc A (\bs{x})  &   x\in \partial \Omega \;,
\end{cases}
\end{equation*}
where $\chi_\mc A(\cdot)$ represents the indicator function of the set
$\mc A$.

The next lemma states that the capacity between two disjoint subsets
$\mc A$, $\mc B$ of $\bb R^d$ coincides with the capacity with respect
to the adjoint process. Recall that we are assuming that $\mc A$ and
$\mc B$ fulfill Assumption S.

\begin{lemma}
\label{lem2}
For every open subsets $\mc A$, $\mc B$ of $\bb R^d$,
\begin{equation*}
\Cap (\mc A,\,\mc B) \;=\; \Cap^* (\mc A,\,\mc B) \;.
\end{equation*}
\end{lemma}

\begin{proof}
As in the proof of Lemma \ref{lem1}, we may write $\Cap (\mc A,\,\mc B)$ as
\begin{equation}
\label{12}
\begin{aligned}
& \int_{\partial \mc A} h^*_{\mc A,\mc B}(\bs{x})  \big[ \bb M(\bs{x}) \, \nabla
h_{\mc A,\mc B}(\bs{x}) \big] \cdot \bs n_{\Omega}(\bs{x})  \, e^{-V(\bs{x})} \sigma(d\bs{x}) \\
&\quad \;-\;
\int_{\partial \mc B} h^*_{\mc A,\mc B}(\bs{x}) \big[ \bb M(\bs{x}) \, \nabla
h_{\mc A,\mc B}(\bs{x}) \big] \cdot \bs n_{\Omega}(\bs{x})  \, e^{-V(\bs{x})} \sigma(d\bs{x}) \\
&\qquad \;-\;
\int_{\Omega} h^*_{\mc A,\mc B}(\bs{x}) \, \nabla \cdot \big[ e^{-V(\bs{x})} \, \bb M(\bs{x})
\, \nabla h_{\mc A,\mc B}(\bs{x}) \big] \, d\bs{x} \;,
\end{aligned}
\end{equation}
By the arguments presented in the proof of the previous lemma, and the
divergence theorem, this expression is equal to
\begin{equation*}
\int_{\Omega} \nabla h^*_{\mc A,\mc B}(\bs{x}) \,\bb M(\bs{x}) \, \nabla
h_{\mc A,\mc B}(\bs{x}) \, \mu(d\bs{x}) \;=\;
\int_{\Omega} \nabla h_{\mc A,\mc B}(\bs{x}) \,\bb M^\dagger(\bs{x}) \, \nabla
h^*_{\mc A,\mc B}(\bs{x}) \, \mu(d\bs{x}) \;.
\end{equation*}
By the divergence theorem once more, we obtain that this integral is
equal to the sum \eqref{12}, in which $\nabla h_{\mc A,\mc B}$ and $\nabla
h^*_{\mc A,\mc B}$ are interchanged and $\bb M(\bs{x})$ is replaced by $\bb
M^\dagger(\bs{x})$. We may remove the function $h^*_{\mc A,\mc B}$ in the first line
because it is equal to $1$ on $\partial \mc A$.  The second line vanishes
because $h^*_{\mc A,\mc B}$ is equal to $0$ at $\partial \mc B$, and the third line
vanishes because $h^*_{\mc A,\mc B}$ is harmonic on $\Omega$. This completes
the proof of the lemma.
\end{proof}

Recall from \eqref{2-3} the definition of the vector fields $\Psi_f$,
$\Phi_f$, $\Phi^*_f$. Note that for every function $f$, $g$ in
$C^2(\overline{\Omega})\cap W^{1,2}(\Omega)$,
\begin{equation*}
\<\Psi_{f} \,,\, \Psi_{g}\> \;=\; \int_{\Omega} (\nabla f)(\bs{x}) \cdot
\, \bb S (\bs{x}) \, (\nabla g)(\bs{x}) \, \mu(d\bs{x}) \;.
\end{equation*}
In particular, by Lemma \ref{lem1},
\begin{equation}
\label{03}
\Cap (\mc A,\,\mc B) \;=\; \<\Psi_{h_{\mc A,\mc B}} \,,\, \Psi_{h_{\mc A,\mc B}}\> \;.
\end{equation}
On the other hand, for every function $f$, $g$ in
$C^2(\overline{\Omega})\cap W^{1,2}(\Omega)$,
\begin{equation}
\label{04}
\begin{aligned}
& \<\Phi_f \,,\, \Psi_{g}\> \;=\; \int_{\Omega} (\nabla g)(\bs{x})
\cdot \bb M^\dagger (\bs{x}) \,  (\nabla f)(\bs{x}) \, \mu(d\bs{x}) \;, \\
&\qquad \<\Phi^*_f \,,\, \Psi_{g}\> \;=\;
\int_{\Omega} (\nabla g)(\bs{x})  \cdot \bb M (\bs{x}) \,
(\nabla f)(\bs{x}) \, \mu(d\bs{x}) \;.
\end{aligned}
\end{equation}

\begin{proof}[Proof of Proposition \ref{prop1}]
We first claim that for all $f\in \mathcal C^{1,0}_{\mc A,\mc B}$, $\varphi
\in \mathcal F^{(0)}$,
\begin{equation}
\label{02}
\<\Phi_f - \varphi \,,\, \Psi_{h_{\mc A,\mc B}}\>  \;=\; \Cap (\mc A,\,\mc B) \;.
\end{equation}
Indeed, on the one hand, for any $f\in \mathcal C^{1,0}_{\mc A,\mc B}$, $\varphi
\in \mathcal F^{(0)}$, by definition of $\Psi_f$, by the divergence
theorem, and since $f$ is bounded and belongs to $ W^{1,2}(\Omega)$,
and since $\int_{\Omega} \Vert\varphi(\bs x)\Vert^2 \exp\{V(\bs{x})\} d\bs{x}$ is
finite, $\< \Psi_f , \varphi\>$ is equal to
\begin{equation*}
\int_{\Omega} \varphi(\bs{x}) \cdot (\nabla
f)(\bs{x}) \, d\bs{x} \;=\; \int_{\partial \Omega} f(\bs{x}) \, \varphi(\bs{x}) \cdot \bs
n_\Omega(\bs{x}) \, \sigma(d\bs{x})
\;-\; \int_{\Omega} (\nabla \cdot \varphi)(\bs{x}) \, f(\bs{x}) \, d\bs{x} \;.
\end{equation*}
The second integral on the right-hand side vanishes because $\varphi$
is divergence free in $\Omega$, while the first integral vanishes
because $f$ is constant on each set $\partial \mc A$, $\partial \mc B$ and the
vector field $\varphi$ belongs to $\mathcal F^{(0)}$.

Therefore, by \eqref{04},
\begin{equation*}
\<\Phi_f - \varphi \,,\, \Psi_{h_{\mc A,\mc B}}\>
\;=\; \<\Phi_f \,,\, \Psi_{h_{\mc A,\mc B}}\> \;=\;
\int_{\Omega} (\nabla f) (\bs{x}) \cdot \bb M(\bs{x}) \,
(\nabla h_{\mc A,\mc B})(\bs{x}) \, e^{-V(\bs{x})}\, d\bs{x} \;.
\end{equation*}
By the divergence theorem, the previous expression is equal to
\begin{equation*}
-\, \int_{\Omega} f (\bs{x}) (\mathcal L h_{\mc A,\mc B})(\bs{x}) \, \mu(d\bs{x})
\;+\; \int_{\partial \Omega} f (\bs{x}) \, e^{-V(\bs{x})}\, \bb M(\bs{x}) \,
(\nabla h_{\mc A,\mc B})(\bs{x}) \cdot \bs n_{\Omega}(\bs{x}) \, \sigma(d\bs{x})\;.
\end{equation*}
The first integral vanishes because $h_{\mc A,\mc B}$ is harmonic on
$\Omega$. Since $f$ belongs to $\mathcal C^{1,0}_{\mc A,\mc B}$, we may first
restrict the second integral to $\partial \mc A$, and then remove the
function $f$ to conclude that
\begin{equation*}
\<\Phi_f - \varphi \,,\, \Psi_{h_{\mc A,\mc B}}\>  \;=\;
\int_{\partial \mc A} e^{-V(\bs{x})}\, \bb M(\bs{x}) \,
(\nabla h_{\mc A,\mc B})(\bs{x}) \cdot \bs n_{\Omega}(\bs{x}) \, \sigma(d\bs{x})\;,
\end{equation*}
which proves claim \eqref{02} in view of \eqref{06}.

By \eqref{02} and by the Cauchy-Schwarz inequality
$\< \varphi , \psi \>^2 \le \< \varphi , \varphi \> \,\< \psi \,,\, \psi \> $, for every $f\in \mathcal C^{1,0}_{\mc A,\mc B}$, $\varphi \in \mathcal F^{(0)}$,
\begin{equation*}
\Cap (\mc A,\,\mc B)^2 \;=\; \<\Phi_f - \varphi \,,\, \Psi_{h_{\mc A,\mc B}}\>^2 \;\le\;
\<\Phi_f - \varphi \,,\, \Phi_f - \varphi \> \,
\< \Psi_{h_{\mc A,\mc B}} \,,\, \Psi_{h_{\mc A,\mc B}}\> \;.
\end{equation*}
By \eqref{03}, the last term is equal to $\Cap (\mc A,\,\mc B)$, which proves
that
\begin{equation*}
\<\Phi_f - \varphi \,,\, \Phi_f - \varphi \> \;\ge\; \Cap (\mc A,\,\mc B)
\end{equation*}
for all $f\in \mathcal C^{1,0}_{\mc A,\mc B}$ and $\varphi \in \mathcal F^{(0)}$.

To complete the proof of the proposition, it remains to show that
$\varphi = \Phi_f - \Psi_{h_{\mc A,\mc B}}$ belongs to $\mathcal F^{(0)}$ for
$f = (1/2) (h_{\mc A,\mc B} + h^*_{\mc A,\mc B})$. This is indeed the case. Recall
that $h_{\mc A,\mc B}$, $h^*_{\mc A,\mc B}$ are bounded and belong to
$C^2(\overline{\Omega})\cap W^{1,2}(\Omega)$. On the one hand, by
definition of $f$, for every $\bs x\in \Omega$,
\begin{equation*}
\nabla \cdot [\Phi_f - \Psi_{h_{\mc A,\mc B}}] \;=\;
e^{-V(\bs{x})} \, (\mathcal L^* f)(\bs{x}) \;-\;
e^{-V(\bs{x})} \, (\mathcal S h_{\mc A,\mc B})(\bs{x}) \;,
\end{equation*}
where $\mathcal S = (1/2) (\mathcal L + \mathcal L^*)$.  Since $f= (1/2) (h_{\mc A,\mc B} +
h^*_{\mc A,\mc B})$ by definition of $\mathcal S$, this expression is equal to
\begin{equation*}
\frac 12 \, e^{-V(\bs{x})} \, (\mathcal L^* h^*_{\mc A,\mc B})(\bs{x}) \;-\;
\frac 12 \, e^{-V(\bs{x})} \, (\mathcal L h_{\mc A,\mc B})(\bs{x}) \;=\; 0 \;.
\end{equation*}
On the other hand, by definition of $f$
\begin{align*}
& \int_{\partial \mc A} [\Phi_f (\bs{x}) - \Psi_{h_{\mc A,\mc B}} (\bs{x})] \cdot \bs n_{\Omega}(\bs{x}) \,
\sigma(d\bs{x}) \\
&\quad  =\; \frac 12 \int_{\partial A} e^{-V(\bs{x})} \, \bb M^\dagger(\bs{x}) \,
(\nabla h_{\mc A,\mc B}^*)(\bs{x}) \cdot \bs n_{\Omega}(\bs{x}) \, \sigma(d\bs{x}) \\
&\qquad -\; \frac 12 \int_{\partial A} e^{-V(\bs{x})} \, \bb M(\bs{x}) \,
(\nabla h_{\mc A,\mc B})(\bs{x}) \cdot \bs n_{\Omega}(\bs{x}) \, \sigma(d\bs{x})\;.
\end{align*}
By definition of the capacities and by Lemma \ref{lem2}, this
expression is equal to $(1/2)\{\Cap (\mc A,\,\mc B) - \Cap^* (\mc A,\,\mc B)\}=0$.  As
$\Phi_f - \Psi_{h_{\mc A,\mc B}}$ is divergence free on $\Omega$, the same
identity holds at $\partial \mc B$, which concludes the proof of the
proposition.
\end{proof}

\begin{proof}[Proof of Proposition \ref{prop2}]
We claim that for every $f\in \mathcal C^{0,0}_{\mc A,\mc B}$, $\varphi \in \mathcal
F^{(1)}$,
\begin{equation}
\label{13}
\< \Phi_f - \varphi \,,\, \Psi_{h_{\mc A,\mc B}} \> \; =\; -1\;.
\end{equation}
Indeed,
\begin{equation*}
\< \Phi_f - \varphi \,,\, \Psi_{h_{\mc A,\mc B}} \> \; =\;
\int_{\Omega} (\nabla f)(\bs{x}) \cdot \bb M(\bs{x}) \, (\nabla h_{\mc A,\mc B}) (\bs{x})
\, \mu(d\bs{x}) \;-\; \int_{\Omega} \varphi (\bs{x}) \cdot (\nabla h_{\mc A,\mc B}) (\bs{x})
\, d\bs{x} \;.
\end{equation*}
By the divergence theorem, this expression is equal to
\begin{align*}
& -\, \int_{\Omega} f(\bs{x}) \, (\mathcal L h_{\mc A,\mc B}) (\bs{x}) \, \mu(d\bs{x})
\;+\; \int_{\partial \Omega} f(\bs{x}) \, \bb M(\bs{x}) \, (\nabla h_{\mc A,\mc B}) (\bs{x})
\cdot \bs n_{\Omega}(\bs{x}) \, e^{-V(\bs{x})} \, \sigma(d\bs{x})\\
& \quad \;+\; \int_{\Omega} (\nabla \cdot \varphi) (\bs{x}) \, h_{\mc A,\mc B} (\bs{x}) \, d\bs{x}
\;-\; \int_{\partial \Omega} h_{\mc A,\mc B} (\bs{x}) \, \varphi (\bs{x})
\cdot \bs n_{\Omega}(\bs{x}) \, \sigma(d\bs{x}) \;.
\end{align*}
The integrals over $\Omega$ vanish because $h_{\mc A,\mc B}$ is harmonic and
the vector field $\varphi$ is divergence free. The second integral in
the first line also vanishes because $f$ belongs to $\mathcal
C^{0,0}_{\mc A,\mc B}$. The last integral is equal to $-1$ because $\varphi$
belongs to $\mathcal F^{(1)}$. This proves claim \eqref{13}.

Therefore, by the Cauchy-Schwarz inequality,
\begin{equation*}
1\;=\; \< \Phi_f - \varphi \,,\, \Psi_{h_{\mc A,\mc B}} \>^2 \;\le\;
\< \Phi_f - \varphi \,,\, \Phi_f - \varphi \>\,
\< \Psi_{h_{\mc A,\mc B}} \,,\, \Psi_{h_{\mc A,\mc B}} \> \;.
\end{equation*}
Since $\< \Psi_{h_{\mc A,\mc B}} \,,\, \Psi_{h_{\mc A,\mc B}} \> = \Cap(\mc A,\,\mc B)$, it
follows from the previous relation that
\begin{equation*}
\Cap(\mc A,\,\mc B) \;\ge\; \sup_{f\in \mathcal C^{0,0}_{\mc A,\mc B}} \sup_{\varphi \in \mathcal
  F^{(1)}} \frac 1{\< \Phi_f - \varphi \,,\, \Phi_f - \varphi
  \>}\;\cdot
\end{equation*}

To complete the proof of the proposition, it remains to check that
$\varphi = \Phi_f - \Psi_{g_{\mc A,\mc B}}$ belongs to $\mathcal F^{(1)}$, where
$g_{\mc A,\mc B} = h_{\mc A,\mc B}/\Cap (\mc A,\,\mc B)$, $f=(1/2) (h_{\mc A,\mc B} - h^*_{\mc A,\mc B})/\Cap
(\mc A,\,\mc B)$. Since for $x\in\Omega$,
\begin{align*}
(\nabla \cdot \varphi)(\bs{x}) \; &=\; \frac 1{2 \, \Cap (\mc A,\,\mc B)}\, e^{-V(\bs{x})}
\Big\{ [\mathcal L^* (h_{\mc A,\mc B} - h^*_{\mc A,\mc B})](\bs{x}) - 2 (\mathcal S h_{\mc A,\mc B})(\bs{x})
\Big\} \\
 &=\; \frac {-1}{2 \, \Cap (\mc A,\,\mc B)}\, e^{-V(\bs{x})}
\Big\{ (\mathcal L^* h^*_{\mc A,\mc B}) (\bs{x}) +  (\mathcal L h_{\mc A,\mc B})(\bs{x}) \Big\}  \;,
\end{align*}
the vector field $\varphi$ is divergence free on $\Omega$. On the
other hand, the integral of $\varphi \cdot \bs n_{\Omega}$ over the set
$\partial \mc A$ is equal to
\begin{equation*}
\frac {-1}{2 \, \Cap (\mc A,\,\mc B)}\, \int_{\partial \mc A}
\big\{ \bb M(\bs{x}) \, (\nabla h_{\mc A,\mc B}) (\bs{x}) + \bb M^\dagger(\bs{x}) \,
(\nabla h^*_{\mc A,\mc B}) (\bs{x}) \big\}
\cdot \bs n_{\Omega}(\bs{x}) \, e^{-V(\bs{x})} \, \sigma(d\bs{x}) \;.
\end{equation*}
By definition \eqref{06} and by Lemma \ref{lem2}, this expression is
equal to $-1$, which proves that $\varphi$ belongs to $\mathcal
F^{(1)}$. This completes the proof of the proposition.
\end{proof}

\subsection{The reversible case.}

When the matrix $\bb M$ is symmetric and the generator $\mathcal L$ is
symmetric in $L^2(\mu)$, the previous variational formulae are
simplified and we recover the Dirichlet's and the Thomson's principles for
reversible diffusions.

Fix two open subsets $\mc A$, $\mc B$ of $\bb R^d$.  On the one hand, since
$h_{\mc A,\mc B} = h^*_{\mc A,\mc B}$, and since all vector fields $\Phi_f$,
$\Phi^*_f$, $\Psi_f$ coincide, by Proposition \ref{prop1}, the minimum
over $\varphi$ in \eqref{26} is attained at $\varphi=0$, so that
\begin{equation}
\label{27}
\Cap (\mc A,\,\mc B) \;=\; \inf_{f\in \mathcal C^{1,0}_{\mc A,\mc B}}
\< \Psi_f  \,,\, \Psi_f \> \;=\; \inf_{f\in \mathcal C^{1,0}_{\mc A,\mc B}}
\int_{\Omega} (\nabla f)(\bs{x}) \cdot
\, \bb S (\bs{x}) \, (\nabla f)(\bs{x}) \, \mu(d\bs{x})\;,
\end{equation}
which is the well-known Dirichlet's principle.

Similarly, the supremum over $f$ in \eqref{28} is attained at $f=0$,
so that
\begin{equation}
\label{29}
\Cap (\mc A,\,\mc B) \;=\; \sup_{\varphi \in \mathcal  F^{(1)}}
\frac 1{\< \varphi \,,\, \varphi \>} \;,
\end{equation}
which is the classical Thomson principle.

\subsection{The equilirbium measure.}

Recall that the sets $\mc A$, $\mc B$ are assumed to fulfill
Assumption S.  Let $\nu_{\mc A,\mc B}$ be the equilirbium measure on
$\partial \mc A$:
\begin{equation}
\label{2-2}
\nu_{A,B} (d\bs{x}) \;=\; \frac {1}{\Cap(\mc A,\,\mc B)} \, \bb M^\dagger(\bs{x}) \, (\nabla
h^*_{\mc A,\mc B}) (\bs{x}) \cdot \bs n_{\Omega}(\bs{x})\, e^{-V(\bs{x})}\, \sigma(d\bs{x})\;.
\end{equation}
By maximum principle, the equilibrium potential $h^*_{\mc A,\mc B}$ is bounded by $1$, and since $\partial \mc A$ is in the $1$-level set of $h^\ast$, it holds $\nabla h^*_{\mc A,\mc B} = c(\bs{x}) \,
  \bs n_{\Omega}(\bs{x})$ for some nonnegative scalar function $c\colon \partial \mc A \to \bb R_+$, so that
\begin{equation*}
\bb M^\dagger(\bs{x}) \, (\nabla h^*_{\mc A,\mc B}) (\bs{x})
\cdot \bs n_{\Omega}(\bs{x}) \;=\;
\bb S(\bs{x}) \, (\nabla h^*_{\mc A,\mc B}) (\bs{x})
\cdot \bs n_{\Omega}(\bs{x}) \; \ge\; 0\;.
\end{equation*}
This shows that $\nu_{\mc A,\mc B}$ is a probability measure.

\begin{proposition}
\label{prop3}
For any two bounded, open subsets $\mc A$, $\mc B$ satisfying Assumption S, and for every bounded function $f$ in $C^\alpha (\bb R^d)$, $0<\alpha<1$,
\begin{equation}
\label{11}
\bb E_{\nu_{\mc A,\mc B}} \Big[ \int_0^{H_\mc B} f(X_s) \, ds \Big] \;=\;
\frac {1} {\Cap (\mc A,\,\mc B)} \int_{\bb R^d} h^*_{\mc A,\mc B} (\bs{x}) \, f(\bs{x}) \,
e^{-V(\bs{x})}\, d\bs{x} \;.
\end{equation}
\end{proposition}

\begin{proof}
Fix a bounded function $f$ in $C^\alpha (\bb R^d)$, and let $\Omega_\mc B
= \bb R^d \setminus \overline{\mc B}$. Denote by $u$ the unique solution
in $W^{1,2}(\Omega_\mc B) \cap C^{2,\alpha}(\Omega_\mc B)$ of the elliptic
equation \eqref{22} with $\Omega = \Omega_\mc B$, $\mf g = f$, $\mf b=0$.
In view of the representation \eqref{23} of $u$ and by definition of
the equilibrium measure $\nu_{A,B}$, the left-hand side of \eqref{11} equals
\begin{equation*}
\frac {1}{\Cap (\mc A,\,\mc B)} \int_{\partial \mc A} u(\bs{x})\, [\bb M^\dagger(\bs{x})\,
\nabla h^*_{\mc A,\mc B} (\bs{x}) ]
\cdot \bs n_{\Omega} (\bs{x}) \, e^{-V(\bs{x})} \, \sigma (d\bs{x}) \;.
\end{equation*}
The integral of the same expression at $\partial \mc B$ vanishes due to
the presence of the function $u$. Hence, by the divergence theorem,
this expression is equal to
\begin{equation*}
\frac {1} {\Cap (\mc A,\,\mc B)} \int_{\Omega} \nabla \cdot
\Big\{ [\bb M^\dagger(\bs{x})\, \nabla h^*_{\mc A,\mc B} (\bs{x}) ] \, e^{-V(\bs{x})} \, u(\bs{x})
\Big\} \, d\bs{x}  \;.
\end{equation*}
Since the equilibrium potential $h^*_{\mc A,\mc B}$ is harmonic on $\Omega$,
the previous equation is equal to
\begin{align*}
\frac {1} {\Cap (\mc A,\,\mc B)} \int_{\Omega}   \, e^{-V(\bs{x})} \,
\nabla h^*_{\mc A,\mc B} (\bs{x}) \cdot \, \bb M(\bs{x}) \, (\nabla u)(\bs{x}) \, d\bs{x} \;.
\end{align*}
By the divergence theorem and since the equilibrium potential
$h^*_{\mc A,\mc B}$ is equal to $1$ on $\partial \mc A$ and $0$ on $\partial \mc B$,
this expression becomes
\begin{equation}
\label{c010}
\begin{aligned}
& \frac {1} {\Cap (\mc A,\,\mc B)} \int_{\partial \mc A}
e^{-V(\bs{x})} \, \bb M(\bs{x}) \, (\nabla u)(\bs{x})  \cdot \bs
n_{\Omega}(\bs{x}) \, \sigma(d\bs{x}) \\
&\quad - \; \frac {1} {\Cap (\mc A,\,\mc B)} \int_{\Omega}
h^*_{\mc A,\mc B} (\bs{x})  \nabla \cdot \Big\{ e^{-V(\bs{x})} \, \bb M(\bs{x}) \,
(\nabla u)(\bs{x}) \Big\} \, d\bs{x} \;.
\end{aligned}
\end{equation}
As $\bs n_{\Omega} = - \bs n_{\mc A}$ on $\partial \mc A$ and $\mc L u = -f$
on $\mc A$, the first term of \eqref{c010} is equal to
\begin{equation*}
\frac {-1} {\Cap (\mc A,\,\mc B)} \int_{\mc A}
\nabla \cdot \big\{ e^{-V(\bs{x})} \, \bb M(\bs{x}) \, (\nabla u)(\bs{x}) \big\} \,
d\bs{x} \;=\; \frac {1} {\Cap (\mc A,\,\mc B)} \int_{\mc A} f(\bs{x}) \, e^{-V(\bs{x})}\, d\bs{x} \;.
\end{equation*}
Since the equilibrium potential $h^*_{\mc A,\mc B}$ is equal to $1$ on $\mc A$, we
may insert it in the integral.
On the other hand, by definition of $u$,
$$\nabla \cdot \{ e^{-V(\bs{x})} \,
\bb M(\bs{x}) \, (\nabla u)(\bs{x}) \} = \exp\{-V(\bs{x})\} (\mathcal L u)(\bs{x}) = -
\exp\{-V(\bs{x})\} f(\bs{x})\;.$$
Thus, the second term of \eqref{c010} is equal to
\begin{equation*}
\frac {1} {\Cap (\mc A,\,\mc B)} \int_{\Omega} h^*_{\mc A,\mc B} (\bs{x}) \, f(\bs{x}) \,
e^{-V(\bs{x})}\, d\bs{x} \;.
\end{equation*}
Since the equilibrium potential $h^*_{\mc A,\mc B}$ vanishes at $\mc B$, we may
extend this integral to the set $\mc B$, which completes the proof of
the proposition.
\end{proof}

Proposition \ref{prop3} can be restated as follows. Let $u$ be the
solution of \eqref{22} with $\mf g = - f$, $\mf b=0$, $\Omega =
\overline{\mc B}^c$.  Then,
\begin{equation}
\label{2-1}
\int_{\partial \mc A} u(\bs{x})\, \nu_{\mc A,\mc B}(d\bs{x}) \;=\;
\frac {1} {\Cap (\mc A,\,\mc B)}
\int_{\bb R^d} h^*_{\mc A,\mc B} (\bs{x}) \, f(\bs{x}) \, e^{-V(\bs{x})}\, d\bs{x} \;.
\end{equation}

Equation \eqref{11} has been derived previously in the
  context of non-reversible Markov chains in \cite[Proposition~A.2]{BL2}, and for non-reversible diffusions in the case $f=1$ in \cite[Proposition~1.8]{ln}.

\section{Dirichlet's and Thomson's principles on a compact manifold}
\label{sec7}
\subsection{Notation}
Let $\mf M$ be a compact manifold without boundary, equipped with a smooth Riemannian tensor $g=a^{-1}$. Since $\mf M$ is compact, $a$ satisfies the ellipticity condition \eqref{19}. Denote by $\nabla$ the gradient,
by $\nabla \cdot$ the divergence and by $\Delta=\nabla\cdot \nabla$ the Laplace-Beltrami operators on $\mf M$. Since in this section it is convenient to keep an intrinsic notation , the tangent and cotangent norms induced by the tensor $g$ on a tangent or cotangent vector $\eta$ are denoted by $|\eta|$, and the tangent-cotangent duality is simply denoted by $\cdot$. Thus, in this section,  $|\nabla V|^2$ stands for what has been denoted by $a \nabla V \cdot \nabla V$ in the previous section.
Recall from \cite[Definition 3.35, page 143]{afp} that a set $A\subset
\mf M$ has finite perimeter if $\chi_A$, the indicator function of $A$, has bounded variation.
 In such a case the notation $\bs
n_A:=\nabla \, \chi_A$ is used, so that $\bs n_A$ represents the
inward pointing unit normal field of the boundary $\partial A$. The
volume measure on $\mf M$ is denoted by $dx$. If $\mu(dx) =\varrho
(x)\, dx$ for some continuous function $\varrho\colon\mf M\to \bb R$ and if
$A$ has finite perimeter, with some abuse of notation, we denote by
$\mu_A$ the measure $\varrho(x) \, \sigma(dx)$ on $\partial A$. Hence,
for every smooth tangent vector field $\varphi$,
\begin{equation*}
\oint \varphi (x) \cdot \bs n (x) \, \mu_A(dx)
\;=\; \oint \varphi  (x) \cdot \bs n (x) \, \varrho(x)
\, \sigma (d x)\;.
\end{equation*}
\subsection{Generator}
Denote by $\tilde {\mc L}$ the generator given by
\begin{equation*}
\tilde{\mc L}f\;=\; \Delta f \;+\; b \cdot \nabla f \qquad
f\in C^2(\mf M)\;,
\end{equation*}
where $b\colon \mf M \to \bb R^d$ is a smooth vector field. Since $\mf M$
is compact and the coefficients are smooth, there exists a unique Borel probability measure such that $\mu \tilde{\mc L}=0$. Moreover, $\mu(dx) =e^{-V(x)}dx$,
where $V$ is the unique solution to the Hamilton-Jacobi equation
\begin{equation*}
|\nabla V|^2+b\cdot \nabla V=\Delta V+\nabla \cdot b\; .
\end{equation*}
such that $\mu(\mf M)=1$. Since $a$ satisfies condition \eqref{19} and $\varrho (x) = e^{-V(x)}$
is the solution of a linear second-order elliptic equation, by
\cite[Theorem 8.3]{gt}, $V$ is smooth.
The generator $\tilde{\mc L}$ extends to a closed, unbounded operator $\mc L$ on $L_2(\mu)$. It is easy to check that $\mc L$ writes uniquely in the form
\begin{equation}
\label{e:gen2}
\mc Lf= e^V \nabla \cdot \left( e^{-V} \nabla f\right)+ c\cdot \nabla f
\end{equation}
for a suitable vector field $c$, which is also smooth and satisfies $\nabla \cdot (e^{-V}c)=0$. This in turn implies that for any
$A\subset \mf M$, and smooth functions $f$, $g\colon \mf M\to \bb R$,
\begin{equation}
\label{e:cintb}
\begin{aligned}
& \oint c \cdot \bs n \, d\mu_A = \int_{\mf M\setminus A} e^V \nabla
\cdot (e^{-V} c)\, d\mu =0\;, \\
& \quad \int f \, c\cdot \nabla g \, d\mu
\;=\; -\, \int g\,  c\cdot \nabla f \, d\mu\;.
\end{aligned}
\end{equation}
Namely, the operator $c\cdot \nabla$ is skew-adjoint in $L_2(\mu)$.
Denote by $H^1=H^1(\mf M)$ the Hilbert space of weakly differentiable
functions endowed with the scalar product $\<f,g\>_1$ given by
\begin{equation*}
\< f , g\>_1 \;=\; \int  f \, g \, d \mu \;+\;
\int \nabla f \cdot \nabla g \, d \mu \;.
\end{equation*}
Functions in $H^1$ admit a weak trace at the boundary of sets of finite perimeter, see \cite{cf}. By \cite[Theorem~2.1-2.2]{cf}, the usual integration by parts
formulae hold w.r.t.\ to this trace.
Let $A$ and $B$ be disjoint closed subsets of $\mf M$ with a finite
perimeter and let $f,\,g\in H^1$. If $f$ and $g$ are such that
$f_{\restriction_{\partial A}}$, $g_{\restriction_{\partial A}}$,
$f_{\restriction_{\partial B}}$ and $g_{\restriction_{\partial B}}$
are (possibly different) constant, then
\begin{equation}
\label{e:cintg}
\int_{\mf M\setminus {A\cup B}} f \, c\cdot \nabla g \, d\mu \,=\,
-\int_{\mf M\setminus {A\cup B}} g \,c\cdot \nabla f \, d\mu
\end{equation}
since all the boundary terms in the integration by parts vanish in
view of \eqref{e:cintb}. In particular $c \cdot \nabla$ is
skew-adjoint on $L_2(\mu_{\restriction_{A\cup B}})$ when restricted to
$H^1$ functions that take a constant value at the boundary.
On the other hand $e^V \nabla \cdot (e^{-V} \nabla)$ is self-adjoint
in $L_2(\mu)$, so that \eqref{e:gen2} provides a decomposition $\mc
L=\mc L_s+\mc L_a$ in the symmetric and skew part of $\mc L$ in
$L_2(\mu)$. The adjoint $\mc L^*$ of $\mc L$ is then defined as
\begin{equation}
\label{e:adjoint}
\mc L^* f= \mc L_sf-\mc L_a f
=e^V \nabla \cdot \left( e^{-V} \nabla f\right)- c\cdot \nabla f\;.
\end{equation}
\subsection{Stochastic processes}
$\mc L$ and $\mc L^*$ are the generators of a Feller process on
$\mf M$ with invariant measure $\mu$. We denote by $(\bb P_x)$ and
$(\bb P_x^*)$ the induced probability measures on
$C([0,+\infty);\mf M)$.
If $A$ is closed, let $H_A$ be the hitting time of $A$ and for a given
$f\in L_2(\mu)$ consider the function
\begin{equation}
\label{e:ell}
u(x):=\bb E_x\Big[\int_0^{H_A} f(X_t)\,dt\Big]\;.
\end{equation}
If $A$ is the closure of an open set with smooth boundary then $u$ is
the unique $\mc H^1(\mf M\setminus A)$ solution to
\begin{equation}
\label{e:elleq}
\begin{cases}
\mc Lu=f & \text{on $\mf M\setminus A$,} \\
u=0 & \text{on $\partial A$.}
\end{cases}
\end{equation}
Similarly, if $A$ and $B$ are closed, disjoint sets that are the
closure of open sets with smooth boundary, define
\begin{equation*}
h_{A,B}(x)=\bb P_x(H_A<H_B) \;, \quad
h_{A,B}^*(x)=\bb P^*_x(H_A<H_B)\;.
\end{equation*}
Then $h$ and $h^*$ are the unique $\mc H^1$ solutions to
\begin{equation}
\label{e:heq}
\begin{cases}
 \mc Lh=0 & \text{on $\mf M\setminus A\cup B$,}
\\
h=1 & \text{on $A$,}
\\
h=0 & \text{on $B$,}
\end{cases}
\qquad
\begin{cases}
\mc L^* h^*=0 & \text{on $\mf M\setminus A\cup B$,}
\\
h^*=1 & \text{on $A$,}
\\
h^*=0 & \text{on $B$.}
\end{cases}
\end{equation}
\subsection{Capacity}
For $\mc L$, $A$ and $B$ as above, the \emph{capacities} $\Cap(A,B)$,
$\Cap^*(A,B)$ are defined as
\begin{equation}
\label{e:cap}
\Cap(A,B) := \int |\nabla h_{A,B}|^2 d\mu \;, \quad
\Cap^*(A,B) := \int |\nabla h^*_{A,B}|^2 d\mu \;.
\end{equation}
Hereafter we \emph{fix} the sets $A$ and $B$ and denote $h\equiv
h_{A,B}$ and $h^*\equiv h^*_{A,B}$.
\begin{lemma}
\label{e:capeq}
We have that
\begin{equation*}
\Cap(A,B) \;= \; \Cap(B,A) \;=\;
\oint_{\partial A} (\nabla h)\cdot \bs n \,d\mu_A
=\oint_{\partial A} \left(\nabla h + h\, c\right)\cdot \bs n
\,d\mu_A \;.
\end{equation*}
Moreover,
\begin{equation}
\label{ma1}
\Cap(A,B) \;= \;
\int \big\{ \nabla h \cdot \nabla h^*- h^* \,c\cdot \nabla
h\big\} \, d\mu
\, =\, \int \big\{ \nabla h \cdot \nabla h^*+ h \,c\cdot \nabla h^*
\big\} \, d\mu\;,
\end{equation}
and $\Cap(A,B)=\Cap^*(A,B)$.
\end{lemma}
\begin{proof}
$\Cap(B,A)=\Cap(A,B)$ since $h_{B,A}=1-h_{A,B}$ as $A\cap
B=\varnothing$. On the other hand, since $h=\mathbf{1}\{\partial A\}$ on $\partial A
\cup \partial B$, by the
explicit form of the generator $\mc  L$ and by an integration by parts,
\begin{equation}
\label{e:capeq1}
\begin{aligned}
\int |\nabla h|^2 d\mu  \;=\;
& -\, \int h\, e^V\nabla \cdot \left(e^{-V} \nabla h\right)\, d\mu
\,+\, \oint h\, \nabla h\cdot\bs n\, d\mu_{A\cup B}
\\
=\; &  \int h\, (c \cdot \nabla h) \, d\mu
\,+\, \oint \nabla h\cdot\bs n  \,  d\mu_{A} \;.
\end{aligned}
\end{equation}
The first term vanishes in view of the second identity of
\eqref{e:cintb}. This proves the second assertion of the lemma.
The third assertion follows from the first equation in \eqref{e:cintb}
and from the fact that $h$ is constant in $\partial A$.
A similar reasoning to the one  in \eqref{e:capeq1} yields
\begin{equation*}
\int \nabla h \cdot \nabla h^* \, d\mu \;=\;
\int h^*\, c\cdot \nabla h \, d\mu\, \;+\;
\oint  \nabla h\cdot\bs n \, d\mu_{A}
\;=\; \int  h^*\,c\cdot \nabla h \, d\mu \, +\, \Cap(A,B)\;,
\end{equation*}
where the last identity follows from the first part of the proof.  The
previous equation is the first identity in \eqref{ma1}. The second
identity in \eqref{ma1} is obtained from \eqref{e:cintb}. The same
computations inverting the roles of $h$ and $h^*$ gives that
\begin{equation*}
\int \nabla h \cdot \nabla h^* \, d\mu \;=\;
 -\, \int  h\,c\cdot \nabla h^* \, d\mu \, +\, \Cap^*(A,B)\;.
\end{equation*}
In particular, $\Cap(A,B)=\Cap^*(A,B)$, which completes the proof of
the lemma.
\end{proof}
Considering $\mc L^*$ in place of $\mc L$ we obtain from the previous
lemma that
\begin{equation}
\label{ma02}
\Cap^*(A,B) \;=\; \Cap^*(B,A) \;=\;
\oint_{\partial A} (\nabla h^* )\cdot \bs n \,d\mu_A
=\oint_{\partial A} \left(\nabla h^* - h^*\, c\right)\cdot \bs n \,d\mu_A
\end{equation}

\subsection{Equilibrium measure}
Fix $A$ and $B$ as above. Define the probability measure $\nu\equiv
\nu_{A,B}$ as the equilirbium measure on $\partial A \cup \partial B$
conditioned to $\partial A$ as
\begin{equation*}
d\nu:=\frac{- 1}{\Cap(A,B)} \nabla h^* \cdot \bs n\, d\mu_A\;.
\end{equation*}
\begin{proposition}
\label{p:harmcap}
For each $f\in L_2(\mu) $ it holds
\begin{equation*}
\bb E_{\nu}\Big[ \int_0^{H_B}f(X_t)\, dt\Big] \;=\;
\frac{1}{\Cap(A,B)}\int h^*\,f\,d\mu\;.
\end{equation*}
\end{proposition}
\begin{proof}
Take $u$ as in \eqref{e:ell} with $A$ changed to $B$. Since $u$
vanishes on $\partial B$,
\begin{equation}
\label{e:aa1}
\begin{aligned}
\oint u\, \nabla h^* \cdot \bs n\, d\mu_A(x)\, \;=
& \;
\int_{\mf M\setminus A\cup B} e^{V} \nabla \cdot
\left(u\, e^{-V} \nabla h^* \right)\, d\mu
\\
= &\;
\int_{\mf M\setminus A\cup B} \big\{ \nabla h^* \cdot \nabla u
+ u \,e^{V} \nabla \cdot \left(e^{-V} \nabla h^* \right)\big\}
\, d\mu
\\
= & \int_{\mf M\setminus A\cup B} \left( \nabla h^* \cdot
  \nabla u
+ u \, c\cdot \nabla h^* \right)\, d\mu \,
\end{aligned}
\end{equation}
where we used the fact that $\mc L^* h^*=0$ in the last equality.
Since $\nabla h^*$ vanishes on $A$, the quantity in \eqref{e:aa1} also
equals
\begin{equation*}
\int_{\mf M\setminus B}
\left( \nabla h^* \cdot \nabla u + u \, c\cdot \nabla h^* \right)
d\mu \, = \, \int_{\mf M\setminus B}  \big\{ -\, h^*\, e^V
\nabla \cdot \left(e^{-V}  \nabla u\right) + u \, c\cdot \nabla h^* \big\}
d\mu \,
\end{equation*}
where in the last equality we used the fact that $h^*=0$ on $\partial
B$, so that boundary terms vanish in the integration by parts.  Since
$u$ satisfies \eqref{e:elleq} we gather
\begin{equation*}
\oint u\, \nabla h^* \cdot \bs n\, d\mu_A
\;=\; \int_{\mf M\setminus B}
\big\{ -\, h^* \, f+h^*\, c \cdot\nabla u + u \, c\cdot \nabla h^* \big\}
\, d\mu \;.
\end{equation*}
However the last two terms sum up to zero since $h$ and $u$ are
constant on $\partial B$ and \eqref{e:cintg} holds (with
$A=\varnothing$). Thus, since $h^* $ vanishes on $B$,
\begin{equation*}
\oint u\, \nabla h^* \cdot \bs n\, d\mu_A
\;=\; - \int_{\mf M} h^* \, f\, d\mu \;.
\end{equation*}
Therefore, by linearity of the expectation and by the previous equation,
\begin{equation*}
\begin{aligned}
\bb E_{\nu} \Big[ \int_0^{H_B}f(X_t)dt\Big]
\;=\; \oint u \, d\nu\;= &\;
\frac{-1}{\Cap(A,B)} \oint u\, \nabla h^* \cdot \bs n\, d\mu_A
\\
=&\; \frac{1}{\Cap(A,B)}\int h^*\,f\,d\mu\;.
\end{aligned}
\end{equation*}
\end{proof}
\subsection{Variational formulae for the capacity}
In view of Proposition~\ref{p:harmcap}, it may be useful to have
variational formulae for the capacity in order to estimate the
expected value of hitting times.
Let $\mc F \equiv \mc F_{A,B}$ be the Hilbert space of
$L_2(\mu_{\restriction_{\mf M\setminus A\cup B}})$ tangent vector
fields on $\mf M\setminus A\cup B$, and let $\langle \cdot,\cdot
\rangle$ be the associated scalar product:
\begin{equation*}
\langle \varphi, \psi \rangle \;:=\;
\int_{\mf M\setminus {A\cup B}} a^{-1} \varphi \cdot \psi\, d\mu
\;.
\end{equation*}
For $\gamma\in \bb R$ let also $\mc F^{(\gamma)}$ be the closure in $\mc
F$ of the space of smooth tangent vector fields $\varphi\in \mc F$
such that
\begin{equation}
\label{e:fc}
\nabla \cdot (e^{-V}\varphi)=0\;,\qquad
\oint \varphi \cdot \bs n\, d\mu_A\,  \;=\; -\, \gamma\;.
\end{equation}
It is a well-known fact that $\mc F^{(\gamma)}$ is the space of tangent
vector fields such that $\nabla \cdot (e^{-V}\varphi) =0$ weakly, and
that such vector fields admit a weak normal trace $\varphi \cdot \bs
n$ such that \eqref{e:fc} holds (cf. \cite[Theorem~2.2]{cf}).
Let also $\mc H_{\alpha,\beta} \equiv \mc H_{\alpha,\beta,A,B}$ be the
space of $H^1$ functions $f$ on $\mf M\setminus A\cup B$ such that
their normal trace at $A$ and $B$ is constant and equal to $\alpha$
and $\beta$ respectively (these traces exist since we assumed $A$ and
$B$ to have finite perimeter). For $f \in \mc H_{\alpha,\beta}$ define
$\Phi_f:=\nabla f-c\,f$.
\begin{lemma}
\label{l:circ}
If $\varphi\in \mc F^{(\gamma)}$ and $f\in \mc H_{\alpha,0}$ then
\begin{equation*}
\langle \Phi_f-\varphi \,,\, \nabla h\rangle
\;=\; \gamma\;+\; \alpha\,\Cap(A,B)\;.
\end{equation*}
\end{lemma}
\begin{proof}
By definition of $\Phi_f$,
\begin{equation*}
\langle \Phi_f-\varphi,\nabla h\rangle \;=\;
\int_{\mf M\setminus A\cup B} \left(\nabla f
-f\,c-\varphi\right)\cdot \nabla h\, d\mu \;.
\end{equation*}
Integrating by parts, since $f=\alpha$ on $\partial A$ and $f=0$ on
$\partial B$, the previous term becomes
\begin{equation*}
- \int_{\mf M\setminus A\cup B}  \Big\{ f\,e^{V} \nabla \cdot
\left(e^{-V}\nabla h\right) \,+\, [ f\,c + \varphi] \cdot \nabla h
  \Big\} \, d\mu
\;+\; \alpha \oint \nabla h \cdot \bs n\, d\mu_A \;.
\end{equation*}
By Lemma \ref{e:capeq}, the last integral is the capacity between $A$
and $B$, while the expression involving $f$ is equal to $- f\, \mc L
h$. Since $h$ is $\mc L$-harmonic in $\mf M\setminus A\cup B$, by an
integration by part, the previous equation is equal
\begin{equation*}
\int_{\mf M\setminus A\cup B} h \, e^V \nabla \cdot (e^{-V}\varphi) \, d\mu
\;-\; \oint \, \varphi \cdot \bs n \, d\mu_A
\;+\; \alpha\,\Cap(A,B) \;.
\end{equation*}
By \eqref{e:fc}, this expression is equal to $\gamma +
\alpha\,\Cap(A,B)$, as claimed.
\end{proof}
\begin{proposition}[Dirichlet's principle]
\label{p:dirich}
It holds
\begin{equation}
\label{e:dirich}
\Cap(A,B)\;=\; \inf_{f\in \mc H_{1,0}}\inf_{\varphi \in \mc F^{(0)}}
\langle \Phi_f-\varphi,\Phi_f-\varphi\rangle\;,
\end{equation}
and the infimum is attained for $\bar f=(1/2) (h+h^*)$ and $\bar
\varphi=\Phi_{\bar f}-\nabla h$.
\end{proposition}
\begin{proof}
From Lemma~\ref{l:circ} (applied with $\gamma=0$ and $\alpha=1$), for
$f$ and $\varphi$ as in \eqref{e:dirich}, by the Cauchy-Schwarz inequality,
\begin{equation*}
\begin{aligned}
\Cap(A,B)^2 \; & =\; \langle \Phi_f-\varphi,\nabla h\rangle^2
\;\le\; \langle \Phi_f-\varphi,\Phi_f-\varphi\rangle \,
\langle \nabla h,\nabla h\rangle  \\
& = \;
\langle \Phi_f-\varphi,\Phi_f-\varphi\rangle \, \Cap(A,B)
\end{aligned}
\end{equation*}
so that $\Cap(A,B)\le \langle \Phi_f-\varphi,\Phi_f-\varphi\rangle$
for every $f$ and $\varphi$ as in \eqref{e:dirich}.
Since $\Cap(A,B)=\langle \Phi_{\bar f}-\bar \varphi,\Phi_{\bar
  f}-\bar\varphi\rangle$, to complete the proof of the proposition,
one only needs to check that $\bar f\in \mc H_{1,0}$, and $\bar
\varphi\in \mc F^{(0)}$. It is easy to check the first condition, while
the second one follows from the identities
\begin{gather*}
\nabla \cdot (e^{-V} \bar \varphi) \;=\;
(1/2) \, e^{-V}\, \left( \mc L^* h^* -\mc L h\right)\;=\;0\;, \\
\oint \bar \varphi \cdot \bs n\, d\mu_A \;=\;
\frac 12 \,  \oint (\nabla h^* - h^* c) \cdot \bs n\, d\mu_A
\;-\; \frac 12 \,
\oint (\nabla h + h c) \cdot \bs n\, d\mu_A \;.
\end{gather*}
By Lemma \ref{e:capeq} and \eqref{ma02}, the previous expression is
equal to $(1/2) \{\Cap^*(A,B)-\Cap(A,B)\}= 0$, which completes the
proof of the proposition.
\end{proof}
\begin{proposition}[Thomson principle]
\label{p:thomp}
It holds
\begin{equation}
\label{e:thomp}
\Cap(A,B)=\sup_{f\in \mc H_{0,0}}\sup_{\varphi \in \mc F^{(1)}} \,
\frac{1}{\langle \Phi_f-\varphi,\Phi_f-\varphi\rangle}\;\cdot
\end{equation}
Moreover, the supremum is attained at $\bar f=(h-h^*)/2\, \Cap(A,B)$ and
$\bar \varphi=\Phi_{\bar f}- \nabla h/\Cap(A,B)$.
\end{proposition}
\begin{proof}
By Lemma~\ref{l:circ} (applied with $\alpha=0$ and $\gamma=1$) and by
the Cauchy-Schwarz inequality, for $f$ and $\varphi$ as in \eqref{e:thomp} we
have that
\begin{equation*}
1\,=\, \langle \Phi_f-\varphi,\nabla h\rangle^2
\le \langle \Phi_f-\varphi,\Phi_f-\varphi\rangle \,
\langle \nabla h, \nabla h\rangle
\,=\, \langle \Phi_f-\varphi,\Phi_f-\varphi\rangle \,
\Cap(A,B)\, .
\end{equation*}
Since $\langle \Phi_{\bar f}-\bar \varphi,\Phi_{\bar f}-\bar
\varphi\rangle = 1/\Cap(A,B)$, one only need to check that $\bar f\in
\mc H_{0,0}$, and $\bar \varphi \in \mc F^{(1)}$. It is easy to verify the
first condition, while the second follows from
\begin{gather*}
\nabla \cdot (e^{-V} \bar \varphi)
=\frac{-1}{2\, \Cap(A,B)} \, e^{-V} \left( \mc L^* h^* +\mc L h\right)
\;=\;0  \\
\oint \bar \varphi \cdot \bs n\, d\mu_A \;=\;
\frac{-1}{2\, \Cap(A,B)} \, \Big\{
\oint (\nabla h^* - h^* c) \cdot \bs n\, d\mu_A
\;+\; \oint (\nabla h + h c) \cdot \bs n\, d\mu_A \Big\}\;.
\end{gather*}
By Lemma \ref{e:capeq} and \eqref{ma02}, the expression inside braces
is equal to $\Cap^*(A,B)+\Cap(A,B)= 2 \, \Cap(A,B)$, so that $\bar
\varphi \in \mc F^{(1)}$. This completes the proof of the proposition.
\end{proof}

\section{Diffusions in a double-well potential field}
\label{sec3}

In this section, we state the Eyring-Kramers formula for a
non-reversible diffusion in a double-well potential field.

\noindent{\bf Potential field.} Consider a potential $U: \bb R^d
\to \bb R$.  Denote by $H_{\bs x, \bs y}$ the height of the saddle
points between $\bs x$ and $\bs y\in\bb R^d$:
\begin{equation}
\label{2-5}
H_{\bs x, \bs y} \;=\; \inf_\gamma H(\gamma)
\;:=\; \inf_\gamma \, \sup_{\bs z \in \gamma} \, U(\bs z)\;,
\end{equation}
where the infimum is carried over the set $\Gamma_{\bs x, \bs y}$ of
all continuous paths $\gamma :[0,1]\to \bb R^d$ such that $\gamma(0) =
\bs x$, $\gamma(1) = \bs y$. Let $G_{\bs x, \bs y}$ be the smallest
subset of $\{\bs z\in\bb R^d : U(\bs z) = H_{\bs x, \bs y}\}$ with the
property that any path $\gamma \in \Gamma_{\bs x, \bs y}$ such that
$H(\gamma) = H_{\bs x, \bs y}$ contains a point in $G_{\bs x, \bs y}$.
The set $G_{\bs x, \bs y}$ is called the set of gates between $\bs x$
and $\bs y$. In addition, we assume that the potential function $U$ is such that
\begin{itemize}
\item[(P1)] $U\in C^{3}(\bb R^d)$ and
  $\lim_{n\rightarrow\infty} \inf_{\bs x : \Vert \bs x\Vert \ge n} U(\bs{x})=\infty$.

\item[(P2)] The function $U$ has finitely many critical points.  Only
  two of them, denoted by $\bs{m}_{1}$ and $\bs{m}_{2}$, are local
  minima.  The Hessian of $U$ at each of these minima has $d$ strictly
  positive eigenvalues.

\item[(P3)] The set of gates between $\bs{m}_{1}$ and $\bs{m}_{2}$ is
  formed by $\ell \ge1$ saddle points, denoted by $\bs{\sigma}_{1},
  \dots, \bs{\sigma}_{\ell}$. The Hessian of $U$ at each saddle point
  $\bs \sigma_i$ has exactly one strictly negative eigenvalue and
  $(d-1)$ strictly positive eigenvalues.

\item[(P4)] The function $U$ satisfies
\begin{equation}
\label{tight1}
\lim_{\Vert\bs{x}\Vert\rightarrow\infty}
\frac{\bs{x}}{\Vert\bs{x}\Vert}\cdot\nabla U(\bs{x})
\;=\;\lim_{\Vert\bs{x}\Vert\rightarrow\infty} \Big\{\Vert\nabla U(\bs{x})\Vert-
2\Delta U(\bs{x})\Big\} \;=\; \infty\;,
\end{equation}
and
\begin{equation*}
Z_{\epsilon}\;:=\;\int_{\bb R^d}\exp\{-U(\bs{x})/\epsilon\}d\bs{x}\;<\;\infty
\end{equation*}
for all $\epsilon>0$.
\end{itemize}

It is not difficult to show that the conditions \eqref{tight1} imply
that, for all $a\in \bb R$,
\begin{equation}
\label{tight2}
\int_{\bs x:U(\bs x)\;\ge\; a}e^{-U(\bs x)/\epsilon} \, d\bs x \;\le\; C(a)\, e^{-a/\epsilon}
\end{equation}
where the constant $C(a)$ is uniform in $\epsilon\in(0,\,1]$.

\smallskip\noindent{\bf Diffusion model.} Let $\bb{M}$ be a $d\times
d$ (generally non-symmetric) positive-definite matrix:
$\bs{v}\cdot\bb{M}\bs{v}>0$ for all $\bs{v}\neq\bs{0}$. Denote by
$\{X_{t}^{\epsilon}:t\in[0,\infty)\}$, $\epsilon>0$, the diffusion
process associated to the generator $\mc{L}_{\epsilon}$ given by
\begin{equation*}
(\mc{L}_{\epsilon}f) (\bs{x})\;=\;-\nabla U(\bs{x}) \cdot \bb{M}
(\nabla f) (\bs{x}) \,+\,\epsilon\sum_{1\le i,j\le d}
\bb{M}_{ij} (\partial^2_{x_{i},x_{j}} f)(\bs{x})\;.
\end{equation*}
Note that we can rewrite the generator $\mc{L}_{\epsilon}$ as
\begin{equation*}
(\mc{L}_{\epsilon}f) (\bs{x})\;=\;\epsilon\,
e^{U(\bs{x})/\epsilon} \, \nabla\,\cdot\,
\left[e^{-U(\bs{x})/\epsilon} \, \bb{M} (\nabla f)(\bs{x})\right]\;.
\end{equation*}
Thus, this generator is a special form of \eqref{gen1} that we
investigated in the first part of the paper. The additional factor
$\epsilon$ can be regarded merely as the time rescaling of the
process. The probability measure
\begin{equation*}
\mu_{\epsilon}(d\bs{x})\;:=\;Z_{\epsilon}^{-1}\exp\{-U(\bs{x})/\epsilon\}\,
d\bs{x}
\end{equation*}
is the stationary state of the process $X_{t}^{\epsilon}$.

The process $X_{t}^{\epsilon}$ can also be written as the solution of
a stochastic differential equation. As in Section \ref{sec1}, denote
by $\bb{K}$ the symmetric, positive-definite square root of the
symmetric matrix $\bb{S}=(\bb{M}+\bb{M}^{\dagger})/2$: $\bb{S}=\bb K
\bb K$. It is easy to check that $X_{t}^{\epsilon}$ is the solution of
the stochastic differential equation \eqref{1-3}.

Let $\mc A$, $\mc B\subset\bb R^d$ be two open subsets of $\bb R^d$
satisfying the assumptions S, and let $\Omega = (\overline{\mc A}\cup
\overline{\mc B})^c$.  In the present context, the capacity, defined in
\eqref{06}, is given by
\begin{equation}
\label{7-16}
\Cap (\mc A, \mc B) \;=\; \frac{\epsilon}{Z_\epsilon}
\int_{\partial \mc A} \big[ \bb M(\bs x) \, \nabla
h_{\mc A, \mc B}(\bs x) \big] \cdot \bs n_{\Omega}(\bs x)
\, e^{-U(\bs x)/\epsilon} \, \sigma(d\bs x)\;.
\end{equation}

\smallskip\noindent{\bf Structure of valleys.} Let $h_{i} =
U(\bs{m}_{i})$, $i=1,\,2$, and assume without loss of generality that
$h_{1}\ge h_{2}$, so that $\bs{m}_{2}$ is the global minimum of the
potential $U$. Denote by $H$ the height of the saddle points
$\mf{S}:=\{\bs{\sigma}_{1},\,\dots,\,\bs{\sigma}_{\ell}\}$:
\begin{equation*}
H\;:=\;U(\bs{\sigma}_{1})\;=\;\cdots\;=\;U(\bs{\sigma}_{\ell})\; .
\end{equation*}
Let $\Omega$ be the level set defined by saddle points which separate
$\bs m_1$ from $\bs m_2$:
\begin{equation*}
\Omega\;:=\;\{\bs{x}\in\bb R^d:U(\bs{x})<H\}\;.
\end{equation*}
Denote by $\mc{W}_{1}$ and $\mc{W}_{2}$ the two connected components
of $\Omega$ such that $\bs{m}_{i}\in\mc{W}_{i}$, $i=1,\,2$,
respectively. Note that $\overline{\mc{W}}_{1} \cap
\overline{\mc{W}}_{2} =\mf{S}$.

Denote by $\mc{V}_{1}$ and $\mc{V}_{2}$ two metastable sets containing
$\bs{m}_1$ and $\bs{m}_2$, respectively. More precisely, $\mc{V}_i$,
$i=1,\,2$, is a open subset of $\mc{W}_i$ which satisfies assumptions S and
such that
\begin{equation*}
B_\epsilon(\bs{m}_i) \;\subset\; \mc V_i \;\subset\;
\{\bs{x} \in \bb R^d :  U(\bs{x}) < U(\bs{\sigma})- \kappa \}
\end{equation*}
for some $\kappa>0$, where $B_\epsilon(\bs{m}_i)$ represents the ball
of radius $\epsilon$ centered at $\bs{m}_i$: $B_\epsilon(\bs{m}_i) =
\{\bs{x}: |\bs{x}-\bs{m}_i| <\epsilon\}$.

\begin{figure}
  \protect
\includegraphics[scale=0.25]{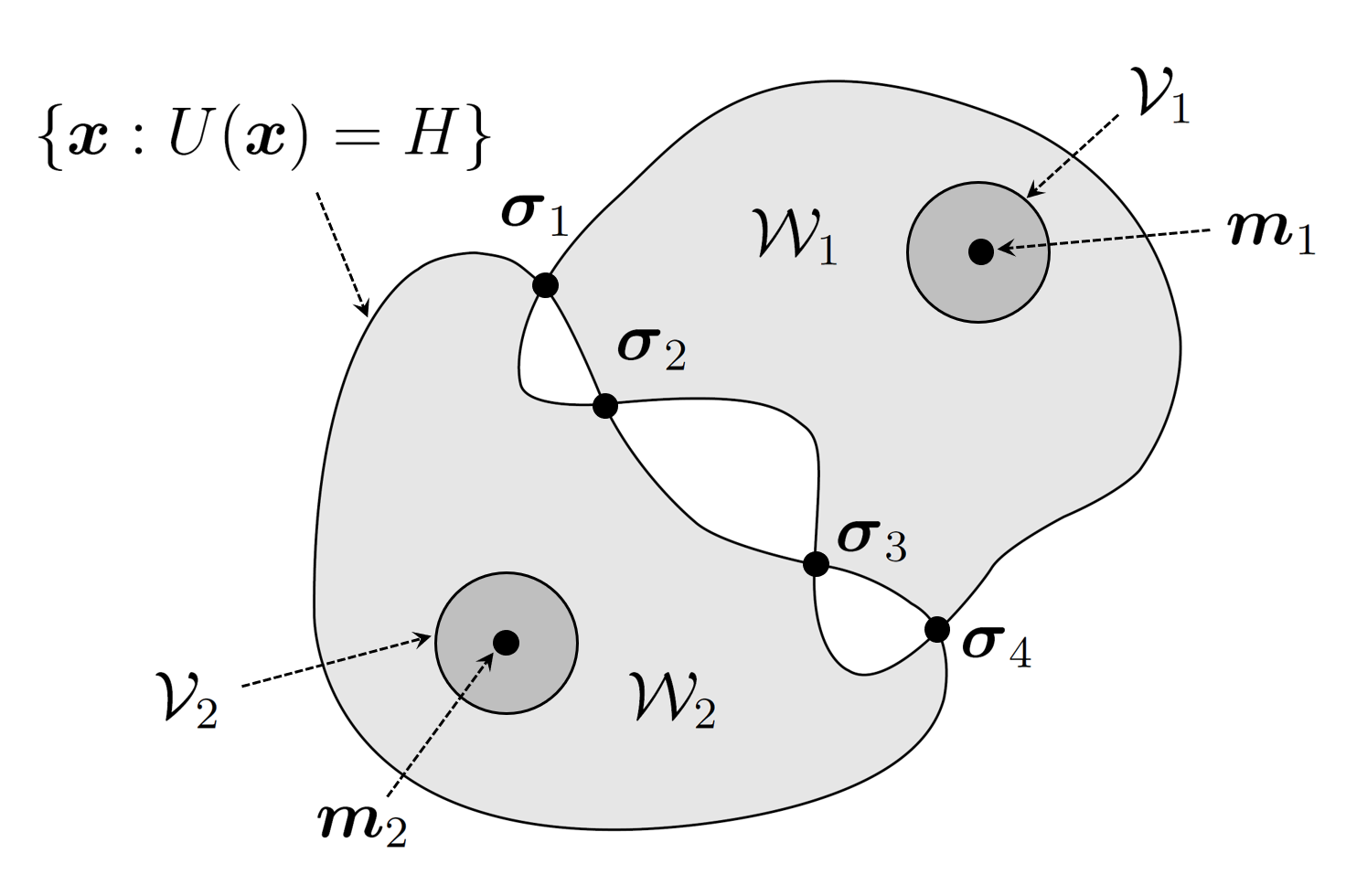}\protect
  \caption{\label{fig1}The structure of metastable wells and valleys.}
\end{figure}

\smallskip\noindent{\bf Metastability results.}  Fix a saddle point
$\bs{\sigma}$ of the potential $U$.  Denote by
$-\lambda_{1}^\bs{\sigma}<0<\lambda_{2}^\bs{\sigma} < \cdots <
\lambda_{d}^\bs{\sigma}$ the eigenvalues of $(\textup{Hess }U)
(\bs{\sigma}):=\bb L^\bs{\sigma}$. By \cite[Lemma A.1]{LS1}, both
$\bb{L}^\bs{\sigma} \bb M$ and $\bb{L}^\bs{\sigma} \bb M^{\dagger}$
have a unique negative eigenvalue. The negative eigenvalues of
$\bb{L}^\bs{\sigma} \bb M$ and $\bb{L}^\bs{\sigma} \bb M^{\dagger}$
coincide because $\bb{L}^\bs{\sigma}
\bb{M}^{\dagger}=\bb{L}^\bs{\sigma}
(\bb{L}^\bs{\sigma}\bb{M})^{\dagger}(\bb{L}^\bs{\sigma})^{-1}$.
Denote by $-\mu^\bs{\sigma}$ this common negative eigenvalue, and let
\begin{equation}
\label{omega}
\omega(\bs{\sigma})\;:=\;\frac{\mu^\bs{\sigma}}
{\sqrt{-\det\left[(\text{Hess }U) \, (\bs{\sigma})\right]}}
\;\;;\;\bs{\sigma}\in\mathfrak{S}.
\end{equation}
We prove in Section \ref{sec4} the following sharp estimate for
capacity between the valleys $\mc{V}_{1}$ and $\mc{V}_{2}$.

\begin{theorem}
\label{thmp1}
We have the following estimate on the capacity.
\begin{equation}
\label{cap}
\textup{cap}(\mc{V}_{1},\,\mc{V}_{2})\;=\;\left[1+o_{\epsilon}(1)\right]\,
\frac 1{Z_{\epsilon}}\, e^{-H/\epsilon}\, \frac{(2\pi \epsilon)^{d/2}}{2\pi}\,
\sum_{i=1}^{\ell}\omega(\bs{\sigma}_{i})\;.
\end{equation}
\end{theorem}

The metastable behavior of $X_{t}^{\epsilon}$ follows from this
result. In Section \ref{sec5}, we derive a sharp estimate for the
transitions time between the two different wells stated below.

Denote by $\bb P_\bs x^\epsilon$, $x\in \bb R^d$, the probability
measure on $C(\bb R_+, \bb R^d)$ induced by the Markov process $X
^{\epsilon}_t$ starting from $\bs x$. Expectation with respect to
$X_t^\epsilon$ is represented by $\bb E_\bs x^\epsilon$. Denote by
$H_\mc C^\epsilon$, $\mc C$ an open subset of $\bb R^d$, the hitting
time of the set $\mc C$:
\begin{equation}
\label{hitt}
H_\mc C^\epsilon \;=\; \inf\{t\ge 0 : X_t^\epsilon \in \mc C\}\;.
\end{equation}
We henceforth omit the superscript $\epsilon$ in these definitions
since there is no risk of confusion.

\begin{theorem}
\label{thmp2}
Under the notations above,
\begin{equation}
\label{2-4}
\bb{E}_{\bs{m}_{1}}^\epsilon\left[H_{\mc{V}_{2}}\right]\;=\;
\left[1+o_{\epsilon}(1)\right]\, \frac{2\pi \,
  e^{(H-h_{1})/\epsilon}}
{\sqrt{\det\left[(\text{\rm Hess }U)\, (\bs{m}_{1})\right]}}\,
\Big(\sum_{i=1}^{\ell}\omega(\bs{\sigma}_{i})\Big)^{-1}\;.
\end{equation}
\end{theorem}

The remaining part of the paper is devoted to provide a detailed proof
of Theorems \ref{thmp1} and \ref{thmp2}. In Section \ref{sec4}, we
prove Theorem \ref{thmp1} by constructing vector fields which
approximate the optimal ones for the Dirichlet's and the Thomson's
principles. The properties of these vector fields are derived in
Section \ref{sec6}, based on general estimates presented in Section
\ref{sec5}. Section \ref{sec8} is dedicated to Theorem \ref{thmp2}.

We close this section with some remarks on the last theorems.

\begin{remark}
\label{rem53}
A careful reading of the proofs reveals that the error term
$o_\epsilon (1)$ appearing in Theorems \ref{thmp1} and \ref{thmp2} are
of order $O(\epsilon^{1/2}(\log \epsilon)^{3/2})$, as in the
reversible case \cite{BEGK1}. It is not clear, however, that this error is sharp.
\end{remark}

\begin{remark}
  Let $\Xi \subset \bb R^d$ be a bounded domain with a boundary in
  $C^{2,\alpha}$ for some $0<\alpha<1$. Assume that the potential has
  no critical points at $\partial \Xi$ and that $\bs n_\Xi \cdot
  \nabla U >0$ at $\partial \Xi$.  A similar result can be proven for
  a diffusion evolving on $\Xi$ with Neumann boundary conditions.
\end{remark}

\begin{remark}
\label{rm11}
In view of \cite{BL1, BL2}, Theorems \ref{thmp1} and \ref{thmp2}
represent the first main step in a complete description of the
metastable behavior of the diffusion $X^\epsilon_t$, which can be
easily foretell. Let $\theta_\epsilon = e^{-(H-h_1)/\epsilon}$, and
denote by $Y^\epsilon_t$ the diffusion $X^\epsilon_{t}$ speeded-up by
$\theta_\epsilon$: $Y^\epsilon_t:= X^\epsilon_{\theta_\epsilon t}$.
Assume that $X^\epsilon_0 = \bs m_1$. As $\epsilon \downarrow 0$, we
expect $Y^\epsilon$ to converge to a two-state Markov chain on $\{\bs
m_1,\bs m_2\}$ which starts from $m_1$. If $U(\bs m_2) < U(\bs m_1)$
the process remains for ever at $m_2$ once it hits this point. In
contrast, if $U(\bs m_2) = U(\bs m_1)$, it jumps back and forth from
$\bs m_2$ to $\bs m_1$.
\end{remark}

\begin{remark}
  The arguments presented in the next sections to prove Theorems
  \ref{thmp1} and \ref{thmp2} apply to the case in which the entries
  of the matrix $\bb M (\bs x)$ belong to $C^2(\bb R^d)$ and satisfy
  conditions \eqref{3-1}, \eqref{19}.
\end{remark}

\section{Proof of Theorem \ref{thmp1}}
\label{sec4}

Throughout this section, to avoid unnecessary technical
considerations, we assume that there is a unique saddle point of
height $H$ between the two valleys around $\bs{m}_{1}$ and
$\bs{m}_{2}$: $\mf{S}=\{\bs{\sigma}\}$. The general case can be
handled without much effort. We refer to \cite{LS1} for the details.

By a translation and change of coordinates we may assume that
$\bs{\sigma}=\bs{0}$ and $(\textup{Hess }U) (\bs{0})=\bb{L}^{\bs 0} =
\textup{diag }(-\lambda_{1}^{\bs 0},\,\lambda_{2}^{\bs
  0},\,\dots,\,\lambda_{d}^{\bs 0})$. We shall drop the superscript
$\bs{0}$ in these notations from now on.  According to these
assumptions, the eigenvectors of $\bb{L}$ are the vectors of the
canonical basis of $\bb R^d$, represented by
$\bs{e}_{1},\,\bs{e}_{2},\,\dots,\,\bs{e}_{d}$. Assume, furthermore,
that $\bs{e}_{1}$ is directed toward $\mc{W}_{1}$, i.e., that there
exists $t_{0}>0$ such that $t\bs{e}_{1}\in\mc{W}_{1}$ for all
$t\in(0,t_{0})$. (cf. Figure \ref{fig2})

\begin{figure}
  \protect
\includegraphics[scale=0.25]{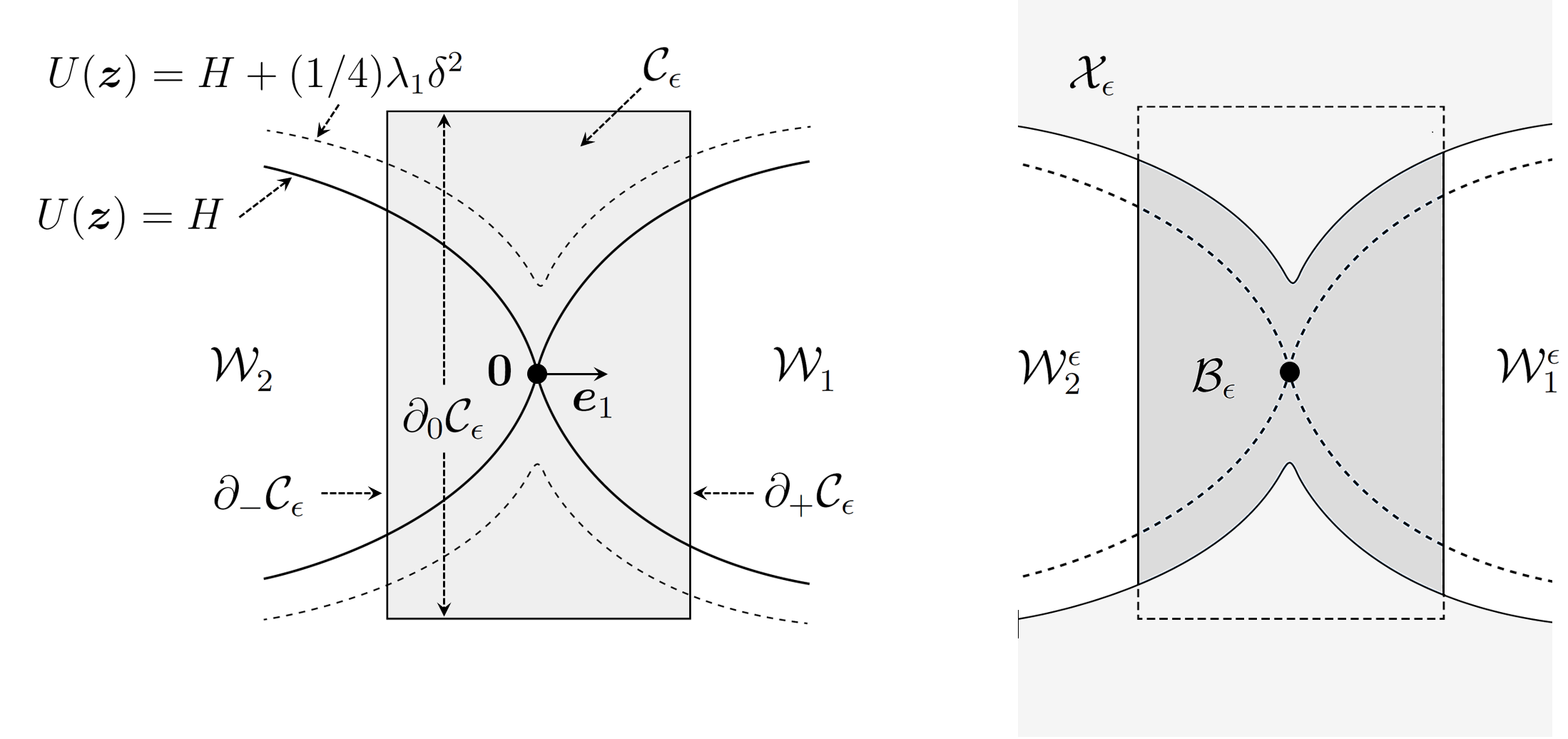}\protect
  \caption{\label{fig2} The neighborhood of the saddle point $\bs{0}$.}
\end{figure}

\smallskip\noindent{\bf A neighborhood of the saddle point.}  We first
introduce a neighborhood of the saddle point, cf. \cite{LS1}.
For a large enough constant $K$ define
\begin{equation}
\label{ins01}
\delta\;:=\; K \, \sqrt{\epsilon\log(1/\epsilon)}\;.
\end{equation}
Let $\mc{C}_{\epsilon}$ be the closed hyperrectangle around the saddle
point $\bs{0}$ defined by
\begin{equation*}
\mc{C}_{\epsilon}\;=\;\left[-\delta,\delta\right]\times
\prod_{i=2}^{d}\Big[-\sqrt{\frac{2\lambda_{1}}{\lambda_{i}}}\,
\delta \,,\,\sqrt{\frac{2\lambda_{1}}{\lambda_{i}}}\, \delta\,
\Big]\;,
\end{equation*}
and denote by $\partial\mc{C}_{\epsilon}$ its boundary. Write
$\bs{z}\in\bb{R}^{d}$ as $\bs{z}=\sum_{i=1}^{d}z_{i}\bs{e}_{i}$, and
define the boundaries $\partial_{-}\mc{C}_{\epsilon}$,
$\partial_{+}\mc{C}_{\epsilon}$, $\partial_{0}\mc{C}_{\epsilon}$ by
\begin{align*}
& \partial_{+}\mc{C}_{\epsilon} \;=\; \{
\bs{z}\in\partial\mc{C}_{\epsilon}:
z_{1}=\delta\} \;,\quad \partial_{-}\mc{C}_{\epsilon}\;=\;
\{ \bs{z}\in\partial\mc{C}_{\epsilon}:z_{1}=-\delta\} \;,\\
& \quad\partial_{0}\mc{C}_{\epsilon}\;=\;\partial\mc{C}_{\epsilon}
\setminus (\partial_{-}\mc{C}_{\epsilon}\cup\partial_{+}\mc{C}_{\epsilon})\;.
\end{align*}

Recall that $U(\bs{0})=H$.

\begin{lemma}
\label{lems0}
For all $\bs{z}\in\partial_{0}\mc{C}_{\epsilon}$, we have that
$U(\bs{z})\ge H+ \left[1+o_{\epsilon}(1)\right]
(1/2)\lambda_{1}\delta^{2}$.
\end{lemma}

\begin{proof}
For $\bs{z}\in\mc{C}_{\epsilon}$, by the Taylor expansion,
\begin{equation*}
U(\bs{z})\;=\;H \;-\, \frac{1}{2} \, \lambda_{1}\, z_{1}^{2}
\;+\; \frac{1}{2} \,  \sum_{j=2}^d \lambda_{j}\, z_{j}^{2}
\;+\; O(\delta^{3})\;.
\end{equation*}
For $\bs{z}\in\partial_{0}\mc{C}_{\epsilon}$, there exists $2\le i\le
d$, such that $z_{i}=\pm\sqrt{2\lambda_{1}/\lambda_{i}}\delta$.
Therefore,
\begin{equation*}
-\, \lambda_{1}\, z_{1}^{2}
\;+\; \sum_{j=2}^d \lambda_{j}\, z_{j}^{2}
\;\ge\; -\, \lambda_{1}\, \delta^{2}
+\lambda_{i} \Big(\sqrt{\frac{2\lambda_{1}}{\lambda_{i}}}\delta\Big)^{2}
\;=\;\lambda_{1}\delta^{2}\;.
\end{equation*}
To complete the proof, it remains to report this estimate to the first
identity.
\end{proof}

Let
\begin{equation*}
\Omega_{\epsilon} \;=\; \big\{\bs{z}\in\bb R^d : U(\bs{z}) <
H+(1/4)\lambda_{1}\delta^{2} \big\}\;, \quad \mc{B}_{\epsilon} =
\mc{C}_{\epsilon} \cap \Omega_{\epsilon}\;.
\end{equation*}
Since the saddle point $\bs \sigma = \bs 0$ is the unique critical
point separating the two local minima $\bs m_1$, $\bs m_2$ of $U$, by
Lemma \ref{lems0}, for $\epsilon$ small enough, there are two
connected components $\mc{W}_{1}^{\epsilon}$ and
$\mc{W}_{2}^{\epsilon}$ of $\Omega_{\epsilon} \setminus
\mc{B}_{\epsilon}$ such that $\bs{m}_{i}\in\mc{W}_{i}^{\epsilon}$,
$i=1,\,2$. Note that $\mc{V}_{i}\subset\mc{W}_{i}^{\epsilon}$,
$i=1,\,2$, for sufficiently small $\epsilon$. Let $\mc{X}_{\epsilon}=
\bb R^d \setminus (\mc{W}_{1}^{\epsilon} \cup \mc{W}_{2}^{\epsilon} \cup
\mc{B}_{\epsilon})$.  These sets are represented in Figure \ref{fig2}.

\smallskip\noindent{\bf Approximations of the equilibrium potentials.}
We introduce in this subsection an approximation of the equilibrium
potentials $h_{\mc{V}_{1},\mc{V}_{2}}$,
$h_{\mc{V}_{1},\mc{V}_{2}}^{*}$. As pointed out in \cite{BEGK1} for
reversible diffusions and in \cite{LS1} for non-reversible Markov
chains, the crucial point consists in defining these approximations in
a mesoscopic neighborhood of the saddle point, denoted above by
$\mc{B}_{\epsilon}$.

Let $-\mu$ be the unique negative eigenvalue of the matrices
$\bb{L} \bb M$, $\bb{L} \bb M^{\dagger}$, and let $\bs{v}$,
$\bs{v}^{*}$ be the associated normal eigenvectors. By Lemma \ref{bl1}
below, the first component of $\bs{v}$, denoted by $v_{1}$, does not
vanish. Assume, without loss of generality, that $v_{1}>0$.
Similarly, assume that $v_{1}^{*}$, the first component of $\bs{v}^*$,
is positive.

Let
\begin{equation}
\label{ins02}
\alpha\;=\;\frac{\mu}{\bs{v}\cdot\bb{M}\bs{v}}\;,
\quad \alpha^{*}\;=\;\frac{\mu}{\bs{v}^{*}\cdot\bb{M}\bs{v}^{*}}\;,
\end{equation}
and let
\begin{equation*}
C_{\epsilon}\;=\; \int_{-\infty}^{\infty}
\exp\left\{ -\frac{\alpha}{2\epsilon}t^{2}\right\} \, dt \;, \quad
C_{\epsilon}^{*} \;=\;\int_{-\infty}^{\infty}
\exp\left\{ -\frac{\alpha^{*}}{2\epsilon}t^{2}\right\} \,dt \;.
\end{equation*}
Of course, $C_{\epsilon} = \sqrt{2\pi\epsilon/\alpha}$,
$C^*_{\epsilon} = \sqrt{2\pi\epsilon/\alpha^*}$.  The constants
$\alpha$ and $\alpha^{*}$ were introduced in \cite{LS1} and they play
a significant role in the estimation of the capacity.
\smallskip

Since $\nabla U(\bs{z})= \bb{L}\bs{z}+O( \Vert \bs z\Vert^2)$, denote
by $\widetilde{\mc{L}}_{\epsilon}$ the approximation of the generator
$\mc{L}_{\epsilon}$ around the origin, namely,
\begin{equation*}
(\widetilde{\mc{L}}_{\epsilon}f) (\bs{z})\;=\;
-(\bb{L}\bs{z})\cdot\bb{M} \, (\nabla f)(\bs{z})
\,+\,\epsilon\sum_{1\le i,j\le d}\bb{M}_{ij}
(\partial^2_{z_{i},z_{j}}f)(\bs{z})\;.
\end{equation*}

Since the equilibrium potential satisfy the boundary conditions
$f\simeq1$ on $\partial \mc{B}_{\epsilon} \cap \partial_{-}
\mc{C}_{\epsilon}$ and $f\simeq 0$ on $\partial \mc{B}_{\epsilon}
\cap \partial_{+} \mc{C}_{\epsilon}$, a natural approximation of the
equilibrium potentials $h_{\mc{V}_{1},\mc{V}_{2}}$,
$h_{\mc{V}_{1},\mc{V}_{2}}^{*}$ in the neighborhood
$\mc{B}_{\epsilon}$ are
\begin{equation}
\label{ins03}
\begin{cases}
p_{\epsilon}(\bs{z})\;=\; (1/C_{\epsilon})\,
\int_{-\infty}^{\bs{z}\cdot\bs{v}}
\exp \{ - (\alpha/2\epsilon)\, t^{2} \} \, dt
& \text{ for } \bs{z}\in\mc{B}_{\epsilon}\;, \\
p_{\epsilon}(\bs{z})\;=\; \bs 1 \{\bs{z}\in\mc{W}_{1}^{\epsilon}\}
& \mbox{ for }\bs{z} \in \mc{B}_{\epsilon}^c\,.
\end{cases}
\end{equation}
\begin{equation*}
\begin{cases}
p^*_{\epsilon}(\bs{z})\;=\; (1/C^*_{\epsilon})\,
\int_{-\infty}^{\bs{z}\cdot\bs{v}^*}
\exp \{ - (\alpha^*/2\epsilon)\, t^{2} \} \, dt
& \text{ for } \bs{z}\in\mc{B}_{\epsilon}\;, \\
p^*_{\epsilon}(\bs{z})\;=\; \bs 1 \{\bs{z}\in\mc{W}_{1}^{\epsilon}\}
& \mbox{ for }\bs{z} \in \mc{B}_{\epsilon}^c\,.
\end{cases}
\end{equation*}

Note that $p_{\epsilon}(\bs{z}) = 1$ on $\mc{W}_{1}^{\epsilon}$ and
that $p_{\epsilon}(\bs{z}) = 0$ on $\mc{W}_{2}^{\epsilon} \cup
\mc{X}_{\epsilon}$, and that $p_{\epsilon}^{*}$ satisfies the same
identities. Moreover, $p_{\epsilon}$ and $p_{\epsilon}^{*}$ are smooth
in the interior of $\mc{B}_{\epsilon}$, $\mc{W}_{1}^{\epsilon}$,
$\mc{W}_{2}^{\epsilon}$ and $\mc{X}_{\epsilon}$, but have jumps along
the boundaries of these domains. These jumps should be removed in
order to use these functions as test functions for the Dirichlet's and Thomson's principles.

To introduce the vector fields $\Theta_{\bs{q}_{\epsilon}}$,
$\Theta_{\bs{q}_{\epsilon}}^{*}$, $\Theta_{\bs{q}_{\epsilon}^{*}}$ and
$\Theta_{\bs{q}_{\epsilon}^{*}}^{*}$ which approximate the vectors
fields $\Phi_{h_{\mc{V}_{1},\mc{V}_{2}}}$,
$\Phi_{h_{\mc{V}_{1},\mc{V}_{2}}}^{*}$,
$\Phi_{h_{\mc{V}_{1},\mc{V}_{2}}^{*}}$ and
$\Phi_{h_{\mc{V}_{1},\mc{V}_{2}}^{*}}^{*}$, respectively, let
\begin{equation*}
\bs{q}_{\epsilon}(\bs{z})\;=\;
\begin{cases}
(\nabla p_{\epsilon}) (\bs{z}) & \mbox{if }\bs{z}\in\mc{B}_{\epsilon}\\
0 & \mbox{otherwise}\;,
\end{cases}
\qquad
\bs{q}_{\epsilon}^{*}(\bs{z})\;=\;
\begin{cases}
(\nabla p_{\epsilon}^{*})(\bs{z}) & \mbox{if }\bs{z}\in\mc{B}_{\epsilon}\\
0 & \mbox{otherwise}\;;
\end{cases}
\end{equation*}
and set
\begin{align*}
&  \Theta_{\bs{q}_{\epsilon}}(\bs{z})\;=\;
\frac{\epsilon}{Z_{\epsilon}}\, e^{-U(\bs{z})/\epsilon} \, \bb{M}^{\dagger}
\bs{q}_{\epsilon}(\bs{z})
\;,\quad \Theta_{\bs{q}_{\epsilon}}^{*}(\bs{z})
\;=\; \frac{\epsilon}{Z_{\epsilon}}\, e^{-U(\bs{z})/\epsilon}\,
\bb{M}\, \bs{q}_{\epsilon}(\bs{z})\;,\\
&  \quad\Theta_{\bs{q}_{\epsilon}^{*}}(\bs{z})\;=\;
\frac{\epsilon}{Z_{\epsilon}} \, e^{-U(\bs{z})/\epsilon} \,
\bb{M}^{\dagger} \, \bs{q}_{\epsilon}^{*}(\bs{z})
\;,\;\;\Theta_{\bs{q}_{\epsilon}^{*}}(\bs{z})
\;=\; \frac{\epsilon}{Z_{\epsilon}}\, e^{-U(\bs{z})/\epsilon}
\, \bb{M} \, \bs{q}_{\epsilon}^{*}(\bs{z})\;.
\end{align*}

One could be tempted to set $\Theta_{\bs{q}_{\epsilon}} =
\Phi_{p_{\epsilon}}$. One has to be cautious, however, because
$p_{\epsilon}$ is discontinuous along $\partial\mc{B}_{\epsilon}$, and
these jumps become significant when applying the divergence theorem.

Let $T_{\epsilon}$ be the time scale given by
\begin{equation}
\label{f03}
T_{\epsilon}\;:=\; \frac 1{Z_{\epsilon}}\,e^{- H/\epsilon}
\, \frac{(2\pi \epsilon)^{d/2}}{2\pi} \;\cdot
\end{equation}
In the presence of a unique saddle point separating two wells,
Theorem \ref{thmp1} becomes

\begin{theorem}
\label{tp}
We have that
\begin{equation}
\label{cap1}
\textup{cap}(\mc{V}_{1},\mc{V}_{2})\;=\;
\left[1+o_{\epsilon}(1)\right]\, T_{\epsilon}\, \omega(\bs{0})\;.
\end{equation}
\end{theorem}

In view of the explicit expression for the minimizers of the
variational problem \eqref{26} in Proposition \ref{prop1}, the
function $f_{\epsilon}= (1/2) (p_{\epsilon}+ p_{\epsilon}^{*})$ and
the vector field $\varphi_{\epsilon}=(1/2)
(\Theta_{\bs{q}_{\epsilon}^{*}} - \Theta_{\bs{q}_{\epsilon}}^{*})$ are
the natural candidates to estimate the capacity \eqref{cap1} through
\eqref{26}. However, $f_{\epsilon}$ does not belong to the set
$\mc{C}_{\mc{V}_{1},\mc{V}_{2}}^{1,0}$, being discontinuous along the
$(d-1)$-dimensional surface $\partial\mc{X}_{\epsilon}
\cup \partial\mc{B}_{\epsilon}$.  To overcome this difficulty, we
convolve $f_{\epsilon}$ with a smooth mollifier $\phi_{\eta}(\cdot) :=
(1/\eta^{d}) \phi(\cdot/\eta)$, where $\phi$ is supported on the
$d$-dimensional unit ball.

Denote by $g^{(\eta)}$ the convolution of a function
$g:\bb R^d\rightarrow\bb{R}$ with the mollifier $\phi_{\eta}$:
$g^{(\eta)}:=g*\phi_{\eta}$.  It follows from the explicit expression
of $p_\epsilon$ and $p_\epsilon^*$ that $p_{\epsilon}^{(\eta)}$ and
$(p_{\epsilon}^*)^{(\eta)}$ belongs to the set
$\mc{C}_{\mc{V}_{1} , \mc{V}_{2}}^{1,0}$ for sufficiently small
$\eta$.

We turn to the test vector field
$\varphi_{\epsilon} = (1/2) (\Theta_{\bs{q}_{\epsilon}^{*}} -
\Theta_{\bs{q}_{\epsilon}}^{*})$.
Note that this flow is discontinuous along
$\partial\mc{B}_{\epsilon}$. As we need a smooth flow to apply the
Dirichlet principle, one might be tempted to continuously extend this
test field all the way to the valleys $\mathcal{V}_1$ and
$\mathcal{V}_2$ along a suitable tube connecting these valleys and
passing through the saddle point. This is indeed the scheme carried
out for the test function in the reversible case considered in
\cite{BEGK1}. The corresponding extension procedure for the test flow
is not as simple, mainly because constructing a smooth divergence-free
field which matches the boundary condition \eqref{gg2} is very
difficult.

We stress that in the discrete case \cite{LS1}, this difficulty is
confronted directly and is solved by a delicate computation. We do not
know yet how to carry out a similar procedure in the continuous
context. Instead of applying Proposition \ref{prop1}, we insert a
discontinuous vector field in the proof of this proposition
and we estimate the error terms coming from the lack of regularity of
the vector field. In particular, the proof below shows that the
regularity conditions imposed on the vector fields in the variational
formulae for the capacity can be overpassed.

\begin{proof}[Proof of Theorem \ref{tp}]
We start with the upper bound. Let $\eta=\epsilon^2$ and let
\begin{equation*}
f_{\epsilon}^{(\eta)}\;=\;\frac{1}{2} \, \Big\{
p_{\epsilon}^{(\eta)} + (p_{\epsilon}^{*})^{(\eta)}\Big\}
\;,\quad \varphi_{\epsilon}\;=\;
\frac{1}{2}\,
(\Theta_{\bs{q}_{\epsilon}^{*}} - \Theta_{\bs{q}_{\epsilon}}^{*})\; .
\end{equation*}
Although $\varphi_{\epsilon}$ does not satisfy the hypotheses of
Proposition \ref{prop1}, inserting $\varphi_{\epsilon}$ in the proof
of Proposition \ref{prop1} provides an upper bound for the capacity.

By the Cauchy-Schwarz inequality,
\begin{equation}
\label{se03}
\big\langle \Phi_{f_{\epsilon}^{(\eta)}} -
\varphi_{\epsilon} \,,\, \Psi_{h_{\mc{V}_{1},\mc{V}_{2}}}
\big\rangle ^{2} \;\le\;
\big\Vert \Phi_{f_{\epsilon}^{(\eta)}} -
\varphi_{\epsilon}\big\Vert ^{2} \,
\Vert \Psi_{h_{\mc{V}_{1},\mc{V}_{2}}} \Vert ^{2}\;.
\end{equation}
Since, for sufficiently small $\epsilon$, $f_{\epsilon}^{(\eta)}$
belongs to $\mc{C}_{\mc{V}_{1},\mc{V}_{2}}^{1,0}$, by the proof of
Proposition \ref{prop1},
\begin{equation*}
\big \langle \Phi_{f_{\epsilon}^{(\eta)}} \,,\,
\Psi_{h_{\mc{V}_{1},\mc{V}_{2}}} \big\rangle
\;=\;\textup{cap}(\mc{V}_{1},\mc{V}_{2})\;.
\end{equation*}
Therefore, by Lemma \ref{tp2},
\begin{equation}
\label{se04}
\big\langle \Phi_{f_{\epsilon}^{(\eta)}} -
\varphi_{\epsilon} \,,\, \Psi_{h_{\mc{V}_{1},\mc{V}_{2}}}
\big\rangle \;=\;\textup{cap}(\mc{V}_{1},\mc{V}_{2})
\;+\; o_{\epsilon}(1)\, T_{\epsilon}\;.
\end{equation}

On the other hand, by \eqref{se01} and by the triangle inequality,
\begin{equation*}
\big\Vert \Phi_{f_{\epsilon}^{(\eta)}}
-\varphi_{\epsilon}\big\Vert \;\le\;
\Big\Vert \frac{\Theta_{\bs{q}_{\epsilon}} +
\Theta_{\bs{q}_{\epsilon}^{*}}}{2} - \varphi_{\epsilon}
\Big\Vert \;+\; o_{\epsilon}(1)\,\sqrt{T_{\epsilon}}\;.
\end{equation*}
Since
\begin{equation*}
\frac{\Theta_{\bs{q}_{\epsilon}}+\Theta_{\bs{q}_{\epsilon}^{*}}}{2}
\;-\; \varphi_{\epsilon} \;=\;
\frac{\Theta_{\bs{q}_{\epsilon}}+\Theta_{\bs{q}_{\epsilon}}^{*}}{2}\;,
\end{equation*}
by the last two displayed equations and by Lemma \ref{tp0},
\begin{equation}
\label{se05}
\big\Vert \Phi_{f_{\epsilon}^{(\eta)}} \,-\,
\varphi_{\epsilon}\big\Vert ^{2}
\;\le\; [1+o_{\epsilon}(1)] \, T_{\epsilon}
\,\omega(\bs{0})\;.
\end{equation}

By \eqref{se03}, \eqref{se04}, \eqref{se05}, and since $\Vert
\Psi_{h_{\mc{V}_{1},\mc{V}_{2}}} \Vert ^{2} =
\textup{cap}(\mc{V}_{1},\mc{V}_{2})$,
\begin{equation*}
\big\{ \textup{cap}(\mc{V}_{1},\mc{V}_{2}) \,+\,
o_{\epsilon}(1) \, T_{\epsilon}\big\}^{2} \;\le\;
[1+o_{\epsilon}(1)] \,T_{\epsilon} \,\omega(\bs{0})
\,\textup{cap}(\mc{V}_{1},\mc{V}_{2})\;.
\end{equation*}
so that,
\begin{equation*}
\textup{cap}(\mc{V}_{1},\mc{V}_{2})
\;\le\; [1+o_{\epsilon}(1)] \,T_{\epsilon}\, \omega(\bs{0})\;.
\end{equation*}
This is the upper bound for the capacity.

In order to obtain the lower bound, we repeat the proof of Proposition
\ref{prop2}. Let
\begin{equation*}
g_{\epsilon}^{(\eta)} \;=\;
\frac{p_{\epsilon}^{(\eta)}-(p_{\epsilon}^{*})^{(\eta)}}
{2\,T_{\epsilon}\,\omega(\bs{0})}\;,
\quad \psi_{\epsilon}\;=\; -\,
\frac{\Theta_{\bs{q}_{\epsilon}}^{*}+\Theta_{\bs{q}_{\epsilon}^{*}}}
{2\,T_{\epsilon}\,\omega(\bs{0})}\;\cdot
\end{equation*}
By the Cauchy-Schwarz inequality,
\begin{equation}
\label{se06}
\big\langle \Phi_{g_{\epsilon}^{(\eta)}} -\psi_{\epsilon} \,,\,
\Psi_{h_{\mc{V}_{1},\mc{V}_{2}}}\big\rangle ^{2}
\;\le\; \big\Vert \Phi_{g_{\epsilon}^{(\eta)}}
-\psi_{\epsilon} \big\Vert ^{2} \,
\Vert \Psi_{h_{\mc{V}_{1},\mc{V}_{2}}} \Vert ^{2}\;.
\end{equation}
Since $g_{\epsilon}^{(\eta)}\in\mc{C}_{\mc{V}_{1},\mc{V}_{2}}^{0,0}$ for
sufficiently small $\epsilon$, as in the proof of Proposition
\ref{prop2}, we obtain that
\begin{equation*}
\big\langle \Phi_{g_{\epsilon}^{(\eta)}} \,,\,
\Psi_{h_{\mc{V}_{1},\mc{V}_{2}}} \big\rangle \;=\;0\;.
\end{equation*}
On the other hand, by Lemma \ref{tp2},
\begin{equation*}
\langle \psi_{\epsilon} \,,\,\Psi_{h_{\mc{V}_{1},\mc{V}_{2}}}\rangle
\;=\; 1\,+\,o_{\epsilon}(1)\;.
\end{equation*}
In particular, the left-hand side of \eqref{se06} is equal to
$1+o_{\epsilon}(1)$.

Consider the first term on the right-hand side of \eqref{se06}.
By \eqref{se01} and by the triangle inequality,
\begin{equation}
\label{se07}
\big\Vert \Phi_{g_{\epsilon}^{(\eta)}} -
\psi_{\epsilon} \big\Vert \;\le\;
\Big\Vert \frac{\Theta_{\bs{q}_{\epsilon}} -
  \Theta_{\bs{q}_{\epsilon}^{*}}} {2\, T_{\epsilon}\, \omega(\bs{0})}
- \psi_{\epsilon} \Big\Vert \;+\;
\frac{o_{\epsilon}(1)}{\sqrt{T_{\epsilon}}}\;\cdot
\end{equation}
Since
\begin{equation*}
\frac{\Theta_{\bs{q}_{\epsilon}} - \Theta_{\bs{q}_{\epsilon}^{*}}}
{2\, T_{\epsilon}\, \omega(\bs{0})} \;-\; \psi_{\epsilon}\;=\;
\frac{\Theta_{\bs{q}_{\epsilon}}+\Theta_{\bs{q}_{\epsilon}}^{*}}
{2\, T_{\epsilon}\, \omega(\bs{0})}\;,
\end{equation*}
by Lemma \ref{tp0}, the right-hand side of \eqref{se07} is less than
or equal to
$$[1+o_{\epsilon}(1)] \{T_{\epsilon}\,
\omega(\bs{0})\}^{-1/2}\;.$$
Putting together the previous estimates,
since $\Vert \Psi_{h_{\mc{V}_{1},\mc{V}_{2}}} \Vert ^{2} =
\textup{cap}(\mc{V}_{1},\mc{V}_{2})$, we obtain from \eqref{se06} that
\begin{equation*}
[1+o_{\epsilon}(1)]^{2} \;\le\;
[1+o_{\epsilon}(1)] \, \frac 1{T_{\epsilon}\, \omega(\bs{0})}
\, \textup{cap }(\mc{V}_{1},\mc{V}_{2})\;.
\end{equation*}
This completes the proof of lower bound.
\end{proof}

We conclude this section with three lemmata, whose proofs are
postponed to Section \ref{sec6}.

\begin{lemma}
\label{tp0}
We have that
\begin{equation*}
\Big \Vert \frac{\Theta_{\bs{q}_{\epsilon}}+\Theta_{\bs{q}_{\epsilon}}^{*}}{2}
\Big\Vert ^{2}\;=\;\left[1+o_{\epsilon}(1)\right]\,T_{\epsilon}\,\omega(\bs{0})\;.
\end{equation*}
\end{lemma}

Recall from \eqref{ins01} that $\delta= K \sqrt{\epsilon
  \log(1/\epsilon)}$, where $K$ is an arbitrary positive number.

\begin{lemma}
\label{tp1}
Assume that $\eta\ll \delta$, in the sense that $\lim_{\epsilon \to 0}
\eta(\epsilon)/\delta(\epsilon) =0$.  There exist positive constants
$C_{1},\,C_{2},\,C_{3}$ and $C_{4}$, which do not depends on
$\epsilon$ and $\eta$, such that
\begin{align*}
& \big\Vert \Phi_{p_{\epsilon}^{(\eta)}} -
\Theta_{\bs{q}_{\epsilon}}\big\Vert ^{2} \\
&\qquad \;\le\; \frac {C_{1} }{Z_{\epsilon}}\, e^{- H/\epsilon}\,
\Big\{ \frac{\epsilon^{C_{2}K^{2}}}{\eta^{d}}\,
e^{C_{3} \eta/ \epsilon} \,+\, o_{\epsilon}(1)\,
\epsilon^{d/2} \Big[ \Big(\frac{\eta}{\epsilon}\Big)^{2}
\,+\, \frac{\eta}{\epsilon} \Big]
\Big(1+e^{C_{4}\eta\delta/\epsilon}\Big)\Big\}\;.
\end{align*}
A similar estimate holds for $\Phi_{p_{\epsilon}^{(\eta)}}^{*}$,
$\Phi_{(p_{\epsilon}^{*})^{(\eta)}}$ and
$\Phi_{(p_{\epsilon}^{*})^{(\eta)}}^{*}$.
\end{lemma}

Since $p_{\epsilon}$ is discontinuous along
$\partial\mc{X}_{\epsilon}\cup\partial\mc{B}_{\epsilon}$, the function
$p_{\epsilon}^{(\eta)}$ has a bump around this boundary.  The first
term in the bracket takes this into account. Taking $\eta=\epsilon^2$
and $K$ a large enough real number, it follows from Lemma \ref{tp1}
that
\begin{equation}
\label{se01}
\big\Vert  \Phi_{p_{\epsilon}^{(\epsilon^2)}} -
\Theta_{\bs{q}_{\epsilon}}\big\Vert^{2}
\;=\;o_{\epsilon}(1)\,T_{\epsilon}\;.
\end{equation}
We could have chosen $\eta=\epsilon$ to complete the proof of
  Theorem \ref{thmp1}. We selected $\eta=\epsilon^2$ to obtain
  an error of order $O(\epsilon^{1/2}(\log \epsilon)^{3/2})$, as
  stated in Remark \ref{rem53}.

\begin{lemma}
\label{tp2}
We have that
\begin{equation}
\label{se02}
\big\langle \Theta_{\bs{q}_{\epsilon}}^{*},
\Psi_{h_{\mc{V}_{1},\mc{V}_{2}}} \big\rangle
\;=\;-\, [1+o_{\epsilon}(1)] \, T_{\epsilon}\,
\omega(\bs{0})\;.
\end{equation}
The same estimate holds for $\Theta_{\bs{q}_{\epsilon}^{*}}$.
\end{lemma}

\section{The equilibrium potential}
\label{sec5}

The main result of this section establishes a pointwise bound on the equilibrium potential between two open sets. The proofs roughly go along the lines of \cite{BEGK1} in the reversible case. We provide detailed proofs for sake of completeness.

We start recalling some classical estimates on the solutions of
elliptic equations.  Fix $0<\alpha<1$. Unless otherwise stated,
throughout this section $\Omega\subset \bb R^d$ is a domain with
boundary in $C^{2,\alpha}$, $\mf g$ a function in $L^2(\Omega) \cap
C^{\alpha}(\overline{\Omega})$ and $\mf b$ a function in
$W^{1,2}(\Omega)\cap C^{2,\alpha}(\overline{\Omega})$, where the
reference measure is $Z^{-1}_\epsilon \exp\{-U(\bs{x})/\epsilon\}
d\bs{x}$. We examine the Dirichlet problem \eqref{22} with $\mc L$
replaced by $\mc L_\epsilon$.

\smallskip\noindent{\bf Harnack and H\"older estimates.}  In this
subsection, $\Omega\subset\bb R^d$ represents a {\sl bounded} domain and
$W^{2,p} (\Omega)$, $p\ge 1$, the space of twice weakly differentiable
functions whose derivatives of order $n\le 2$ are in $L^p(\Omega)$.

Since $\bb M$ is a positive-definite matrix, there exist $0<\lambda <
\Lambda$ such that
\begin{equation}
\label{7-5}
\lambda\, \Vert \bs x \Vert^2 \;\le\; \bs x \cdot \bb M \, \bs x
\;\le\; \Lambda\, \Vert \bs x \Vert^2 \quad
\text{for all $\bs x\in \bb R^d$.}
\end{equation}
Clearly, $\gamma=\Lambda/\lambda<\infty$.  For a domain $\Omega
\subset \bb R^d$, let
\begin{equation}
\label{7-1}
\nu_{\Omega} \;=\; \frac 1{\epsilon^2} \,
\frac{\Vert \bb M \Vert^2_2} {\lambda^2} \,
\sup_{\bs{x}\in \Omega} \Vert (\nabla U)(\bs x)\Vert^2 \;,
\end{equation}
where $\Vert (\nabla U)(\bs x)\Vert^2 = \sum_j (\partial_{x_j} U)(\bs
x)^2$, $\Vert \bb M \Vert^2_2 = \sum_{j,k} \bb M_{j,k}^2$.

The Harnack inequality presented in the next result is \cite[Corollary
9.25]{gt}. Denote by $B_{r}(\bs x)$ the open ball of radius $r>0$
centered at $\bs x\in\bb R^d$.

\begin{lemma}
\label{7-l4}
Let $u\in W^{2,d} (\Omega)$ be a non-negative function which satisfies
the equation $\mc L_\epsilon u = 0$ in $\Omega$. Suppose that
$B_{2R}(\bs x) \subset \Omega$ for some $R>0$, $\bs x\in\Omega$.
Then, there exists a constant $C_0 = C_0(d, \gamma, \nu_{\Omega}
R^2)<\infty$ such that
\begin{equation*}
\sup_{\bs x\in B_R(\bs x)} u(\bs x) \;\le\; C_0
\, \inf_{\bs x\in B_R(\bs x)} u(\bs x) \;.
\end{equation*}
\end{lemma}

Denote by $\text{\rm osc}(u,\mc A)$ the oscillation of a function
$u:\mc A \to \bb R$ in the set $\mc A$: $\text{\rm osc}(u,\mc A) =
\sup_{\bs x \in \mc A} u(\bs x) - \inf_{\bs x \in \mc A} u(\bs x)$.
The H\"older estimate stated below is \cite[Corollary 9.24]{gt}.

\begin{lemma}
\label{7-l2}
Let $u\in W^{2,d} (\Omega)$ satisfy the equation $\mc L_\epsilon u =
f$ in $\Omega$ for some $f\in L^d(\Omega)$.  Suppose that $B_{R_0}(\bs
x) \subset \Omega$ for some $R_0>0$, $\bs x\in\Omega$.  Then, there
exist constants $C_0 = C_0(d, \gamma, \nu_{\Omega} R^2_0)<\infty$,
$\alpha = \alpha(d, \gamma, \nu_{\Omega} R^2_0)>0$ such that for all
$R\le R_0$,
\begin{equation*}
\text{\rm osc} (u,B_R(\bs x)) \;\le\; C_0 \Big(\frac{R}{R_0}\Big)^\alpha
\, \Big( \text{\rm osc} (u,B_{R_0}(\bs x)) + R_0 \, \Vert f\Vert_{d,
  B_{R_0}(\bs x)} \Big)\;,
\end{equation*}
where $\Vert f\Vert_{d, B_{R_0}(\bs x)}$ stands for the
$L^d(B_{R_0}(\bs x))$ norm of $f$.
\end{lemma}

\smallskip\noindent{\bf The Green function.}  We present in this
subsection several properties of the Green function associated to the
boundary-value problem \eqref{22}. We do not assume
$\Omega\subset\bb R^d$ to be bounded.

By the assumptions (P4) in Section \ref{sec3} and by
\cite[Theorems 6.1.3, 4.2.1 (ii), 4.2.5]{p95}, the generator $\mathcal{L}_{\epsilon}$ possesses a
non-negative Green function, denoted by $G_{\Omega}: \Omega \times
\Omega \to \bb R_+$, such that for each $\bs y\in\Omega$,
\begin{equation}
\label{7-13}
\text{$G_{\Omega}(\cdot,\,\bs{y}) \in
C^{2,\alpha}(\Omega\setminus \{\bs y\})$,\quad $\mathcal{L}_{\epsilon}
G_{\Omega}(\cdot,\,\bs{y}) = 0$ on $\Omega\setminus \{\bs y\}$.}
\end{equation}

The solutions of the boundary-value problem \eqref{22} can be
represented in terms of the Green function.  Next result follows from
hypothesis (P4), which guarantees that the process is positive
recurrent, and Theorems 3.6.4 and 4.3.7 in \cite{p95} with
$\lambda=0$.

\begin{lemma}
\label{7-l14}
Assume that $\Omega$ has a $C^{2,\alpha}$-boundary for some
$0<\alpha<1$. Then, for any function $g$ in
$C^\alpha(\overline{\Omega}) \cap L^2(\Omega)$ which vanishes at
$\partial \Omega$, the function
\begin{equation*}
f(\bs{x})\;=\;\int_{\Omega} G_{\Omega}(\bs{x},\,\bs{y}) \,
g(\bs{y}) \, d\bs{y}
\end{equation*}
belongs to $C^{2,\alpha}(\overline{\Omega})$ and is the unique
solution of the problem \eqref{22} with $\mf g= g$, $\mf b=0$.
\end{lemma}

The previous result asserts that the Green function, as an operator, is
the inverse of $-\mc L_\epsilon$. In particular, it inherits the dual
properties of the generator. More precisely, if we denote by
$G_{\Omega}^{*}$ the Green function of the adjoint generator
$\mathcal{L}_{\epsilon}^{*}$, it follows from the previous lemma that
$G_{\Omega}^{*}$ is the adjoint of $G_{\Omega}$ in $L^2(\mu_\epsilon)$
so that
\begin{equation}
\label{e62}
e^{-U(\bs{x})/\epsilon} \, G_{\Omega}(\bs{x},\,\bs{y})
\;=\;
e^{-U(\bs{y})/\epsilon} \,G_{\Omega}^{*}(\bs{y},\,\bs{x})\;,
\quad \bs x\; \not = \; \bs y\in\overline{\Omega}\;.
\end{equation}

By Lemma \ref{7-l14},
\begin{equation}
\label{7-11}
G_{\Omega}(\bs{x},\,\bs{y})\;=\; 0 \quad \text{for all $\bs
  x\in\partial\Omega$, $\bs y\in\Omega$.}
\end{equation}
On the other hand, by \eqref{e62} and \eqref{7-11} for $G_{\Omega}^*$ in place
of $G_{\Omega}$,
\begin{equation}
\label{7-12}
G_{\Omega}(\bs{x},\,\bs{y})\;=\; 0 \quad \text{for all $\bs
  x\in\Omega$, $\bs y\in\partial\Omega$.}
\end{equation}
Of course, all previous properties are in force for the adjoint Green
function $G_{\Omega}^*$.

Next result is Theorem 4.2.8 in \cite{p95}.
\begin{lemma}
\label{7-l5}
For each compact set $\mc K \subset \Omega$, there exist constants
$0<c_1<c_2<\infty$ and $r_0 \in (0,1)$ such that for each $\bs x\in
\mc K$,
\begin{equation*}
c_1\, \Vert \bs y-\bs x\Vert^{2-d} \;\le\; G_{\Omega}(\bs x, \bs y)
\;\le\; c_2\, \Vert \bs y-\bs x\Vert^{2-d}
\end{equation*}
for all $\Vert \bs y-\bs x\Vert < r_0$ if $d\ge 3$, and
\begin{equation*}
- c_1\, \log \Vert \bs y-\bs x\Vert \;\le\; G_{\Omega}(\bs x, \bs y)
\;\le\; - c_2\, \log \Vert \bs y-\bs x\Vert
\end{equation*}
for all $\Vert \bs y-\bs x\Vert < r_0$ if $d=2$.
\end{lemma}

By \eqref{7-13}, the function $G_\Omega(\cdot, \bs x)$ is harmonic on
$\Omega\setminus\{\bs x\}$, and, by Lemma \ref{7-l5}, it diverges at $\bs
x$.  The next lemma turns rigorous the formal identity $[\mc
L_\epsilon G_\Omega(\cdot, \bs x)](\bs y) = -\, \delta_{\bs x}(\bs
y)$, where $\delta_{\bs x}$ is the Dirac delta function at $\bs x$.

\begin{lemma}
\label{7-l13}
Assume that $\Omega$ has a $C^{2,\alpha}$-boundary for some
$0<\alpha<1$, and let $f$ be a function in
$C^{2,\alpha}(\overline{\Omega})$. Then, for all $\bs x\in \Omega$,
\begin{equation*}
f(\bs x) \;=\; \lim_{\delta\to 0} \epsilon\, \int_{\partial B_\delta(\bs x)}
e^{[U(\bs x) - U(\bs y)]/\epsilon} f(\bs y)\, \bb M^\dagger (\nabla G_\Omega^*)(\bs
y,\bs{x}), \cdot \bs n_{B_\delta(\bs x)^c} (\bs{y})\, \sigma(d\bs y)\;.
\end{equation*}
\end{lemma}

\begin{proof}
Fix $\bs x\in \Omega$ and modify $f$ outside a neighborhood of $\bs x$
for $f$ to vanish at $\partial \Omega$. We first claim that
\begin{equation*}
f(\bs x) \;=\; -\, \int_{\Omega} e^{[U(\bs x) - U(\bs y)]/\epsilon}
\, G_\Omega^*(\bs y, \bs x) \, (\mc L_\epsilon f) (\bs y)\, d\bs y\;.
\end{equation*}
To prove this identity, denote by $h$ the function defined by integral
on the right-hand side.  Applying \eqref{e62}, we may replace $\exp\{
[U(\bs x) - U(\bs y)]/\epsilon\} \, G_\Omega^*(\bs y, \bs x)$ by $G_\Omega(\bs x,
\bs y)$. By assumption, $\mc L_\epsilon f$ belongs to
$C^\alpha(\overline{\Omega})$ and vanishes at $\partial
\Omega$. Therefore, by Lemma \ref{7-l14}, $h$ is the unique solution
of \eqref{22} with $\mf g = \mc L_\epsilon f$ and $\mf b=0$. Since
$-f$ solves the same equation, by uniqueness, $h=-f$, proving the
identity.

By Lemma \ref{7-l5}, the integral on the right-hand side of the
previous displayed equation is equal to
\begin{equation*}
\lim_{\delta\to 0} \int_{\Omega\setminus B_\delta(\bs x)} e^{[U(\bs x) - U(\bs y)]/\epsilon}
\, G_\Omega^*(\bs y, \bs x) \, (\mc L_\epsilon f) (\bs y)\, d\bs y\;.
\end{equation*}
By the divergence theorem and since $G_\Omega^*(\cdot, \bs x)$ vanishes at
$\partial \Omega$, the previous integral is equal to
\begin{equation}
\label{7-14}
\begin{aligned}
& \epsilon \, \int_{\partial B_\delta(\bs x)} e^{[U(\bs x) - U(\bs y)]/\epsilon}
\, G_\Omega^*(\bs y, \bs x) \, \bb M\, \nabla f (\bs y) \cdot
\bs n_{B_\delta(\bs x)^c}(\bs y) \, \sigma(d\bs y) \\
&\quad - \,
\epsilon\, \int_{\Omega\setminus B_\delta(\bs x)} e^{[U(\bs x) - U(\bs y)]/\epsilon}
\bb M^\dagger \, \nabla \, G_\Omega^*(\bs y, \bs x) \, \nabla f (\bs y)\, d\bs y\;.
\end{aligned}
\end{equation}
By Lemma \ref{7-l5}, the first integral vanishes as $\delta\to 0$. By
the divergence theorem, since $f$ vanishes on $\partial \Omega$ and
since $[\mc L^*_\epsilon G_\Omega^*(\cdot, \bs x)] (\bs y) = 0$ on
$\Omega\setminus\{\bs x\}$, the second one is equal to
\begin{equation*}
-\, \epsilon \, \int_{\partial B_\delta(\bs x)} e^{[U(\bs x) - U(\bs y)]/\epsilon}
\, f (\bs y) \, \bb M^\dagger \, \nabla  G_\Omega^*(\bs y, \bs x)
\cdot \bs n_{B_\delta(\bs x)^c}(\bs y) \, \sigma(d\bs y) \;.
\end{equation*}
This completes the proof of the lemma.
\end{proof}

\begin{lemma}
\label{7-l7}
Assume that $\Omega$ has a $C^{2,\alpha}$-boundary for some
$0<\alpha<1$, and let $b$ be a function in
$C^{2,\alpha}(\overline{\Omega}) \cap W^{1,2}(\Omega)$. The unique
solution in $C^{2,\alpha}(\overline{\Omega})$ of the Dirichlet problem
\eqref{22} with $\mf g=0$, $\mf b = b$, denoted by $f$, can be
represented as
\begin{equation*}
f(\bs x) \;=\; -\, \epsilon\, \int_{\partial \Omega} b(\bs y)\, e^{[U(\bs x)
  - U(\bs y)]/\epsilon} \, \bb M^\dagger \nabla G_{\Omega}^* (\bs y, \bs x)
\cdot \bs n_{\Omega }(\bs{y}) \, \sigma(d\bs y)\;, \quad \bs x\in\Omega\;.
\end{equation*}
\end{lemma}

\begin{proof}
Fix $\bs x$ in $\Omega$, and let $h(\bs y) = G^*_{\Omega}(\bs y, \bs
x)$. By \eqref{7-13} and \eqref{7-11}, $h\in
C^{2,\alpha}(\Omega\setminus \{\bs x\})$, $(\mc L^*_\epsilon h)(\bs
y)= 0$, $\bs y\in\Omega\setminus\{\bs x\}$, and $h(\bs y)= 0$, $\bs
y\in\partial \Omega$.

Denote by $f$ the solution of the Dirichlet problem \eqref{22} with
$\mf g=0$, $\mf b = b$. Fix $\delta>0$. Since $\mc L_\epsilon f=0$ and
since $h$ vanishes on $\partial \Omega$, by the divergence theorem,
\begin{align*}
0\; &=\; \int_{\Omega\setminus B_\delta(\bs x)} h(\bs y) \,
\nabla\cdot \big( e^{-U/\epsilon}
\bb M \nabla f  \big) (\bs y) \, d\bs y \\
\; & =\; -\, \int_{\Omega\setminus B_\delta(\bs x)}
e^{-U(\bs y)/\epsilon} \, \bb M^\dagger \,
\nabla h(\bs y) \, \cdot \nabla f (\bs y) \, d\bs y
\;+\; o_\delta(1)\;.
\end{align*}
The expression $o_\delta(1)$ comes from the integral on $\partial
B_\delta(\bs x)$, which vanishes as $\delta\to 0$ as we have seen in
\eqref{7-14}.

Applying the divergence theorem once more, since $\mc L^*_\epsilon
h=0$ on $\Omega\setminus B_\delta(\bs x)$, the right-hand side of the
previous identity is equal to
\begin{align*}
& -\; \int_{\partial B_\delta(\bs x)}  f (\bs y) \, e^{-U(\bs y)/\epsilon} \bb M^\dagger \,
\nabla h(\bs y) \cdot \bs n_{B_\delta(\bs x)^c}(\bs y) \, \sigma(d \bs y) \\
&\quad -\; \int_{\partial \Omega}  f (\bs y) \, e^{-U(\bs y)/\epsilon} \bb M^\dagger \,
\nabla h(\bs y) \cdot \bs n_{\Omega}(\bs y) \, \sigma(d \bs y)
\;+\; o_\delta(1)\;.
\end{align*}
As $f$ belongs to $C^{2,\alpha}(\overline{\Omega})$ and is equal to
$b$ on $\partial \Omega$, by Lemma
\ref{7-l13}, letting $\delta \to 0$, this sum converges to
\begin{equation*}
-\, \epsilon^{-1}\, f (\bs x)\, e^{-U(\bs x)/\epsilon}
\;-\; \int_{\partial \Omega}  b (\bs y) \, e^{-U(\bs y)/\epsilon} \bb M^\dagger \,
\nabla h(\bs y) \cdot \bs n_{\Omega}(\bs y) \, \sigma(d \bs y) \;.
\end{equation*}
This completes the proof of the lemma.
\end{proof}

\smallskip\noindent{\bf The equilibrium potential.} In this
subsection, we establish a bound on the harmonic function in terms of
capacities and simple bounds for the capacity between two
sets. Together these estimate provide a useful bound on the harmonic
function.

Let $\mc A$, $\mc B\subset\bb R^d$ be two open subsets of $\bb R^d$
satisfying Assumption S, and let $\Omega = (\overline{\mc A}\cup
\overline{\mc B})^c$.  The next result presents a formula for the
equilibrium potential. The same proof provides an identity for
$h^*_{\mc A,\mc B}$ in place of $h_{\mc A,\mc B}$.

\begin{lemma}
\label{7-l8}
Let $\mc A$ and $\mc B$ be open sets satisfying Assumption S. Then,
for all $\bs x\not\in\mc B$,
\begin{equation*}
h_{\mc A,\mc B} (\bs x) \;=\; \epsilon\, \int_{\partial \mc A}
G_{\mc B^c} (\bs x, \bs y)\, \bb M \, \nabla h_{\mc A,\mc B} (\bs y)
\cdot \bs n_{\mc A^c}(\bs y) \, \sigma(d \bs y)\;.
\end{equation*}
\end{lemma}

\begin{proof}
Consider the integral on the right-hand side.
Since $G_{\mc B^c} (\bs x, \cdot )$ vanishes at $\partial \mc B$, we
may extend the integral to $\partial \mc A \cup \partial \mc B$.
By \eqref{e62}, we may replace $G_{\mc B^c} (\bs x, \bs y)$ by
$\exp\{[U(\bs x) - U(\bs y)]/\epsilon\} \, G^*_{\mc B^c} (\bs y, \bs x)$. On
the other hand, as $h_{\mc A, \mc B} = 1 - h_{\mc B, \mc A}$, we may
also replace $\nabla h_{\mc A, \mc B}$ by $- \nabla h_{\mc B, \mc A}$.
After these modifications, the integral appearing in the statement of
the lemma becomes
\begin{equation*}
-\, \epsilon\, \int_{\partial \mc A\cup \partial \mc B} e^{[U(\bs x) - U(\bs y)]/\epsilon}
\, G^*_{\mc B^c} (\bs y , \bs x)\, \bb M \, \nabla h_{\mc B, \mc A} (\bs y)
\cdot \bs n_{(\mc A\cup \mc B)^c} (\bs y) \, \sigma(d \bs y)\;.
\end{equation*}
In the argument below, as we did in the two previous lemmata, we need
to remove from the integration region a ball $B_\delta(\bs x)$ and let
$\delta\to 0$. As the argument should be clear at this point, we
ignore the singularity of the Green function at $\bs x$.  By the
divergence theorem, and since $\mc L_\epsilon h_{\mc B, \mc A}=0$ on
$(\mc A\cup \mc B)^c$, this expression is equal to
\begin{equation*}
-\, \epsilon\, \int_{(\mc A\cup \mc B)^c} e^{[U(\bs x) - U(\bs y)]/\epsilon}
\, \bb M^\dagger \, \nabla G^*_{\mc B^c} (\bs y , \bs x)\,
\nabla h_{\mc B, \mc A} (\bs y)\, d \bs y\;.
\end{equation*}
Applying the divergence theorem a second time, as $\mc L^*_\epsilon
G^*_{\mc B^c} (\cdot , \bs x) = - \delta_{\bs x}(\cdot)$ and $h_{\mc
  B, \mc A} = \bs 1\{\mc B\}$ on $\partial \mc A\cup \partial \mc B$,
this
expression becomes
\begin{equation*}
-\,  h_{\mc B, \mc A} (\bs x)
\; -\, \epsilon\, \int_{\partial \mc B} e^{[U(\bs x) - U(\bs y)]/\epsilon}
\, \bb M^\dagger \, \nabla G^*_{\mc B^c} (\bs y , \bs x)
\cdot \bs n_{\mc B^c} (\bs y) \, \sigma(d \bs y)\;.
\end{equation*}
By Lemma \ref{7-l7} the integral is equal to $f(\bs x)$ where $f$ is
the solution of \eqref{22} with $\Omega = \mc B^c$, $\mf g=0$ and
$\mf b=1$. Since the solution of this equation is equal to $1$, the
previous expression is equal to $1- h_{\mc B, \mc A} (\bs x) = h_{\mc
  A, \mc B} (\bs x)$, as claimed.
\end{proof}

In the present context, the equilirbium measure $\nu_{\mc A, \mc B}$ ,
introduced in \eqref{2-2}, is the probability measure on $\partial \mc A$ given by
\begin{equation}
\label{7-15}
\nu_{\mc A, \mc B} (d\bs y) \;=\; \frac{\epsilon}{Z_\epsilon\,
  \Cap(\mc A,\mc B)} \, e^{- U(\bs y)/\epsilon}\, \bb M^\dagger \,
\nabla h^*_{\mc A,\mc B} (\bs y) \cdot \bs n_{\Omega}(\bs y) \,
\sigma (d \bs y)\;.
\end{equation}
In particular, in view of \eqref{e62}, in terms of the equilirbium
measure, the formula for $h^*_{\mc A,\mc B}$ becomes
\begin{equation*}
h^*_{\mc A,\mc B} (\bs x) \;=\; Z_\epsilon\, \Cap(\mc A,\mc B) \,
\, e^{U(\bs x)/\epsilon} \int_{\partial \mc A} G_{\mc B^c} (\bs y, \bs x)
\, \nu_{\mc A, \mc B} (d \bs y)\;, \quad \bs x\not\in\mc B\;.
\end{equation*}
Therefore, since $h^*_{\mc A,\mc B} (\bs x)=1$ for $\bs x\in \mc A$ and since the equilirbium measure is a probability measure, we obtain
that
\begin{equation}
\label{7-4}
\inf_{\bs y\in \partial \mc A} G_{\mc B^c} (\bs y, \bs x) \;\le\;
\frac{e^{-U(\bs x)/\epsilon}} {Z_\epsilon\, \Cap(\mc A,\mc B)} \;,
\quad \bs x\in \mc A  \;\cdot
\end{equation}

\begin{lemma}
\label{7-l9}
Let $\mc D$ be an open set with a $C^{2,\alpha}$-boundary for some
$0<\alpha<1$. Fix $\bs x\not\in \overline{\mc D}$. For every $r>0$
there exists a finite constant $C_0$, depending only on $r$ and $U$,
such that for all $0<\epsilon <\epsilon_0 = d(\bs x, \mc D)/2r$,
\begin{equation*}
\sup_{\bs y\in \partial B_{r\epsilon}(\bs x) } G_{\mc D^c} (\bs y, \bs x) \;\le\;
C_0\, \inf_{\bs y\in \partial B_{r\epsilon}(\bs x)} G_{\mc D^c} (\bs y, \bs x) \;.
\end{equation*}
\end{lemma}

\begin{proof}
The proof is a well-known application of the Harnack inequality, see
e.g. \cite[Lemma 4.6]{BEGK1}. Note that the
supremum and infimum are carried over the boundary of
$B_{r\epsilon}(\bs x)$. The result does not hold if this boundary is
replaced by the ball since the Green function diverges on the
diagonal, as stated in Lemma \ref{7-l5}.
\end{proof}

\begin{proposition}
\label{7-l6}
Let $\mc A$ and $\mc B$ be open sets satisfying Assumption S. Fix
$\bs x\not\in \overline{\mc A\cup \mc B}$ and $r>0$. Let $\epsilon_0 = d(\bs x,
\mc A\cup \mc B)/2r$. There exists a finite constant $C_0$, depending
only on $r$ and $U$, such that for all $\epsilon < \epsilon_0$,
\begin{equation*}
h_{\mc A, \mc B}(\bs x) \;\le\; C_0\,
\frac{\Cap (B_{r\epsilon}(\bs x), \mc A) }
{\Cap (B_{r\epsilon}(\bs x), \mc A \cup \mc B) }\;\cdot
\end{equation*}
\end{proposition}

\begin{proof}
Fix two open sets $\mc A$, $\mc B$ satisfying Assumption S and $\bs
x\not\in \overline{\mc A \cup\mc B}$. By Lemma \ref{7-l7},
\begin{equation*}
h_{\mc A, \mc B}(\bs x) \;=\;
-\, \epsilon\, \int_{\partial \mc A} e^{[U(\bs x) - U(\bs y)]/\epsilon} \, \bb
M^\dagger \nabla G_{(\mc A \cup \mc B)^c}^* (\bs y, \bs x)
\cdot \bs n_{\mc A^c}(\bs{y}) \, \sigma(d\bs y)\;,
\end{equation*}
Let $\mc C$ be an open set with a smooth boundary and such that $d(\mc
C ,\mc A \cup \mc B) >0$, $\bs x\in \mc C$. Since $h_{\mc A, \mc
  B\cup\mc C}=1$ on $\partial \mc A$ we may add $h_{\mc A, \mc
  B\cup\mc C}$ inside the integral and then extend the integral to
$\partial \mc A \cup \partial \mc B \cup \partial \mc C$.  By the
divergence theorem, since $\bs x\in \mc C$ and $\mc L_\epsilon^*
G_{(\mc A \cup \mc B)^c}^* (\cdot, \bs x)=0$ on $(\mc A \cup \mc B
\cup \overline{\mc C})^c$, the previous expression is equal to
\begin{equation*}
-\, \epsilon\, \int_{(\mc A \cup  \mc B \cup \mc C)^c}
e^{[U(\bs x) - U(\bs y)]/\epsilon} \,
\nabla h_{\mc A, \mc B\cup\mc C} (\bs y) \cdot \bb M^\dagger
\nabla G_{(\mc A \cup \mc B)^c}^* (\bs y, \bs x) \, d\bs y \;.
\end{equation*}
Applying once more the divergence theorem, and since $\mc L_\epsilon
h_{\mc A, \mc B\cup\mc C} =0$ on $(\mc A \cup \mc B \cup \overline{\mc
  C})^c$, the previous expression is equal to
\begin{equation}\label{ha1}
-\, \epsilon\, \int_{\partial \mc C}
e^{[U(\bs x) - U(\bs y)]/\epsilon} \,
G_{(\mc A \cup \mc B)^c}^* (\bs y, \bs x)
\, \bb M \, \nabla h_{\mc A, \mc B\cup\mc C} (\bs y)
\cdot \bs n_{\mc C^c}(\bs{y}) \, \sigma(d\bs y)\;.
\end{equation}
Note that the integration is carried over $\partial \mc C$ because
$G_{(\mc A \cup \mc B)^c}^* (\bs y, \bs x)$ vanishes on $\partial \mc
A\cup \partial \mc B$.

We prove below that for all $\bs y \in \partial \mc C$, it holds that
\begin{equation}
\label{ha2}
\bb M \, \nabla h_{\mc A,
  \mc B\cup\mc C} (\bs y) \cdot \bs n_{\mc C^c}(\bs{y}) \ge \bb M \, \nabla h_{\mc A,
  \mc C} (\bs y) \cdot \bs n_{\mc C^c}(\bs{y})\;.
\end{equation}
Then, after replacing $\nabla h_{\mc A, \mc B\cup\mc C}$ by $\nabla
h_{\mc A, \mc C}$ in \eqref{ha1} and observing that $\nabla h_{\mc A,
  \mc C} = -\, \nabla h_{\mc C, \mc A}$, we conclude that
\begin{align*}
h_{\mc A, \mc B}(\bs x) \; & \le\; \epsilon\,
e^{U(\bs x)/\epsilon} \, \int_{\partial \mc C}
e^{- U(\bs y)/\epsilon} \,
G_{(\mc A \cup \mc B)^c}^* (\bs y, \bs x)
\, \bb M \, \nabla h_{\mc C, \mc A} (\bs y)
\cdot \bs n_{\mc C^c} (\bs{y}) \, \sigma(d\bs y) \\
& \le\; \epsilon\, e^{U(\bs x)/\epsilon} \,
\sup_{\bs y\in \partial \mc C}
G_{(\mc A \cup \mc B)^c}^* (\bs y, \bs x)
\int_{\partial \mc C} e^{- U(\bs y)/\epsilon} \,
\, \bb M \, \nabla h_{\mc C, \mc A} (\bs y)
\cdot \bs n_{\mc C^c} (\bs{y})\, \sigma(d\bs y) \\
& =\; e^{U(\bs x)/\epsilon} \,
\sup_{\bs y\in \partial \mc C}
G_{(\mc A \cup \mc B)^c}^* (\bs y, \bs x)\,
Z_\epsilon \, \Cap(\mc C, \mc A)\;.
\end{align*}
In the last step we used the formula \eqref{06} for the capacity.

Fix $r>0$, let $\epsilon_0 = d(\bs x, \mc A \cup \mc B)/2r$, and
set $\mc C = B_{r \epsilon}(\bs x)$. By Lemma \ref{7-l9} with $\mc D =
\mc A \cup \mc B$, there exists a finite constant $C_0 = C_0(r)$ such
that
\begin{equation*}
h_{\mc A, \mc B}(\bs x) \; \le\;
C_0 \, e^{U(\bs x)/\epsilon} \,
\inf_{\bs y\in \partial B_{r \epsilon}(\bs x)}
G_{(\mc A \cup \mc B)^c}^* (\bs y, \bs x)\,
Z_\epsilon \, \Cap(B_{r \epsilon}(\bs x), \mc A)\;.
\end{equation*}
To complete the proof of the lemma, it remains to recall estimate
\eqref{7-4}.

It remains to show that \eqref{ha2}. Fix $\bs x \in \partial \mc C$.  The vector $\bb
M^\dagger \bs n_{\mc C^c} (\bs x)$ points inward to $\mc C$ because
$\bs n_{\mc C^c} (\bs x) \cdot \bb M^\dagger \bs n_{\mc C^c} (\bs x) =
\bs n_{\mc C^c} (\bs x) \cdot \bb S \,\bs n_{\mc C^c} (\bs x)>0$. In
particular, for $\delta$ small enough, $\bs x -\delta \bb M^\dagger
\bs n_{\mc C^c} (\bs x) \in (\mc A \cup \mc B\cup\mc C)^c$. Since, by
\eqref{30}, $h_{\mc A, \mc B\cup\mc C} \le h_{\mc A, \mc C}$ on this
set,
\begin{equation*}
h_{\mc A, \mc B\cup\mc C} (\bs x -\delta \bb M^\dagger
\bs n_{\mc C^c} (\bs x))
\;\le\; h_{\mc A, \mc C} (\bs x -\delta \bb M^\dagger
\bs n_{\mc C^c} (\bs x))
\end{equation*}
for $\delta$ small enough. Subtracting $h_{\mc A, \mc B\cup\mc C} (\bs
x) = h_{\mc A, \mc C} (\bs x) =0$ on both sides, dividing by $\delta$
and letting $\delta\to 0$ yields that
\begin{equation*}
\bs n_{\mc C^c} (\bs x) \cdot \bb M \, \nabla h_{\mc A, \mc B\cup\mc C} (\bs x)
\;\ge\; \bs n_{\mc C^c} (\bs x) \cdot \bb M \,  \nabla  h_{\mc A, \mc C} (\bs x) \;,
\end{equation*}
as claimed.
\end{proof}

\begin{lemma}
\label{7-l10}
There exists a finite constant $C_0$ and $\epsilon_0>0$ such that for
all $\bs y\in \mc W_2$, $0<\epsilon<\epsilon_0$,
\begin{equation*}
\textup{cap} (B_{\epsilon}(\boldsymbol{y}),\,\mathcal{V}_{1})
\; \le\; {C_0}\,{Z_\epsilon^{-1}} \, e^{-H/\epsilon}\;, \quad
\textup{cap} (B_{\epsilon}(\boldsymbol{y}),\,\mathcal{V}_{2})
\;\ge\; {C_0}\,{Z_\epsilon^{-1}} \, \epsilon^{d} \,
e^{-U(\boldsymbol{y})/\epsilon}\; .
\end{equation*}
\end{lemma}

\begin{proof}
The generator $\mc L_\epsilon$ satisfies a sector condition with
constant $\Lambda/\lambda$. Indeed, by the Cauchy-Schwarz inequality and by \eqref{7-5},
for any smooth functions $f$, $g:\bb R^d\to \bb R$,
\begin{equation*}
\< f ,\, \mc L_\epsilon g \>^2_{\mu_\epsilon} \;\le\; \frac{\Lambda}\lambda\,
\< f ,\, (- \mc L_\epsilon) f \>_{\mu_\epsilon} \,
\< g ,\,(-\mc L_\epsilon) g \>_{\mu_\epsilon}\;.
\end{equation*}
It follows from Lemmata 2.5 and 2.6 in \cite{GL} that the capacity
between two sets can be estimated from below and from above by the
capacity associated to the symmetric operator $(1/2) (\mc L_\epsilon +
\mc L^*_\epsilon)$. Denote by $\textup{cap}^{s}(\mc A, \mc B)$ the
capacity between the sets $\mc A$, $\mc B$ for the symmetric process.

We start with the upper bound.  Let $\overline{\mathcal{W}}_{2}^{t} =
\{\boldsymbol{x}: d(\boldsymbol{x}, \mathcal{W}_{2}) \le t\}$, $t\ge
0$. Let $\epsilon_0>0$ such that $\mathcal{V}_{1} \cap
\overline{\mathcal{W}}_{2}^{2\epsilon_0} = \varnothing$, and fix $0<\epsilon<
\epsilon_0$. There exist a smooth function $h_{\epsilon}$ and a finite
constant $C_0$, independent on $\epsilon$, such that
$h_{\epsilon}\equiv 1$ on $\overline{\mathcal{W}}_{2}^{\epsilon}$,
$h_{\epsilon}\equiv 0$ on $(\overline{\mathcal{W}}_{2}^{2\epsilon})^{c}$, and
\begin{equation*}
\Vert \nabla h_{\epsilon}(\boldsymbol{x})\Vert\;\le\;
C_0\, \epsilon^{-1}\;\;\;\mbox{for all }
\boldsymbol{x}\in\overline{\mathcal{W}}_{2}^{2\epsilon}
\setminus\overline{\mathcal{W}}_{2}^{\epsilon}\;.
\end{equation*}
Then, since
$B_{\epsilon}(\boldsymbol{y})\subset\overline{\mathcal{W}}_{2}^{\epsilon}$ and
$\mathcal{V}_{1}\subset(\overline{\mathcal{W}}_{2}^{2\epsilon})^{c}$, and since
$U(\boldsymbol{x})=H+O(\epsilon)$ for all
$\boldsymbol{x}\in\mathcal{W}_{2}^{2\epsilon}\setminus\mathcal{W}_{2}^{\epsilon}$,
by the Dirichlet's principle for reversible processes,
\begin{align*}
\textup{cap}^{s}(B_{\epsilon}(\boldsymbol{y}),\,\mathcal{V}_{1})
\; &\le\; \frac{\epsilon}{Z_\epsilon} \, \int_{\mathbb{R}^{d}}
e^{-U(\boldsymbol{x})/\epsilon}\, \nabla h_{\epsilon}(\boldsymbol{x})
\cdot\mathbb{S}\, \nabla h_{\epsilon}(\boldsymbol{x})\, d\boldsymbol{x} \\
& \le C_0 \, \frac{\epsilon}{Z_\epsilon} \,
e^{-H/\epsilon}\, \epsilon^{-2}
\, \mbox{vol }(\overline{\mathcal{W}}_{2}^{2\epsilon}
\setminus\overline{\mathcal{W}}_{2}^{\epsilon})
\;\le \; \frac{C_0}{Z_\epsilon} \, e^{-H/\epsilon}\;.
\end{align*}

We turn to the lower bound, where we follow the argument of
\cite[Proposition 4.7] {BEGK1}. Fix $0<\epsilon<1$. Let $\bs \rho(t)$ be
a smooth path connecting $\boldsymbol{y}$ to $\boldsymbol{m}_{2}$ such
that $U(\bs \rho(t))$ is decreasing in $t$, and $\|\dot{\bs \rho}(t)\|=1$
for all $t$. Let $D_{\epsilon}$ be a $(d-1)$-dimensional disk of
radius $\epsilon$ centered at origin. By the proof \cite[Proposition
4.7] {BEGK1} up to equation (4.26), we obtain that
\begin{equation*}
\textup{cap}^{s}(B_{\epsilon}(\boldsymbol{y}),\,\mathcal{V}_{2})
\;\ge\; \frac{\epsilon}{Z_{\epsilon}}\,
\int_{D_{\epsilon}} d\boldsymbol{z}_{\perp}
\Big[\int_{0}^{|\bs \rho|}dt \,
e^{U(\bs \rho(t)+\boldsymbol{z}_{\perp})/\epsilon}\Big]^{-1}\; .
\end{equation*}
Let $L_{0}=\sup_{\boldsymbol{x}\in\overline{\mathcal{W}}_{2}^{1}} \|\nabla
U(\boldsymbol{x})\|$. As $U(\bs \rho(t))$ decreases in $t$,
\begin{equation*}
\int_{0}^{|\bs \rho|} e^{U(\bs \rho(t)+\boldsymbol{z}_{\perp})/\epsilon}
\, dt
\;\le \; e^{L_{0}}\int_{0}^{|\bs \rho|} e^{U(\bs \rho(t))/\epsilon} \, dt
\;\le\; e^{L_{0}} \, |\bs \rho| \, e^{U(\boldsymbol{y})/\epsilon}\;.
\end{equation*}
Since the set $\mc W_2$ is bounded, we can choose smooth paths with
length $|\bs \rho|$ uniformly bounded. Hence, by the previous estimates,
\begin{equation*}
\textup{cap}^{s}(B_{\epsilon}(\boldsymbol{y}),\,\mathcal{V}_{2})
\;\ge\; C_0 \frac{\epsilon^{d}}{Z_{\epsilon}}\,
e^{-U(\boldsymbol{y})/\epsilon}\; ,
\end{equation*}
as claimed.
\end{proof}

\begin{proposition}
\label{7-l11}
There exists a finite constant $C_0$ and $\epsilon_0>0$ such that for
all $\bs y\in \mc W_2$, $0<\epsilon<\epsilon_0$,
\begin{equation}
\label{7-6}
h_{\mc V_1, \mc V_2}(\bs y) \;\le\; C_0\,
\epsilon^{-d}\, e^{- [H - U(\bs y)] /\epsilon}\;.
\end{equation}

\end{proposition}


\begin{proof}
Fix $\bs y \in \mc W_2$. By Proposition \ref{7-l6} with $r=1$, for all
$\epsilon$ small enough and since the capacity is monotone in its
arguments,
\begin{equation*}
h_{\mc V_1, \mc V_2}(\bs y) \;\le\; C_0\,
\frac{\Cap (B_{\epsilon}(\bs y), \mc V_1) }
{\Cap (B_{\epsilon}(\bs y), \mc V_2) }\;\cdot
\end{equation*}
By Lemma \ref{7-l10}, this expression is bounded above by the right-hand side of \eqref{7-6} for all $\epsilon$ small enough, as claimed.
\end{proof}

\section{The vector fields $\Theta_{\bs{q}_{\epsilon}}$,
$\Theta^*_{\bs{q}_{\epsilon}}$}
\label{sec6}

We prove in this section Lemmata \ref{tp0}, \ref{tp1} and \ref{tp2}.
Throughout this section, $C_1$, $C_2$, $C_3$ represent large but
finite positive constants, independent of the variables $\epsilon$ and
$\eta$ introduced in Section \ref{sec4}, and whose value may change
from line to line. Similarly, $c_1$, $c_2$ represent small but
positive constants with the same properties of $C_1$, $C_2$.

We start by recalling basic properties of the vector $\bs v$ and the
matrices $\bb M$, $\bb L$. Most of these results were proven in
Section 4 of \cite{LS1}.  Recall that we write a vector $\bs{u} \in
\bb R^d$ as $\sum_{1\le i\le d} u_{i} \, \bs{e}_{i}$, that we
represent by $\bs v$ the eigenvector of $\bb{L} \bb{M}$ associated to
the eigenvalue $-\mu$, and that we assumed $v_1>0$.

\begin{lemma}
\label{bl1}
We have that
\begin{equation*}
\bs{v}\cdot\bb{L}^{-1}\bs{v}\;=\;-\frac{v_{1}^{2}} {\lambda_{1}} \;+\;
\sum_{k=2}^{d}\frac{v_{k}^{2}}{\lambda_{k}}\;=\;-\, \frac{1}{\alpha}\;.
\end{equation*}
\end{lemma}

\begin{proof}
Since $\bs{v}$ is the eigenvector of $\bb{L} \bb{M}$ associated
to the eigenvalue $-\mu$, by \eqref{ins02},
\begin{equation*}
-\bs{v}\cdot\bb{L}^{-1}\bs{v}\;=\;-\bs{v}\cdot\bb{M}
(\bb{L}\bb {M})^{-1} \bs{v}\;=\;
\frac{1}{\mu}\, \bs{v}\cdot\bb{M}\bs{v}\;=\;\frac{1}{\alpha}\;\cdot
\end{equation*}
\end{proof}

The next two results are Lemmata 4.1 and 4.2  of \cite{LS1}. Denote by
$\bs w ^{\dagger}$ the transpose of a vector $\bs w\in\bb R^d$.

\begin{lemma}
\label{bl2}
The matrix $\bb{L}+2\alpha\bs{v}\bs{v}^{\dagger}$ is positive definite and
$\det(\bb{L}+2\alpha\bs{v}\bs{v}^{\dagger})=-\det\bb{L}$.
\end{lemma}

\begin{lemma}
\label{bl3}
The matrix $\bb{L}+\alpha\bs{v}\bs{v}^{\dagger}$ is non-negative definite
and $\det(\bb{L}+\alpha\bs{v}\bs{v}^{\dagger})=0$.  The null space of the
matrix $\bb{L}+\alpha\bs{v}\bs{v}^{\dagger}$ is one-dimensional and
spanned by the vector $\bb{L}^{-1}\bs{v}$.
\end{lemma}

\noindent{\bf A. Proof of Lemma \ref{tp0}.} The proof of Lemma \ref{tp0}
is based on the following estimate.

\begin{lemma}
\label{bl6}
We have that
\begin{equation*}
\int_{\mc{B}_{\epsilon}}\nabla p_{\epsilon}(\bs{z})
\cdot\bb{M}\nabla p_{\epsilon}(\bs{z})
\, e^{-\left(U(\bs{z})-H\right)/\epsilon}d\bs{z}
\;=\;\left[1+o_{\epsilon}(1)\right]\,(2\pi\epsilon)^{\frac{d}{2}-1}
\omega(\bs{0})\;.
\end{equation*}
\end{lemma}

\begin{proof}
By the definition \eqref{ins03} of $p_{\epsilon}$,
\begin{equation*}
\nabla p_{\epsilon}(\bs{z})\;=\;\frac{1}{C_{\epsilon}}
\exp\left\{ -\frac{\alpha}{2\epsilon}(\bs{z}\cdot\bs{v})^{2}\right\}
\,\bs{v}\;,
\end{equation*}
and by the Taylor expansion of the potential $U$ around $\bs{0}$, on
the set $\mc{B}_{\epsilon}$,
\begin{equation*}
U(\bs{z})\,-\,H\;=\;(1/2)\, \bs{z}\cdot\bb{L}\bs{z}+O(\delta^{3})\;.
\end{equation*}
Since $\exp\{\delta^3/\epsilon\} = 1 + o_{\epsilon}(1)$ and
$C_\epsilon = \sqrt{2\pi\epsilon/\alpha}$, by \eqref{ins02} and by the
two previous identities, the left-hand side of the expression appearing
in the statement of the lemma is equal to
\begin{align*}
 &  \left[1+o_{\epsilon}(1)\right]\,
\frac{\bs{v}\cdot\bb{M}\bs{v}}{C_{\epsilon}^{2}}
\int_{\mc{B}_{\epsilon}}\exp \Big\{ \frac{1}{2\epsilon}
\bs{z}\cdot [\bb{L}+2\alpha\bs{v}\bs{v}^{\dagger}]\, \bs{z}\Big\} \, d\bs{z}\\
 &  \quad =\;\left[1+o_{\epsilon}(1)\right]\,
\frac{\mu}{2\pi\epsilon} \int_{\mc{B}_{\epsilon}}
\exp\Big\{ \frac{1}{2\epsilon}\bs{z}\cdot
[\bb{L}+2\alpha\bs{v}\bs{v}^{\dagger} ]\, \bs{z}\Big\} \, d\bs{z}\;.
\end{align*}
It is easy to verify that
\begin{equation*}
[-\delta,\,\delta]\times \prod_{i=2}^d
\Big[-\sqrt{\frac{\lambda_1}{4(d-1)\lambda_i}}\delta ,\,
\sqrt{\frac{\lambda_1}{4(d-1)\lambda_i}}\delta \, \Big] \subseteq \mc{B}_\epsilon\;.
\end{equation*}
Hence, by the change of coordinates $\bs{y}= (1/\sqrt{\epsilon}) \,
\bs{z}$, and by Lemma \ref{bl2}, the last integral is equal to
\begin{equation*}
\left[1+o_{\epsilon}(1)\right]\frac{(2\pi\epsilon)^{d/2}}
{\sqrt{-\det\bb{L}}}\;\cdot
\end{equation*}
This completes the proof of the lemma since
$\omega(\bs{0})\;=\;\mu/\sqrt{-\det\bb{L}}$.
\end{proof}

We may now turn to the Proof of Lemma \ref{tp0}.

\begin{proof}[Proof of Lemma \ref{tp0}]
By the definition of $\Theta_{\bs{q}_{\epsilon}}$ and
$\Theta_{\bs{q}_{\epsilon}}^{*}$, it is easy to check that
\begin{equation*}
\Big\Vert \frac{\Theta_{\bs{q}_{\epsilon}}+
\Theta_{\bs{q}_{\epsilon}}^{*}}{2}\Big\Vert ^{2}
\;=\; \frac{\epsilon}{Z_{\epsilon}}
\int_{\mc{B}_{\epsilon}}\nabla p_{\epsilon}(\bs{x})\cdot\bb{M}
\nabla p_{\epsilon}(\bs{x})e^{-U(\bs{x})/\epsilon}d\bs{x}\;.
\end{equation*}
At this point, the assertion of Lemma \ref{tp0} follows from
Lemma \ref{bl6}.
\end{proof}

\noindent{\bf B. Proof of Lemma \ref{tp1}.} Define a mollified version
of the vector field $\bs{q}_{\epsilon}$ as $\bs{q}_{\epsilon}^{(\eta)}
= \bs{q}_{\epsilon}*\phi_{\eta}$, where $\eta = \eta(\epsilon)$ is
such that $\lim_{\epsilon\to 0} \eta(\epsilon)/ \delta(\epsilon)
=0$. Let $\Theta_{\bs{q}_{\epsilon}^{(\eta)}} (\bs{z}) = \epsilon
Z_{\epsilon}^{-1} e^{-U(\bs{z})/\epsilon} \bb{M}^{\dagger}
\bs{q}_{\epsilon}^{(\eta)} (\bs{z})$. By Young's inequality,
\begin{equation*}
\big\Vert \Phi_{p_{\epsilon}^{(\eta)}} - \Theta_{\bs{q}_{\epsilon}}
\big\Vert ^{2} \;\le\;2 \, \big\Vert \Phi_{p_{\epsilon}^{(\eta)}} -
\Theta_{\bs{q}_{\epsilon}^{(\eta)}} \big\Vert ^{2}
\;+\; 2\, \big \Vert \Theta_{\bs{q}_{\epsilon}^{(\eta)}}
-\Theta_{\bs{q}_{\epsilon}}\big\Vert ^{2}\;.
\end{equation*}
We estimate the two terms on the right-hand side separately.  Lemma
\ref{tp1} follows from Lemmata \ref{tp11} and \ref{tp12} below.

\begin{lemma}
\label{tp11}
There exist finite constants $C_{1}$, $c_{2}$, $C_{3}$, such that
\begin{equation*}
\big\Vert \Phi_{p_{\epsilon}^{(\eta)}} -
\Theta_{\bs{q}_{\epsilon}^{(\eta)}}\big\Vert ^{2}
\;\le\; \frac{C_{1}}{Z_{\epsilon}}\,
e^{-H/\epsilon}\, \frac{\epsilon^{c_{2}K^{2}}}{\eta^{d}}
\, e^{C_{3}\eta/\epsilon}\;.
\end{equation*}
\end{lemma}

The proof of this lemma is divided in several steps.  The crucial
point is the control of the discontinuity of $p_{\epsilon}$ along the
boundary $\partial\mc{X}_{\epsilon}\cup\partial\mc{B}_{\epsilon}$.
For $\bs{z}\in\partial\mc{X}_{\epsilon}$, let $\bs{n}(\bs{z})$ be the
inner normal vector to $\mc{X}_{\epsilon}$ at $\bs{z}$ (and hence the
outer normal vector to $\mc{W}_{1}^{\epsilon} \cup
\mc{W}_{2}^{\epsilon} \cup \mc{B}_{\epsilon}$). Similarly, for
$\bs{z}\in\partial \mc{B}_{\epsilon} \setminus \partial
\mc{X}_{\epsilon}$, let $\bs{n}(\bs{z})$ be the outer normal vector to
$\mc{B}_{\epsilon}$ at $\bs{z}$. In this manner, the normal vector is
defined for all $\bs{z} \in \partial \mc{X}_{\epsilon} \cup \partial
\mc{B}_{\epsilon}$.

Define the functions $\mf{d}^{+}$, $\mf{d}^{-}$ on
$\partial\mc{X}_{\epsilon}\cup\partial\mc{B}_{\epsilon}$ by
\begin{equation*}
\mf{d}^{+}(\bs{z})\;=\; \lim_{t\rightarrow0^{+}}
p_{\epsilon}\left(\bs{z}+t\, \bs{n}(\bs{z})\right)\;,\quad
\mf{d}^{-}(\bs{z})\;=\;
\lim_{t\rightarrow0^{+}}p_{\epsilon} \left(\bs{z}-t\,
\bs{n}(\bs{z})\right)\;.
\end{equation*}
Let $\mf{d} : \partial\mc{X}_{\epsilon} \cup \partial\mc{B}_{\epsilon}
\to \bb R$ be given by $\mf{d} = \mf{d}^{+} - \mf{d}^{-}$, so that
$\mf{d}(\bs{z})$ represents the discontinuity of $p_{\epsilon}$ at
$\bs{z}$.  The next assertion provides an estimate of $\mf{d}$.

\begin{asser}
\label{bl5}
There exist finite constants $C_{1}$, $c_{2}>0$,
such that
\begin{equation*}
\left[\mf{d}(\bs{z})\right]^{2}e^{-U(\bs{z})/\epsilon}
\;\le\;C_{1} \, e^{-H/\epsilon} \, \epsilon^{c_{2}K^{2}}
\end{equation*}
for all $\bs{z}\in\partial\mc{X}_{\epsilon}\cup\partial\mc{B}_{\epsilon}$.
\end{asser}

\begin{proof}
Fix $\bs{z}\in\partial\mc{X}_{\epsilon}$, so that $|\mf{d}(\bs{z})|
\le 1$, and $U(\bs{z})=H + (1/4) \lambda_{1} \delta^{2}$. In this case
Assertion \ref{bl5} follows from the definition of $\delta = K
\sqrt{\epsilon\log(1/\epsilon)}$.

Fix $\bs{z} \in \partial \mc{B}_{\epsilon} \setminus \partial
\mc{X}_{\epsilon}$ so that, by Lemma \ref{lems0}, $\bs{z}
\in \partial_{+} \mc{C}_{\epsilon} \cup \partial_{-}
\mc{C}_{\epsilon}$.  The proof in this case is similar to the one of
Lemma 4.7 in \cite{LS1}. Assume that
$\bs{z}\in \partial_{+}\mc{C}_{\epsilon}$, the proof for
$\bs{z}\in \partial_{-}\mc{C}_{\epsilon}$ being similar.  For $\bs{z}
\in \partial_{+}\mc{C}_{\epsilon}$, $\mf{d}^{+}(\bs{z})=1$ and
\begin{equation*}
\mf{d}^{-}(\bs{z})\;=\;\frac{1}{C_{\epsilon}}
\int_{-\infty}^{\bs{z}\cdot\bs{v}} \exp \left\{
-\frac{\alpha}{2\epsilon}t^{2}\right\} \, dt\;,
\end{equation*}
so that
\begin{equation}
\label{se21}
\mf{d}(\bs{z})\;=\; \frac{1}{C_{\epsilon}}
\int_{\bs{z}\cdot\bs{v}}^{\infty}\exp\left\{
-\frac{\alpha}{2\epsilon}t^{2}\right\} \, dt
\;=\;\frac{1}{\sqrt{2\pi}}
\int_{\sqrt{\frac{\alpha}{\epsilon}}(\bs{z}\cdot\bs{v})}^{\infty}
e^{- t^{2}/2 } \, dt\;.
\end{equation}

We claim that there exists a constant $c>0$ such that, for every
$\bs{z}\in\partial_{+}\mc{C}_{\epsilon}$, either $\bs{z}\cdot\bs{v}\ge
c\delta$ or $\bs{z}\cdot\bb{L}\bs{z}\ge c\delta^{2}$. Indeed,
by Lemma \ref{bl1} and since $v_1>0$, there exists $c>0$ such that
\begin{equation}
\label{se22}
(\lambda_{1}+c) \, \sum_{k=2}^{d}\frac{v_{k}^{2}}{\lambda_{k}}
\;<\;(v_{1}-c)^{2}\;.
\end{equation}
The claim is in force with this constant $c$. Assume it is not. This
means that there exists $\bs{z}\in\partial_{+}\mc{C}_{\epsilon}$ such
that $\bs{z}\cdot\bs{v} < c\delta$ and $\bs{z}\cdot\bb{L}\bs{z} <
c\delta^{2}$. Since $\bs{z}\in\partial_{+} \mc{C}_{\epsilon}$,
$\bs{z}$ can be expressed as
\begin{equation*}
\bs{z}\;=\;\delta \, \Big(\bs{e}_{1} +
\sum_{k=2}^{d} z_{k} \, \bs{e}_{k}\Big)\;.
\end{equation*}
Since $\bs{v}=\sum_{1\le i \le d} v_{k}\, \bs{e}_{k}$, the condition
$\bs{z}\cdot\bs{v}<c\delta$ is equivalent to
\begin{equation*}
v_{1}-c\;<\;-\sum_{k=2}^{d}z_{k}\, v_{k}\;.
\end{equation*}
On the other hand, the condition $\bs{z}\cdot\bb{L}\bs{z}<c\delta^{2}$
can be rewritten as
\begin{equation*}
\sum_{k=2}^{d} z^2_{k} \, \lambda_{k}\;<\;\lambda_{1} \,+\, c\;.
\end{equation*}
Inserting the two previous bounds in \eqref{se22} we obtain that
\begin{equation*}
\sum_{k=2}^{d}\frac{v_{k}^{2}}{\lambda_{k}}\,
\sum_{k=2}^{d}z^2_{k}\, \lambda_{k}\;<\;\Big(\sum_{k=2}^{d}z_{k}v_{k}\Big)^{2}\;,
\end{equation*}
which contradicts to the Cauchy-Schwarz inequality. This proves the
claim.

We are now in a position to prove Assertion \ref{bl5} for $\bs{z}
\in \partial_{+} \mc{C}_{\epsilon}$. Suppose first that
$\bs{z}\cdot\bs{v} \ge c\delta$. Since $\int_{a}^{\infty}\exp \{ -
t^{2}/2 \} \, dt \le (1/a) \exp \{ - a^{2}/2 \}$ for $a>0$, by
\eqref{se21} and since $\bs{z}\cdot\bs{v} \ge c\delta$,
\begin{equation*}
0\;\le\;\mf{d}(\bs{z})\;\le\;\frac{1}{\sqrt{2\pi}}\,
\frac{\sqrt{\epsilon}} {\sqrt{\alpha} \, (\bs{z}\cdot\bs{v})}\,
e^{-(\alpha/2\epsilon) \, (\bs{z}\cdot\bs{v})^{2}} \;\le\;
C\, \frac{\sqrt{\epsilon}}{\delta}\,
e^{-(\alpha/2\epsilon)\, (\bs{z}\cdot\bs{v})^{2}}
\end{equation*}
for some finite constant $C$. On the other hand, by the Taylor
expansion,
\begin{equation}
\label{se26}
U(\bs{z})\;=\;H+\frac{1}{2}\, \bs{z}\cdot\bb{L}\bs{z} \;+\; O(\delta^{3})\;.
\end{equation}
In view of the two previous displayed equations,
\begin{equation*}
\mf{d}(\bs{z})^{2} \, e^{-U(\bs{z})/\epsilon} \;\le\;
C\, e^{-H/\epsilon}\, \exp\Big\{ -\frac{1}{2\epsilon}\,
\bs{z}\cdot\left[\bb{L}+2\alpha\bs{v}\bs{v}^{\dagger}\right]\bs{z}\Big\} \;.
\end{equation*}
By Lemma \ref{bl2}, $\bb{L}+2\alpha\bs{v}\bs{v}^{\dagger}\ge r_{0}I$, where
$r_{0}>0$ is the smallest eigenvalue of the positive-definite matrix
$\bb{L}+2\alpha\bs{v}\bs{v}^{\dagger}$. Hence, as $z_1=\delta$,
\begin{equation*}
\bs{z}\cdot\left[\bb{L}+2\alpha\bs{v}\bs{v}^{\dagger}\right]\bs{z}
\;\ge\;r_{0}\, |\bs{z}|^{2}\;\ge\;r_{0}\, \delta^{2}\;.
\end{equation*}
In view of the previous two displayed equations, to complete the proof
of Assertion \ref{bl5}, it remains to recall the definition of
$\delta$.

Assume now that $\bs{z}$ is such that $\bs{z}\cdot\bb{L}\bs{z}\ge
c\delta^{2}$. In this case, Assertion \ref{bl5} is direct consequence
from the bound $|\mf{d}(\bs{z})|\le1$ and from \eqref{se26}.
\end{proof}

The next result expresses the difference $\nabla p_{\epsilon}^{(\eta)}
- \bs{q}_{\epsilon}^{(\eta)}$ in terms of the function $\mf{d}$.

\begin{asser}
\label{bl4}
For any $\bs{z}\in\bb R^d,$
\begin{equation*}
\nabla p_{\epsilon}^{(\eta)}(\bs{z}) -
\bs{q}_{\epsilon}^{(\eta)}(\bs{z})
\;=\;\oint_{\partial\mc{X}_{\epsilon}\cup\partial\mc{B}_{\epsilon}}
\mf{d}(\bs{y})\, \phi_{\eta}(\bs{z}-\bs{y})\, \bs{n}(\bs{y})\,
\sigma(d\bs{y})\;.
\end{equation*}
\end{asser}

\begin{proof}
We first note that
\begin{equation*}
\nabla p_{\epsilon}^{(\eta)}(\bs{z})\;=\;
\int_{\bb{R}^{d}} p_{\epsilon}(\bs{y}) \, (\nabla\phi_{\eta})
(\bs{z}-\bs{y}) \, d\bs{y}\;.
\end{equation*}
Since $p_{\epsilon}$ is smooth on each domain $\mc{B}_{\epsilon}$,
$\mc{W}_{1}^{\epsilon}$, $\mc{W}_{2}^{\epsilon}$ and $\bb{R}^{d}
\setminus \left(\mc{B}_{\epsilon} \cup \mc{W}_{1}^{\epsilon} \cup
  \mc{W}_{2}^{\epsilon}\right)$, we decompose the last integral into
four integrals in these domains, and then apply divergence theorem for
each integrals. For instance,
\begin{align*}
 &   \int_{\mc{B}_{\epsilon}} p_{\epsilon}(\bs{y}) \,
(\nabla\phi_{\eta})(\bs{z}-\bs{y}) \, d\bs{y} \\
 &   \;=\;\int_{\mc{B}_{\epsilon}} \bs{q}_{\epsilon}(\bs{y})\,
\phi_{\eta}(\bs{z}-\bs{y})\, d\bs{y}
\;+\; \oint_{\partial\mc{B}_{\epsilon}} \mf{d}^{+}(\bs{y})\,
\phi_{\eta}(\bs{z}-\bs{y})\, \bs{n}(\bs{y})\, \sigma(d\bs{y})\;.
\end{align*}
The proof is completed by adding four identities obtained in this
manner.
\end{proof}

\begin{proof}[Proof of Lemma \ref{tp11}]
There exists a finite constant $C$ such that $\bb{M} \bb{S}^{-1}
\bb{M}^{\dagger} < C \, \bb{I}$, where $\bb{I}$ stands for the $d\times d$
identity matrix. Therefore,
\begin{align*}
& \big\Vert \Phi_{p_{\epsilon}^{(\eta)}} -
\Theta_{\bs{q}_{\epsilon}^{(\eta)}} \big\Vert ^{2} \\
&  \;=\;\frac{\epsilon}{Z_{\epsilon}}\, \int_{\bb{R}^{d}}
e^{-U(\bs{z})/\epsilon} \left[\nabla p_{\epsilon}^{(\eta)}(\bs{z}) -
  \bs{q}_{\epsilon}^{(\eta)} (\bs{z})\right] \cdot
\bb{M}\bb{S}^{-1}\bb{M}^{\dagger}
\left[\nabla p_{\epsilon}^{(\eta)}(\bs{z}) -
\bs{q}_{\epsilon}^{(\eta)}(\bs{z})\right]d\bs{z}\\
&  \;\le\; \frac{C\, \epsilon}{Z_{\epsilon}}\,
\int_{\bb{R}^{d}} e^{-U(\bs{z})/\epsilon} \, \big|\nabla
p_{\epsilon}^{(\eta)}(\bs{z})
-\bs{q}_{\epsilon}^{(\eta)}(\bs{z})\big|^{2} \, d\bs{z}\;.
\end{align*}
By Assertion \ref{bl4}, this expression is equal to
\begin{equation*}
\frac{C\, \epsilon}{Z_{\epsilon}}\,
\int_{\bb{R}^{d}} e^{-U(\bs{z})/\epsilon}\,
\Big |\, \oint_{\partial\mc{X}_{\epsilon}\cup\partial\mc{B}_{\epsilon}}
\mf{d}(\bs{y}) \, \phi_{\eta}(\bs{z}-\bs{y})\, \bs{n}(\bs{y})
\, \sigma(d\bs{y}) \, \Big|^{2} \, d\bs{z}\;.
\end{equation*}

Since the surface volume of $\partial\mc{X}_{\epsilon}
\cup \partial\mc{B}_{\epsilon}$ is $[1+o_{\epsilon}(1)]M$,
where $M$ is the surface volume of $\partial\mc{W}_{1}
\cup \partial\mc{W}_{2}$, by the Cauchy-Schwarz inequality, the last
expression is bounded by
\begin{equation*}
\frac{\epsilon C}{Z_{\epsilon}}\,
\int_{\bb{R}^{d}}
\oint_{\partial\mc{X}_{\epsilon}\cup\partial\mc{B}_{\epsilon}}
e^{-U(\bs{z})/\epsilon}\, \mf{d}(\bs{y})^2 \,
\phi_{\eta}(\bs{z}-\bs{y})^{2} \, \sigma(d\bs{y})\, d\bs{z}
\end{equation*}
for some finite constant $C$, whose value may change from line to
line.

Since $U$ is Lipschitz continuous on the compact set
\begin{equation*}
\left\{ \bs{z}:|\bs{z}-\bs{y}|\le\eta\;\mbox{for some }
\bs{y}\in\partial\mc{X}_{\epsilon}\cup\partial\mc{B}_{\epsilon}\right\}
\;,
\end{equation*}
there exists a finite constant $C$, independent of $\epsilon$, such
that $U(\bs{z}) \ge U(\bs{y}) - C\eta$ for $\bs y
\in \partial\mc{X}_{\epsilon} \cup \partial \mc{B}_{\epsilon}$,
$|\bs{z}-\bs{y}| \le \eta$. As $\phi_{\eta}(\bs{z}-\bs{y})=0$ if
$|\bs{z}-\bs{y}|\ge\eta$, and as
\begin{equation*}
\int_{\bb{R}^{d}}\phi_{\eta}^{2}(\bs{z})\,
d\bs{z}\;=\;\frac{C}{\eta^{d}}\;,
\end{equation*}
the last integral is bounded by
\begin{equation*}
\frac{C\, \epsilon}{Z_{\epsilon}}\,
\frac{e^{C\eta/\epsilon}}{\eta^{d}}
\oint_{\partial\mc{X}_{\epsilon}\cup\partial\mc{B}_{\epsilon}}
e^{-U(\bs{y})/\epsilon} \, \mf{d}(\bs{y})^{2}\, \sigma(d\bs{y})\;.
\end{equation*}
To complete the proof of the lemma it remains to recall Assertion
\ref{bl5}.
\end{proof}

\begin{lemma}
\label{tp12}
Assume that $\eta\ll \delta$.  There exists a finite constant $C_{1}$,
independent of $\epsilon$ and $\eta$, such that
\begin{equation*}
\big\Vert \Theta_{\bs{q}_{\epsilon}^{(\eta)}}
-\Theta_{\bs{q}_{\epsilon}}\big\Vert ^{2}
\;\le\; o_{\epsilon}(1) \, \frac 1{Z_{\epsilon}} \,
e^{- H/\epsilon}\, \epsilon^{d/2} \, \frac{\eta}{\epsilon} \,
\Big( 1 + \frac{\eta}{\epsilon} \Big) \,
\Big( 1+ e^{C_1\eta\delta/\epsilon}\Big)\;.
\end{equation*}
\end{lemma}

Denote by $\partial^{\eta}\mc{B}_{\epsilon} $ the neighborhood of the
boundary $\partial\mc{B}_{\epsilon}$ defined by
\begin{equation*}
\partial^{\eta}\mc{B}_{\epsilon}\;=\;
\big\{ \bs{z} : |\bs{z}-\bs{y}|\le\eta\;\mbox{for some }
\bs{y}\in\partial\mc{B}_{\epsilon}\big\}\;,
\end{equation*}
and let $\mc{B}_{\epsilon}^{\eta} = \mc{B}_{\epsilon}
\setminus \partial^{\eta} \mc{B}_{\epsilon}$.  Since
$\bs{q}_{\epsilon}^{(\eta)} (\bs{z}) = \bs{q}_{\epsilon}(\bs{z}) =0$
if $\bs{z}\notin\mc{B}_{\epsilon}^{\eta} \cup \partial^{\eta}
\mc{B}_{\epsilon}$,
\begin{align*}
& \big\Vert \Theta_{\bs{q}_{\epsilon}^{(\eta)}}
-\Theta_{\bs{q}_{\epsilon}}\big\Vert ^{2} \\
& \quad =\; \frac{\epsilon}{Z_{\epsilon}}\,
\int_{\mc{B}_{\epsilon}^{\eta}\cup\partial^{\eta}\mc{B}_{\epsilon}}
e^{-U(\bs{z})/\epsilon} \, \big[\bs{q}_{\epsilon}^{(\eta)}(\bs{z})-
\bs{q}_{\epsilon}(\bs{z})\big] \cdot \bb{MS}^{-1}\bb{M}^{\dagger}
\big[\bs{q}_{\epsilon}^{(\eta)}(\bs{z})-
\bs{q}_{\epsilon}(\bs{z})\big]\;.
\end{align*}
In particular, since $\bb{M} \bb{S}^{-1} \bb{M}^{\dagger} < C \, \bb{I}$ for
some finite constant,
\begin{equation*}
\big\Vert \Theta_{\bs{q}_{\epsilon}^{(\eta)}}
-\Theta_{\bs{q}_{\epsilon}}\big\Vert ^{2}
\;\le\; \frac{C \epsilon}{Z_{\epsilon}}\,
\int_{\mc{B}_{\epsilon}^{\eta}\cup\partial^{\eta}\mc{B}_{\epsilon}}e^{-U(\bs{z})/\epsilon}
\big |\bs{q}_{\epsilon}^{(\eta)}(\bs{z})-
\bs{q}_{\epsilon}(\bs{z})\big|^{2} \, d\bs{z}\;.
\end{equation*}

In Assertions \ref{as01} and \ref{as02} below, we estimate the last
integral on $\mc{B}_{\epsilon}^{\eta}$ and $\partial^{\eta}
\mc{B}_{\epsilon}$, respectively. Lemma \ref{tp12} follows from these
two assertions.

\begin{asser}
\label{as01}
Assume that $\eta\ll \delta$.  There exist finite constants $C_{1}$,
$C_{2}$, independent of $\epsilon$ and $\eta$, such that
\begin{equation*}
\int_{\mc{B}_{\epsilon}^{\eta}}
e^{-U(\bs{z})/\epsilon} \, \big|\bs{q}_{\epsilon}^{(\eta)}(\bs{z})
-\bs{q}_{\epsilon}(\bs{z})\big|^{2} \, d\bs{z}\;\le\;
\frac{C_2 \, \delta^{2}}{\epsilon^{3}}\, \epsilon^{d/2} \, \eta^{2}\,
e^{-H/\epsilon} \, e^{C_1 \eta\, \delta/\epsilon}\;.
\end{equation*}
\end{asser}

\begin{proof}
We first derive a pointwise estimate of
$|\bs{q}_{\epsilon}^{(\eta)}(\bs{z}) - \bs{q}_{\epsilon}(\bs{z})|$ for
$\bs{z}\in\mc{B}_{\epsilon}^{\eta}$. Recall that $B_\eta(\bs{x})$ the
ball of radius $\eta$ centered at $\bs{x}\in\bb{R}^{d}$. Note
that $B_\eta(\bs{z})\subset\mc{B}_{\epsilon}$ for all
$\bs{z}\in\mc{B}_{\epsilon}^{\eta}$.  Hence, by the Cauchy-Schwarz
inequality and the mean value theorem,
\begin{equation}
\label{se30}
\begin{aligned}
& \big|\bs{q}_{\epsilon}^{(\eta)}(\bs{z})
-\bs{q}_{\epsilon}(\bs{z})\big|^{2} \;=\;
\Big|\int_{\bb{R}^{d}}\phi_{\eta}(\bs{y})\,
\big[\nabla p_{\epsilon}(\bs{z}+\bs{y})
-\nabla p_{\epsilon}(\bs{z})\big]\, d\bs{y}\Big|^{2} \\
&  \;\le\; \int_{\bb{R}^{d}}\phi_{\eta}(\bs{y})
\, \big|\nabla p_{\epsilon}(\bs{z}+\bs{y})
-\nabla p_{\epsilon}(\bs{z})\big|^{2}\, d\bs{y} \\
& \;\le\; \int_{\bb{R}^{d}} \phi_{\eta}(\bs{y})\, |\bs{y}|^{2}\,
\sum_{i,j=1}^d \sup_{\bs{x}\in B_\eta(\bs{z})}
[\nabla^{2}_{x_i,x_j}p_{\epsilon}(\bs{x})]^{2}\, d\bs{y}\;.
\end{aligned}
\end{equation}

By a direct computation,
\begin{equation*}
\nabla^{2}_{x_i,x_j} p_{\epsilon}(\bs{x})
\;=\;-\frac{\alpha}{C_{\epsilon}\epsilon}\,
(\bs{x}\cdot\bs{v})\,
e^{-(\alpha/2\epsilon) \, (\bs{x}\cdot\bs{v})^{2}} \, v_{i}\, v_{j}\;.
\end{equation*}
Since $\eta\ll \delta$ and $\bs z \in \mc{B}_{\epsilon}^{\eta}$, there
exists a finite constant $C_1$ such that $(\bs{x}\cdot\bs{v})^{2} \ge
(\bs{z}\cdot\bs{v})^{2} - C_1\, \eta\, \delta$ for all $\bs{x}\in
B_\eta(\bs{z})$. Hence, as $C_{\epsilon} = \sqrt{2\pi\epsilon/\alpha}$ and
$(\bs{x}\cdot\bs{v})^2 \le C_2 \delta^2$, there exists a finite constant
$C$, independent of $\epsilon$ and $\eta$ such that
\begin{equation*}
[\nabla^{2}_{x_i,x_j} p_{\epsilon}(\bs{x})]^2 \;\le\;
\frac{C_2 \, \delta^{2}}{\epsilon^{3}}\,
e^{-(\alpha/\epsilon)(\bs{z}\cdot\bs{v})^{2}}
e^{C_1\delta\eta/\epsilon}
\end{equation*}
for all $\bs x \in B_\eta(\bs{z})$.

Therefore, in view of \eqref{se30},
\begin{align*}
\big|\bs{q}_{\epsilon}^{(\eta)}(\bs{z})
-\bs{q}_{\epsilon}(\bs{z})\big|^{2}
\; &\le\; \frac{C_2 \, \delta^{2}}{\epsilon^{3}}\,
e^{-(\alpha/\epsilon)(\bs{z}\cdot\bs{v})^{2}}
e^{C_1\delta\eta/\epsilon}
\int_{\bb{R}^{d}}\phi_{\eta}(\bs{y})\, |\bs{y}|^{2}\, d\bs{y}\\
\;&=\; \frac{C_2 \, \delta^{2}}{\epsilon^{3}}\,
e^{-(\alpha/\epsilon)(\bs{z}\cdot\bs{v})^{2}}
e^{C_1\delta\eta/\epsilon} \, \eta^{2}\;.
\end{align*}
In consequence,
\begin{equation*}
\int_{\mc{B}_{\epsilon}^{\eta}} e^{-U(\bs{z})/\epsilon}\,
\big|\bs{q}_{\epsilon}^{(\eta)}(\bs{z})
-\bs{q}_{\epsilon}(\bs{z})\big|^{2}\, d\bs{z}
\;\le\; \frac{C_2 \, \delta^{2}}{\epsilon^{3}}\,
e^{C_1\delta\eta/\epsilon} \, \eta^{2} \,
\int_{\mc{B}_{\epsilon}^{\eta}} e^{-U(\bs{z})/\epsilon}\,
e^{-(\alpha/\epsilon)(\bs{z}\cdot\bs{v})^{2}} \, d\bs{z}\;.
\end{equation*}
By the Taylor expansion, $U(\bs{z})=H+(1/2)\, \bs{z} \cdot \bb{L}
\bs{z} +O(\delta^{3})$ for $\bs{z}\in\mc{B}_{\epsilon}^{\eta}$. Hence,
by Lemma \ref{bl2}, the last integral is bounded by
\begin{equation*}
C_3 \, e^{-H/\epsilon} \int_{\bb{R}^{d}}
\exp\Big\{ -\frac{1}{2\epsilon} \bs{z}\cdot \big[\bb{L} +
2\alpha\bs{v}\bs {v}^{\dagger}\big]\bs{z}\Big\} \, d\bs{z}
\;=\;C_3\,  e^{-H/\epsilon} \, \epsilon^{d/2}\, \sqrt{-\det\bb{L}}
\end{equation*}
for some finite constant $C_3$, independent on $\epsilon$ and $\eta$,
This completes the proof.
\end{proof}

\begin{asser}
\label{as02}
Assume that $\eta\ll \delta$.  There exist finite constants $C_{1}$,
$C_{2}$, independent of $\epsilon$, $\eta$, such that
\begin{equation*}
\int_{\partial^{\eta}\mc{B}_{\epsilon}}
e^{-U(\bs{z})/\epsilon} \, \big|\bs{q}_{\epsilon}^{(\eta)}(\bs{z})
-\bs{q}_{\epsilon}(\bs{z})\big|^{2} \, d\bs{z}\;\le\;
C_{1}\frac{\eta\, \delta^{d-1}}{\epsilon}\,
e^{-H/\epsilon} \Big(1+e^{C_{2}\eta\delta/\epsilon}\Big)\;.
\end{equation*}
\end{asser}

\begin{proof}
We first derive pointwise bounds for $|\bs{q}_{\epsilon}(\bs{z})|$
and $|\bs{q}_{\epsilon}^{(\eta)}(\bs{z})|$. For the former, we have
the trivial bound
\begin{equation}
\label{se41}
|\bs{q}_{\epsilon}(\bs{z})| \;\le\; \frac{C_1}{C_{\epsilon}}\,
e^{- (\alpha/2\epsilon)\, (\bs{z}\cdot\bs{v})^{2}}
\end{equation}
for some finite constant $C_1$.  For the latter, by definition, by the
previous inequality and by the bound on $(\bs{x}\cdot\bs{v})^{2}$ in
terms of $(\bs{z}\cdot\bs{v})^{2}$, obtained in the proof of the
previous assertion,
\begin{align*}
|\bs{q}_{\epsilon}^{(\eta)}(\bs{z})|
\; &\le\;\int_{\bb{R}^{d}} |\bs{q}_{\epsilon}(\bs{z}+\bs{y})|
\, \phi_{\eta}(\bs{y})\, d\bs{y}
\;\le\;\int_{\bb{R}^{d}}  \frac{C_1}{C_{\epsilon}}
\, e^{- (\alpha/2\epsilon)\, ((\bs{z}+\bs{y})\cdot\bs{v})^{2}} \,
\phi_{\eta}(\bs{y}) \, d\bs{y} \\
\; &\le\; \int_{\bb{R}^{d}} \frac{C_1}{C_{\epsilon}} \,
e^{- (\alpha/2\epsilon)\, (\bs{z}\cdot\bs{v})^{2}}
e^{C_2 \eta\delta/\epsilon}
\, \phi_{\eta}(\bs{y})\, d\bs{y}
\;\le\; \frac{C_1}{C_{\epsilon}} \,
e^{- (\alpha/2\epsilon)\, (\bs{z}\cdot\bs{v})^{2}}
e^{C_2 \eta\delta/\epsilon}
\end{align*}
for some finite constant $C_2$.

By the two previous estimates of $|\bs{q}_{\epsilon}(\bs{z})|$ and
$|\bs{q}_{\epsilon}^{(\eta)}(\bs{z})|$, and by the Taylor expansion,
\begin{align*}
e^{-U(\bs{z})/\epsilon} \, \big|\bs{q}_{\epsilon}^{(\eta)}(\bs{z})
-\bs{q}_{\epsilon}(\bs{z})\big|^{2} \; & \le\;
\frac{C_1}{\epsilon}\, e^{-H/\epsilon} \, e^{-(1/2\epsilon)\bs{z}\cdot
[\bb{L}+2\alpha\bs{v}\bs{v}^{\dagger}]\bs{z}} \,
\Big( 1+e^{C_2 \eta\delta/\epsilon} \Big) \\
\; & \le\; \frac{C_1}{\epsilon}\, e^{-H/\epsilon} \,
\Big( 1+e^{C_2 \eta\delta/\epsilon} \Big)\;.
\end{align*}
This pointwise estimate completes the proof of the assertion since the
volume of $\partial^{\eta}\mc{B}_{\epsilon}$ is bounded by
$C_1\eta\, \delta^{d-1}$.
\end{proof}

\noindent{\bf C. Proof of Lemma \ref{tp2}.} Since $\bs{q}_{\epsilon}$
vanishes everywhere, but on the set $\mc{B}_{\epsilon}$,
\begin{equation*}
\big\langle \Theta_{\bs{q}_{\epsilon}}^{*} ,\,
\Psi_{h_{\mc{V}_{1},\mc{V}_{2}}}\big\rangle \;=\;
\int_{\mc{B}_{\epsilon}} \nabla h_{\mc{V}_{1},\mc{V}_{2}}(\bs{z})
\cdot\Theta_{\bs{q}_{\epsilon}}^{*}\, d \bs{z}\;.
\end{equation*}
Note that the inner product of vector fields has been defined only for
smooth vector fields, but it can be extended to weakly
differentiable vector fields.

By the divergence theorem, the last integral is equal to
\begin{equation}
\label{se50}
-\, \frac 1{Z_{\epsilon}}\, \int_{\mc{B}_{\epsilon}}
h_{\mc{V}_{1},\mc{V}_{2}}(\bs{z}) \, e^{-U(\bs{z})/\epsilon}\,
[\mc{L}_{\epsilon}p_{\epsilon}](\bs{z}) \, d\bs{z}
\;+\; \int_{\partial\mc{B}_{\epsilon}}
h_{\mc{V}_{1},\mc{V}_{2}}(\bs{z}) \,
\big[\Theta_{\bs{q}_{\epsilon}}^{*}\cdot\bs{n}(\bs{z})\big]
\, \sigma(d\bs{z})\;,
\end{equation}
where, we recall, $\bs{n}(\bs{z})$ stands for the outer normal vector
to $\mc B_{\epsilon}$ at $\bs z$. The next lemma states that the first
term is negligible. This result holds because
$p_{\epsilon}$ has been defined as an approximation in $\mc
B_{\epsilon}$ of the solution of the equation $\mc{L}_{\epsilon} f =0$
with some boundary conditions.

\begin{lemma}
\label{pt21}
We have that
\begin{equation*}
Z_{\epsilon}^{-1}\int_{\mc{B}_{\epsilon}} e^{-U(\bs{z})/\epsilon}
\, \big|\mc{L}_{\epsilon} p_{\epsilon}(\bs{z})\big| \, d\bs{z}
\;=\;o_{\epsilon}(1)\,T_\epsilon\;.
\end{equation*}
\end{lemma}

\begin{proof}
By definition of $p_{\epsilon}$,
\begin{equation*}
\nabla p_{\epsilon}(\bs{z})\;=\;\frac{1}{C_{\epsilon}}\,
e^{-(\alpha/2\epsilon) (\bs{z}\cdot\bs{v})^{2}}\, \bs{v}\;, \quad
\partial^2_{z_i,z_j} p_{\epsilon} (\bs{z})\;=\;-\,
\frac{\alpha}{\epsilon\, C_{\epsilon}}e^{- (\alpha/2\epsilon) (\bs{z}\cdot\bs{v})^{2}}
\, (\bs{z}\cdot\bs{v})\, v_{i}\, v_{j}\;.
\end{equation*}
By the Taylor expansion, $\nabla
U(\bs{z})=\bb{L}\bs{z}+O(\delta^{2})$. Hence, since $\bs v$ is the
eigenvector of $\bb L \bb M = \bb L^* \bb M$ associated to $-\mu$, the
first order part of $(\mc{L}_{\epsilon} p_{\epsilon}) (\bs{z})$ is
equal to
\begin{equation*}
-\, \frac{1}{C_{\epsilon}}\, \Big\{ \bb L \bs z  \cdot \bb M \bs v +
O(\delta^{2}) \Big\} \, e^{- (\alpha/2\epsilon)
  (\bs{z}\cdot\bs{v})^{2}}
\;=\; \frac{1}{C_{\epsilon}}\, \Big\{ \mu \, (\bs z \cdot \bs v) +
O(\delta^{2}) \Big\} \, e^{- (\alpha/2\epsilon)
  (\bs{z}\cdot\bs{v})^{2}}\;.
\end{equation*}
By \eqref{ins02}, the second order part of $(\mc{L}_{\epsilon}
p_{\epsilon}) (\bs{z})$ is equal to
\begin{equation*}
-\, \frac{\alpha}{C_{\epsilon}}\,  (\bs z \cdot \bs v) \,
(\bs v  \cdot \bb M \bs v) \,
e^{- (\alpha/2\epsilon)  (\bs{z}\cdot\bs{v})^{2}} \;=\;
-\, \frac{\mu}{C_{\epsilon}}\,  (\bs z \cdot \bs v) \,
e^{- (\alpha/2\epsilon)  (\bs{z}\cdot\bs{v})^{2}}\;.
\end{equation*}
Since $C_{\epsilon}=O(\sqrt{\epsilon})$, it follows from the two
previous identities that there exists a finite constant $C_1$ such
that
\begin{equation*}
\big|\mc{L}_{\epsilon} p_{\epsilon}(\bs{z})\big|
\;\le\;\frac{C_1\, \delta^{2}}{\sqrt{\epsilon}}
e^{- (\alpha/2\epsilon) (\bs{z}\cdot\bs{v})^{2}}
\end{equation*}
for all $\bs z\in \mc{B}_{\epsilon}$.

As $U(\bs{z})=H+ (1/2)\, \bs{z}\cdot\bb{L}\bs{z}+O(\delta^{3})$, for
$\bs{z}\in\mc{B}_{\epsilon}$, by the previous estimate,
\begin{equation}
\label{se52}
\int_{\mc{B}_{\epsilon}}e^{- U(\bs{z})/\epsilon}\,
\big|\mc{L}_{\epsilon}p_{\epsilon}(\bs{z})\big|\, d\bs{z}
\;\le\;\frac{C_1 \delta^{2}}{\sqrt{\epsilon}}\, e^{-H/\epsilon}
\, \int_{\mc{B}_{\epsilon}}
e^{-(1/2\epsilon) \bs{z}\cdot
  (\bb{L}+\alpha\bs{v}\bs{v}^{\dagger})\bs{z}} \, d\bs{z} \;.
\end{equation}

It remains to estimate the last integral. Recall Lemma \ref{bl3}, and
denote by $\theta_{1}=0, \theta_{2}, \dots, \theta_{d}>0$ the
eigenvalues of the matrix $\bb{L}+\alpha\bs{v}\bs{v}^{\dagger}$, and by
$\bs{w}_{i}$, $1\le i\le d$, the normal eigenvector corresponding to
$\theta_{i}$. Let $\mathscr{P}_{a}$, $a\in\bb{R}$, be the
$(d-1)$-dimensional space given by
\begin{equation*}
\mathscr{P}_{a}\;=\;a\bs{w}_{1} \;+\;
\left\langle \bs{w}_{2},\,\bs{w}_{3},\,\dots,\,\bs{w}_{d}\right\rangle \;,
\end{equation*}
where $\langle \bs{x}_{1},  \dots, \bs{x}_{n}\rangle$ stands for the
linear space generated by the vectors $\bs{x}_{1},  \dots, \bs{x}_{n}
\in \bb R^d$.

Since $\mc{B}_{\epsilon}\subset\mc{C}_{\epsilon}$, there exists $M>0$
such that
\begin{equation*}
\mc{B}_{\epsilon}\;\subseteq\;\bigcup_{a:|a|\le M\delta}\mathscr{P}_{a}\;.
\end{equation*}

Consider the last integral of \eqref{se52}. Perform the change of
variable $\bs{z}=\sum x_{i}\bs{w}_{i}$, and extend the region of
integration to $\bigcup_{a:|a|\le M\delta}\mathscr{P}_{a}$, to obtain
that
\begin{align*}
\int_{\mc{B}_{\epsilon}} e^{-(1/2\epsilon) \bs{z}\cdot
  (\bb{L}+\alpha\bs{v}\bs{v}^{\dagger})\bs{z}} \, d\bs{z}
\; & \le\; C_1 \, \int_{-M\delta}^{M\delta}dx_{1} \int_{\bb{R}^{d-1}}
\exp\Big\{ -\frac{1}{2\epsilon}\sum_{i=2}^{d}\theta_{i}x_{i}^{2}\Big\}
\, dx_{2}\cdots dx_{d} \\
& =\;C_1 M\delta \, (2\pi\epsilon)^{(d-1)/{2}}
\, \prod_{i=2}^{d} \frac 1{\theta_{i}^{1/2}} \; .
\end{align*}
Therefore, by \eqref{se52},
\begin{equation*}
\int_{\mc{B}_{\epsilon}}e^{- U(\bs{z})/\epsilon}\,
\big|\mc{L}_{\epsilon}p_{\epsilon}(\bs{z})\big|\, d\bs{z}
\;\le\;\frac{C_1 \delta^{3}}{\epsilon}\, e^{-H/\epsilon}
\, \epsilon^{d/{2}} \;.
\end{equation*}
This completes the proof of the lemma since
$\delta^{3}/\epsilon=o_{\epsilon}(1)$.
\end{proof}

We turn to the second integral of \eqref{se50}.

\begin{lemma}
\label{pt22}
We have that
\begin{equation*}
\int_{\partial\mc{B}_{\epsilon}}
h_{\mc{V}_{1},\mc{V}_{2}}(\bs{z}) \,
\big[\Theta_{\bs{q}_{\epsilon}}^{*}\cdot\bs{n}(\bs{z})\big]
\, \sigma(d\bs{z})
\;=\;\left[1+o_{\epsilon}(1)\right]\,T_{\epsilon}\,\omega(\bs{0})\;.
\end{equation*}
\end{lemma}

Decompose the boundary $\partial\mc{B}_{\epsilon}$ in three pieces,
denoted by $\partial_{+}\mc{B}_{\epsilon}$, $\partial_{-}
\mc{B}_{\epsilon}$ and $\partial_{0} \mc{B}_{\epsilon}$, where
$\partial_{+}\mc{B}_{\epsilon} = \partial\mc{B}_{\epsilon}
\cap \partial_{+} \mc{C}_{\epsilon}, \partial_{-}
\mc{B}_{\epsilon}= \partial \mc{B}_{\epsilon}
\cap\partial_{-}\mc{C}_{\epsilon}$, and $\partial_{0}
\mc{B}_{\epsilon} =\partial \mc{B}_{\epsilon} \setminus(\partial_{+}
\mc{B}_{\epsilon} \cup\partial_{-} \mc{B}_{\epsilon})$.

We claim that the contribution to the integral of the piece
corresponding to the boundary $\partial_{0}\mc{B}_{\epsilon}$ is
negligible. Since $|h_{\mc{V}_{1},\mc{V}_{2}}|\le 1$, this claim
follows from the next assertion.

\begin{asser}
\label{bs1}
We have that
\begin{equation*}
\int_{\partial_{0}B_{\epsilon}}
\big| \Theta_{\bs{q}_{\epsilon}}^{*} (\bs{z})\big|
\, \sigma(d\bs{z})\;=\;o_{\epsilon}(1)\,T_{\epsilon}\;.
\end{equation*}
\end{asser}

\begin{proof}
Since $\partial_{0}\mc{B}_{\epsilon}\subset\partial\mc{X}_{\epsilon}$,
\begin{equation}
U(\bs{z})\;=\;H \,+\,
(\lambda_{1}/4)\, \delta^{2}\;\;\mbox{for all}
\;\;\bs{z}\;\in\;\partial_{0}\mc{B}_{\epsilon}\;.
\end{equation}
Therefore, since $\exp\left\{ - (\alpha/2\epsilon)
  (\bs{z}\cdot\bs{v})^{2}\right\} \le 1$ and $C_\epsilon =
\sqrt{2\pi\epsilon/\alpha}$,  by definition of $\delta$,
\begin{equation*}
\big|\Theta^*_{\bs{q}_{\epsilon}}(\bs{z})\big| \;\le\;
\frac {C_1}{Z_{\epsilon}}\, e^{-H/\epsilon}\, \sqrt{\epsilon}
\, e^{-(\lambda_{1}\delta^{2}/4\epsilon)}
\;=\; \frac {C_1}{Z_{\epsilon}}\,  e^{-H/\epsilon}\,
\epsilon^{(1/2)+(\lambda_{1}\, K^{2}/4)}
\end{equation*}
for some finite constant $C_1$. Since the surface volume of
$\partial_{0}\mc{B}_{\epsilon}$ is of order $\delta^{d-1}$, the
statement of the assertion is straightforward consequence from this
uniform bound on $\Theta^*_{\bs{q}_{\epsilon}}$.
\end{proof}

We turn to the boundaries $\partial_{+}\mc{B}_{\epsilon}$ and
$\partial_{-}\mc{B}_{\epsilon}$. To estimate the integral appearing in
the statement of Lemma \ref{pt22} on these sets, bounds on the
equilibrium potential $h_{\mc{V}_{1},\mc{V}_{2}}$ are needed.

\begin{asser}
\label{bl10}
There exist a finite constant $C_0$ and $\epsilon_0>0$ such that for
all $\epsilon<\epsilon_0$,
\begin{equation*}
1-h_{\mc{V}_{1},\mc{V}_{2}}(\bs{y})\;\le\; \frac{C_0}{\epsilon^d}
\exp\Big\{ \frac{U(\bs{y})-H}{2\epsilon}\Big\} \quad
\text{for all $\bs{y}\in\partial_{+}\mc{B}_{\epsilon}$}
\end{equation*}
and
\begin{equation}
\label{se71}
h_{\mc{V}_{1},\mc{V}_{2}}(\bs{y})\;\le\;\frac{C_0}{\epsilon^{d}}
\exp\Big\{ \frac{U(\bs{y})-H}{2\epsilon}\Big\} \quad
\text{for all $\bs{y}\in\partial_{-}\mc{B}_{\epsilon}$}\;.
\end{equation}
\end{asser}

\begin{proof}
Consider first \eqref{se71}. If $\bs{y} \in\partial_{-}
\mc{B}_{\epsilon}$ satisfies $U(\bs{y})\ge H$, then \eqref{se71} is
obvious for all sufficiently small $\epsilon$. Otherwise, $\bs{y}
\in \mc W_2$, and the result follows by Proposition \ref{7-l11}.

The proof of the first claim of the Assertion is analogous. The previous
arguments provide an upper bound for $h_{\mc{V}_{2},\mc{V}_{1}}$
which is the function $1-h_{\mc{V}_{1},\mc{V}_{2}}$.
\end{proof}

We are now in a position to prove Lemma \ref{pt22} at the boundaries
$\partial_{+}\mc{B}_{\epsilon}$ and $\partial_{-}\mc{B}_{\epsilon}$.

\begin{asser}
\label{bs2}
For sufficiently large $K$, we have that
\begin{align*}
& \int_{\partial_{+}\mc{B}_{\epsilon}}
\big[1-h_{\mc{V}_{1},\mc{V}_{2}}(\bs{z})\big]\,
\big[\Theta_{\bs{q}_{\epsilon}}^{*}\cdot\bs{n}(\bs{z})\big]\,
\sigma(d\bs{z})
\;=\;o_{\epsilon}(1)\, T_{\epsilon}\;,
\\
&  \quad \int_{\partial_{-}\mc{B}_{\epsilon}}
h_{\mc{V}_{1},\mc{V}_{2}}(\bs{z}) \,
\big[\Theta_{\bs{q}_{\epsilon}}^{*}
\cdot\bs{n}(\bs{z})\big]\, \sigma(d\bs{z})
\;=\;o_{\epsilon}(1)\, T_{\epsilon}\;.
\end{align*}
\end{asser}

\begin{proof}
We concentrate on the first claim, the proof of the second one being
similar. By Assertion \ref{bl10} and since $C_\epsilon =
O(\sqrt{\epsilon})$,
\begin{align*}
&  \Big|\int_{\partial_{+}\mc{B}_{\epsilon}}
\big[1-h_{\mc{V}_{1},\mc{V}_{2}}(\bs{z})\big] \,
\big[\Theta_{\bs{q}_{\epsilon}}^{*}\cdot\bs{n}(\bs{z})\big]\,
\sigma(d\bs{z})\Big| \\
&  \quad \le\; \frac {C_1 \, \epsilon^{1/2-d}} {Z_{\epsilon}}
\int_{\partial_{+}\mc{B}_{\epsilon}}
\exp\Big\{ \frac{U(\bs{z})-H}{2\epsilon}\Big\} \,
\exp\Big\{ -\frac{U(\bs{z})}{\epsilon}-\frac{\alpha}{2\epsilon}
(\bs{z}\cdot\bs{v})^{2}\Big\} \, \sigma(d\bs{z})
\end{align*}
for some finite constant $C_1$. Note that the exponential terms in the last integral can be written
as
$$-\frac{H}{\epsilon} - \frac{U(\bs{z})-H}{2\epsilon}  - \frac{\alpha}{2\epsilon}
(\bs{z}\cdot\bs{v})^{2}\;.$$
By applying the Taylor expansion of $U$ at the second term, we are able to deduce that the right-hand side of the penultimate displayed equation
is bounded above by
\begin{align*}
& \frac {C_1 \, \epsilon^{1/2-d}} {Z_{\epsilon}} \, e^{- H/\epsilon}
\int_{\partial_{+}\mc{B}_{\epsilon}} \exp\Big\{
-\frac{1}{4\epsilon}\, \bs{z}\cdot
\big[ \bb{L}+2\alpha\bs{v}\bs{v}^{\dagger}\big]\bs{z}\Big\} \, \sigma(d\bs{z})\\
&\quad \le\; \frac {C_1 \, \epsilon^{1/2-d}} {Z_{\epsilon}} \, e^{- H/\epsilon}
\int_{\partial_{+}\mc{B}_{\epsilon}} \exp\Big\{
-\frac{\gamma}{4\epsilon} \, \Vert \bs{z}\Vert^{2}\Big\} \,
\sigma(d\bs{z})\;,
\end{align*}
where $\gamma>0$ is the smallest eigenvalue of the positive-definite
matrix $\bb{L} + 2\alpha\bs{v}\bs{v}^{\dagger}$.  Since $\Vert\bs{z}\|^{2}
\ge \delta^{2}$ for $\bs{z}\in\partial_{+}\mc{B}_{\epsilon}$, the last
integral is less than or equal to $C_1 \, \epsilon^{\gamma K^{2}/4}\,
\delta^{d-1}$, which completes the proof of the assertion in view of
the definition of $T_\epsilon$, provided that $K$ is sufficiently large.
\end{proof}

Next assertion completes the proof of Lemma \ref{pt22}.

\begin{asser}
We have that
\begin{equation*}
\int_{\partial_{+}\mc{B}_{\epsilon}}
\Theta_{\bs{q}_{\epsilon}}^{*}\cdot\bs{n}(\bs{z})
\, \sigma(d\bs{z})\;=\; [1+o_{\epsilon}(1)]\,
T_{\epsilon}\,\omega(\bs{0})
\end{equation*}
\end{asser}

\begin{proof}
Since $\bs{n}(\bs{z})=\bs{e}_{1}$ for
$\bs{z}\in\partial_{+}\mc{B}_{\epsilon},$ by the Taylor expansion, the
lef-hand side of the previous equation can be written as
\begin{equation}
\label{se94}
[1+o_{\epsilon}(1)]\,
\frac {\epsilon} {Z_{\epsilon}} \, \sqrt{\frac{\alpha}{2\pi\epsilon}}
\, e^{- H/\epsilon} \, (\bs{e}_{1}\cdot\bb{M}\bs{v})\,
\int_{\partial_{+}\mc{B}_{\epsilon}}
e^{- (1/2\epsilon) \, \bs{z}\cdot [ \bb{L}+\alpha\bs{v} \bs{v}^{\dagger} ] \bs{z}}
\, \sigma(d\bs{z})\;.
\end{equation}

Let $\theta_{k}=(\delta v_{k}\lambda_{1})/(v_{1}\lambda_{k})$, $2\le k\le d$,
and define the variable $\bs y = (y_{2},\,\dots,\,y_{d})$ by
\begin{equation*}
\bs{z}\;=\; \delta \,\bs{e}_{1} \;+\;
\sum_{k=2}^{d}(y_{k}-\theta_{k})\, \bs{e}_{k}\;,
\end{equation*}
An elementary computation, based on the identity provided by Lemma
\ref{bl1}, yields that
\begin{equation*}
\bs{z}\cdot\left(\bb{L}+\alpha\, \bs{v}\bs {v}^{\dagger}\right)\bs{z}
\;=\;\bs{y}\cdot (\widetilde{\bb{L}} +
\alpha \, {\bs w} {\bs w}^{\dagger})\, \bs{y}\;,
\end{equation*}
where ${\bs{w}}$ is the $(d-1)$-dimensional vector given by $\bs w =
(v_{2},\,\dots,\,v_{d})$ and $\widetilde{\bb{L}}$ is the
$(d-1)\times(d-1)$ diagonal matrix $\textup{diag }
(\lambda_{2},\,\dots,\,\lambda_{d})$.

Perform the change of variables presented in the penultimate displayed
equation to write the last integral in \eqref{se94} as
\begin{equation*}
\int_{D_\epsilon} e^{- (1/2\epsilon) \, \bs{y}\cdot [
  \widetilde{\bb{L}} + \alpha \bs{w} \bs {w}^{\dagger} ] \bs{y} }
\, d\bs{y}\;,
\end{equation*}
where $D_\epsilon\subset \bb R^{d-1}$ is the domain of integration
obtained from $\partial_{+}\mc{B}_{\epsilon}$ by the change of
variables. By Lemma \ref{bl1} and a Taylor expansion $U(\delta,
-  \theta_1, \dots, -  \theta_d) < H + (1/4) \lambda_1
\delta^2$ for all $\epsilon$ small enough. In particular, for
$\epsilon$ small enough, $D_\epsilon$ contains a ball centered at the
origin and of radius $r\, \delta$ for some $r>0$, $D_\epsilon \supset
B(\bs 0, r\, \delta)$. Furthermore, it is easy to verify that
$\widetilde{\bb{L}} \,+\, \alpha\, {\bs w} {\bs w}^{\dagger}$ is
positive definite and hence the last integral is equal to
\begin{equation*}
\left[1+o_{\epsilon}(1)\right] \, (2\pi\epsilon)^{(d-1)/2}
\, \Big\{\det \big( \widetilde{\bb{L}} \,+\,
\alpha\, {\bs w} {\bs w}^{\dagger}\big)\Big\}^{-1/2}\;.
\end{equation*}
Since $\det(\bb{A} + \bs{x} \bs {y}^{\dagger}) =
(1+\bs{y}^{\dagger}\bb{A}^{-1}\bs{x})\det\bb{A}$, by Lemma \ref{bl1},
$\det \big( \widetilde{\bb{L}} \,+\, \alpha\, {\bs w} {\bs
  w}^{\dagger}\big)$ is equal to
\begin{equation*}
(1+\alpha\, {{\bs w}}^{\dagger} \widetilde{\bb{L}}^{-1} \bs{w}) \,
\det{\widetilde{\bb L}}\;=\;\alpha \, \Big( \frac{1}{\alpha}
\,+\, \sum_{k=2}^{d}\frac{v_{k}^{2}}{\lambda_{k}}\Big)
\prod_{i=2}^{d}\lambda_{i}\;=\;
\alpha\, \frac{v_{1}^{2}}{\lambda_{1}}\, \prod_{i=2}^{d}\lambda_{i}\;.
\end{equation*}

On the other hand, since $\bs v$ is the eigenvector of $\bb{L}\,
\bb{M}$ associated to the eigenvalue $-\mu$,
\begin{equation*}
\bs{e}_{1}\cdot\bb{M}\, \bs{v}\;=\;
\bs{e}_{1}\cdot\bb{L}^{-1}\, \bb{L}\, \bb{M}\, \bs{v}
\;=\; -\, \mu\, \bs{e}_{1} \cdot \bb{L}^{-1}\, \bs{v}
\;=\;\frac{\mu}{\lambda_{1}}\, v_{1}\;.
\end{equation*}
To complete the proof of the assertion, it remains to recollect all
estimates, and to recall the definition of $\omega(\bs 0)$, introduced
in \eqref{omega}, and the one of $T_{\epsilon}$, given in \eqref{f03}.
\end{proof}

\section{Proof of Theorem \ref{thmp2}}
\label{sec8}

Recall from \eqref{2-5} that we denote by $H_{\bs x, \bs y}$ the
height of the saddle point between $\bs x$ and $\bs y\in\bb R^d$.

For $U(\bs m_1) < r< U(\bs \sigma_1)$, let $\mc N_r$ be the
neighborhood of $\bs m_1$ given by all points which are connected to
$\bs m_1$ by a continuous path whose height lies below $r$:
\begin{equation*}
\mc N_r \;=\; \big\{\bs x \in\bb R^d : H_{\bs x, \bs m_1} \le
r\big\}\;.
\end{equation*}
An elementary computation shows that there exists a finite constant
$C_0 = C_0(r)$ such that $(\mc L_\epsilon U)(\bs x) \le C_0$ on $\mathscr{N}_r$. Thus, if $H_r$ stands for the hitting time of the boundary of
$\mc N_r$, which is finite because, by condition (P4), the process is
positive recurrent,
\begin{equation*}
\bb E_{\bs x}[U(X^\epsilon_{t\wedge H_r})] \;-\; U(\bs x) \le\; C_0 \,
\bb E_{\bs x}[H_r]
\end{equation*}
for all $t\ge 0$. Letting $t\to\infty$, we obtain that for all
$\epsilon >0$, $U(\bs m_1) < s< r< U(\bs \sigma_1)$, $\bs x \in \mc
N_s$,
\begin{equation}
\label{7-7}
\bb E_{\bs x}[H_r] \;\ge\; (r\;-\; s)/C_0\;.
\end{equation}

Since $\bs m_1$ is a non-degenerate critical point of $U$, there
exists a finite constant $C_1$ such that $\Vert \nabla U (\bs x) \Vert
\le C_1 \Vert \bs x -\bs m_1\Vert$ for all $\bs x\in B_1(\bs m_1)$. In
particular, on $B_{r\sqrt{\epsilon}}(\bs m_1)$, $r>0$, $\Vert \nabla U
\Vert \le C_1 r \sqrt{\epsilon}$. Hence, by \eqref{7-1}, $\nu_{B_{r
    \sqrt{\epsilon}}(\bs m_1)} \le C_1 r^2 \epsilon^{-1}$, so that
\begin{equation}
\label{7-2}
\nu_{B_{r \sqrt{\epsilon}}(\bs m_1)} \, \epsilon \;\le\; C_0(r)\;.
\end{equation}

\begin{lemma}
\label{7-l3}
Let $w(\bs{x})=\mathbb{E}_{\bs{x}}[H_{\mc V_2}]$.  There exists a
finite constant $C_0$ such that
\begin{equation*}
\sup_{\bs x \in B_{\sqrt{\epsilon}}(\bs m_1)} w(\bs x)  \;\le\; C_0
\inf_{\bs x \in B_{\sqrt{\epsilon}}(\bs m_1)} w(\bs x) \;.
\end{equation*}
\end{lemma}

\begin{proof}
Fix $\bs x$, $\bs x' \in B_{\sqrt{\epsilon}}(\bs m_1)$.
Let $G_{\mc V^c_2}$ be the Green function associated to the diffusion
killed at $\mc V_2$. Since $w$ solves \eqref{22} with $\mf g=1$, $\mf
b=0$, by Lemma \ref{7-l14},
\begin{equation*}
w(\bs{x})\;=\;\int_{\mc V_2^c} G_{\mc V^c_2}(\bs{x},\,\bs{y}) \,
d\bs{y}\;.
\end{equation*}
Fix $\bs y\in \mc V_2^c \setminus B_{2\sqrt{\epsilon}}(\bs m_1)$. The
function $G_{\mc V^c_2}(\cdot,\,\bs{y})$ is non-negative and harmonic in
$B_{2\sqrt{\epsilon}(\bs m_1)}$. Hence, by Lemma \ref{7-l4} and
\eqref{7-2}, $G_{\mc V^c_2}(\bs x, \bs{y}) \le C_0 G_{\mc V^c_2}(\bs x',
\bs{y})$. The right-hand side of the previous formula is thus bounded
above by
\begin{equation*}
C_0 \int_{\mc V_2^c \setminus B_{2\sqrt{\epsilon}}(\bs m_1)} G_{\mc V^c_2}(\bs{x}',\,\bs{y}) \,
d\bs{y} \;+\; \int_{B_{2\sqrt{\epsilon}}(\bs m_1)} G_{\mc V^c_2}(\bs{x},\,\bs{y}) \,
d\bs{y} \;.
\end{equation*}
The first term is bounded by $C_0 w(\bs{x}')$, while the second one,
in view of Lemma \ref{7-l5}, is less than or equal to $C_0
a_d(\epsilon)$, where $a_d(\epsilon) = \epsilon$, $d\ge 3$, and
$a_2(\epsilon) = \epsilon \log \epsilon^{-1}$. Fix $U(\bs m_1) < r <
U(\bs \sigma_1)$. Since $H_r \le H_{\mc V_2}$, where $H_r$ has been
introduced above \eqref{7-7}, $w(\bs{x}') \ge [U(\bs \sigma_1) - U(\bs
m_1)]/C_0$. We may therefore bound $a_d(\epsilon)$ by
$w(\bs{x}')$ to complete the proof of the lemma.
\end{proof}

\begin{lemma}
\label{7-l1}
We have that
\begin{equation*}
\bb E_{\bs m_1} [H_{\mc V_2}] \;=\; \big( 1 + o_\epsilon(1)
\big) \, \frac 1{\Cap (B_\epsilon(\bs m_1), \mc V_2)}
\int_{\bb R^d} h^*_{B_\epsilon(\bs m_1), \mc V_2} (\bs y)
\, \mu_{\epsilon}(d\bs y) \;.
\end{equation*}
\end{lemma}

\begin{proof}
Recall that $w(\bs{x})=\mathbb{E}_{\bs{x}}[H_{\mc V_2}]$.  Since $w$
solves $\mc L_\epsilon w =-1$ on $B_{\sqrt{\epsilon}}(\bs m_1)$, and
since, by \eqref{7-2}, $\nu_{B_{\sqrt{\epsilon}}(\bs m_1)}
(\sqrt{\epsilon})^2 \le C_0$, by Lemma \ref{7-l2} with $R=\epsilon$,
\begin{align*}
\sup_{\bs x \in B_{\epsilon}(\bs m_1)} w(\bs x)
\; & \le\; w(\bs m_1)  \;+\; C_0 \, \epsilon^{\alpha/2}
\, \Big( \text{\rm osc } (w,B_{\sqrt{\epsilon}}(\bs x))
+ R_0 \Big) \\
& \le \; w(\bs m_1)  \;+\; C_0 \, \epsilon^{\alpha/2}
\, \Big( \sup_{\bs x' \in B_{\sqrt{\epsilon}}(\bs m_1)} w (\bs x')
+  \sqrt{\epsilon} \Big)\;,
\end{align*}
where we used the fact that $w$ is non-negative in the last inequality
and we replaced $R_0$ by $\sqrt{\epsilon}$. Recall from the proof of the
previous lemma that $w(\bs m_1) > c_0 >0$ for some positive constant
$c_0$ independent of $\epsilon$. Hence, by Lemma \ref{7-l3}, the
previous expression is bounded by
\begin{equation*}
w(\bs m_1)  \;+\; C_0 \, \epsilon^{\alpha/2}
\, \Big( w (\bs m_1) + \sqrt{\epsilon} \Big)
\;\le\; w(\bs m_1)  \;+\; C_0 \, \epsilon^{\alpha/2}
\, w (\bs m_1) \;,
\end{equation*}
so that
\begin{equation*}
\sup_{\bs x \in B_{\epsilon}(\bs m_1)} w(\bs x)
\;\le \; \big( 1 + o_\epsilon(1) \big) \, w (\bs m_1) \;.
\end{equation*}
A lower bound for $\inf_{\bs x \in B_{\epsilon}(\bs m_1)} w(\bs x)$ is
derived analogously.

Recall from \eqref{2-2} the definition of the equilirbium measure $\nu_{
  B_{\epsilon}(\bs m_1), \mc V_2}$.  Since it is concentrated on
$\partial B_{\epsilon}(\bs m_1)$, it follows from the previous
estimates that
\begin{equation*}
w (\bs m_1) \;=\; \big( 1 + o_\epsilon(1) \big)
\int_{\partial B_\epsilon(\bs m_1)} w (\bs{y}) \,
\nu_{ B_{\epsilon}(\bs m_1), \mc V_2} \, (d\bs{y}) \;,
\end{equation*}
To complete the proof of the lemma, it remains to recall identity
\eqref{2-1}.
\end{proof}

\begin{lemma}
\label{7-l12}
We have that
\begin{equation}
\label{7-8a}
\int_{\bb R^d} h^*_{B_\epsilon(\bs m_1), \mc V_2} (\bs y) \,
e^{-U(\bs y)/\epsilon}\, d\bs y \;=\; (1+o_\epsilon(1)\big) \,
\frac{(2\pi \epsilon)^{d/2}  \, e^{-U(\bs m_1)/\epsilon}}
{\sqrt{\det [(\text{\rm Hess }U)\, (\bs{m}_{1})]}}
\;.
\end{equation}
\end{lemma}

\begin{proof}
We estimate separately the integral on different parts. Recall from
\eqref{ins01} the definition of $\delta$. We claim that
\begin{equation}
\label{7-8}
\int_{B_\delta(\bs m_1)} h^*_{B_\epsilon(\bs m_1), \mc V_2} (\bs y) \,
e^{-U(\bs y)/\epsilon}\, d\bs y \;=\; (1+o_\epsilon(1)\big) \,
\frac{(2\pi \epsilon)^{d/2}  \, e^{-U(\bs m_1)/\epsilon}}
{\sqrt{\det [(\text{\rm Hess }U)\, (\bs{m}_{1})]}}
\;.
\end{equation}
Indeed, by Proposition \ref{7-l11}, on $B_\delta(\bs m_1)$,
$h^*_{B_\epsilon(\bs m_1), \mc V_2} (\bs y) = 1 -  h^*_{\mc V_2,
  B_\epsilon(\bs m_1)} (\bs y) = 1 + o_\epsilon(1)$. A Taylor
expansion of $U$ around $\bs m_1$ together with Gaussian estimates
permits to conclude.

Let $\kappa_1$ the smallest eigenvalue of $(\text{\rm Hess }U)\,
(\bs{m}_{1})$. There exists $r_0>0$ such that $U(\bs x) - U(\bs m_1)
\ge (1/4) \kappa_1 \Vert \bs x\Vert^2$ for all $\bs x\in B_{r_0}(\bs
m_1)$. We claim that
\begin{equation}
\label{7-9}
\int_{ B_{r_0}(\bs m_1) \setminus B_\delta(\bs m_1)}
h^*_{B_\epsilon(\bs m_1), \mc V_2} \, (\bs y) \,
e^{-U(\bs y)/\epsilon}\, d\bs y \;=\; o_\epsilon(1) \,
\epsilon^{d/2}  \, e^{-U(\bs m_1)/\epsilon}\;.
\end{equation}
By the bound on $U$ and since the harmonic function is bounded by $1$,
the integral is less than or equal to
\begin{equation*}
e^{-U(\bs m_1)/\epsilon} \int_{ B_{r_0}(\bs m_1) \setminus B_\delta(\bs m_1)}
e^{-(1/4) (\kappa_1/\epsilon) \Vert \bs y \Vert^2 }\, d\bs y \;.
\end{equation*}
A change of variables and an elementary computation yields that the
integral is equal to $o_\epsilon(1) \, \epsilon^{d/2}$, which
completes the proof of \eqref{7-9}.

Let $a_0 = \inf_{\bs x\in \partial B_{r_0}(\bs m_1)} U(\bs x) > U(\bs
m_1)$. Assume that $a_0 < U(\bs\sigma_1)$. If this is not the case
replace $a_0$ by $a'_0$ where $U(\bs m_1) < a'_0 < U(\bs
\sigma_1)$. Let $\mc S(a_0)= \{\bs x\in\bb R^d : U(\bs x) \ge a_0\}$. It
follows from this bound and \eqref{tight2} that
\begin{equation}
\label{7-10}
\int_{\mc S(a_0)}
h^*_{B_\epsilon(\bs m_1), \mc V_2} \, (\bs y) \,
e^{-U(\bs y)/\epsilon}\, d\bs y \;=\;
e^{-U(\bs m_1)/\epsilon} e^{-a/\epsilon}
\end{equation}
for some $a>0$, which is exponentially smaller than the right-hand side of \eqref{7-8a}

It remains to estimate the integral over the set $\mc A = \{\bs
x\in\bb R^d : U(\bs x) \le a_0 , U(\bs x) = U(\bs z(\bs x , \bs m_2))\}$.
By property (P1), the set $\mc A$ is bounded.  On this set, by
Proposition \ref{7-6}, $h^*_{B_\epsilon(\bs m_1), \mc V_2} \, (\bs x)
\le C_0 \epsilon^{-d}\, \exp \big\{- \epsilon^{-1}\, \big[ U(\bs
\sigma_1) - U(\bs x) \big] \big\}$. Therefore,
\begin{equation*}
\int_{\mc A} h^*_{B_\epsilon(\bs m_1), \mc V_2} \, (\bs y) \,
e^{-U(\bs y)/\epsilon}\, d\bs y \;\le \;
C_0  \, \epsilon^{-d}
\int_{\mc A} e^{-U(\bs \sigma_1)/\epsilon}\, d\bs y
\;\le\; C_0\epsilon^{-d}\, e^{-U(\bs m_1)/\epsilon} e^{-a/\epsilon}
\end{equation*}
for some $a>0$. This completes the proof of the lemma.
\end{proof}

\begin{proof}[Proof of Theorem \ref{thmp2}]
It is enough to put together the estimates of Lemmata \ref{7-l1},
\ref{7-l12} with the estimate of the capacity, stated in Theorem
\ref{thmp1} with $\mc V_1 = B_\epsilon(\bs m_1)$.
\end{proof}

\smallskip\noindent{\bf Grants:} C. Landim has been partially
supported by FAPERJ CNE E-26/201.207/2014, by CNPq Bolsa de
Produtividade em Pesquisa PQ 303538/2014-7, and by ANR-15-CE40-0020-01
LSD of the French National Research Agency.  M.\ Mariani's visit to IMPA was supported by the grant FAPERJ CNE E-26/102.338/2013, and he also acknowledges Russian Academic Excellence Project ’5-100'. I. Seo was supported by the National Research Foundation of Korea(NRF) grant funded by the Korea government(MSIT) (No. 2018R1C1B6006896) and by POSCO Science Fellowship from the POSCO TJ Park Foundation.

\smallskip\noindent{\bf Conflict of Interest:} The authors declare
that they have no conflict of interest.

\end{document}